\documentclass{scrartcl}
\usepackage{lmodern}
\usepackage[T1]{fontenc}
\usepackage[utf8]{inputenc}
\usepackage{amsmath}
\usepackage{amsthm}
\usepackage{amssymb,amsfonts,dsfont}
\usepackage{graphicx}
\usepackage[authoryear]{natbib}
\usepackage[unicode=true, bookmarks=false, breaklinks=true, pagebackref=false,
pdfborder={0 0 0},pdfborderstyle={},colorlinks=false]
{hyperref}

\makeatletter
\newtheoremstyle{exampstyle}
  {\topsep} 
  {\topsep} 
  {\itshape} 
  {} 
  {\bfseries} 
  {.} 
  {.5em} 
  {} 

\theoremstyle{exampstyle}

\newtheorem{thm}{\protect\theoremname}
  \theoremstyle{exampstyle}
  
  \newtheorem{lem}{Lemma}
\newtheorem{prop}{Proposition}

\newtheorem{rem}{Remark}

\@ifundefined{date}{}{\date{}}
\usepackage{ifxetex}
\usepackage{ifluatex}
\ifnum 0\ifxetex 1\fi\ifluatex 1\fi=0 
  \else 
  \ifxetex
    \usepackage{mathspec}
\else
    \fi
  \defaultfontfeatures{Ligatures=TeX,Scale=MatchLowercase}
\fi
\IfFileExists{upquote.sty}{\usepackage{upquote}}{}
\IfFileExists{microtype.sty}{%
\usepackage{microtype}
\UseMicrotypeSet[protrusion]{basicmath} 
}{}

\IfFileExists{parskip.sty}{%
\usepackage{parskip}
}{
\setlength{\parindent}{0pt}
\setlength{\parskip}{6pt plus 2pt minus 1pt}
}
\newcommand\numberthis{\addtocounter{equation}{1}\tag{\theequation}}
\newcommand{\sgn}{\text{sgn}}

\newcommand{\dx}{\mathrm{d}x}
\newcommand{\dmu}{\mathrm{d}\mu}
\newcommand{\dy}{\mathrm{d}y}
\newcommand{\bigO}{\mathcal{O}}
\newcommand{\smallo}{o}
\newcommand{\smalloP}{o_{\mathbb{P}}}
\newcommand{\bigOP}{\bigO_{\mathbb{P}}}
\renewcommand{\ln}{\ell_{n}}
\newcommand{\lIS}{\ell_{n,m}^{\mathrm{IS}}}
\newcommand{\lNCE}{\ell_{n,m}^{\mathrm{NCE}}}
\newcommand{\R}{\mathbb{R}}

\newcommand{\Z}{\mathcal{Z}}
\newcommand{\Zt}{\mathcal{Z}(\theta)}
\newcommand{\Zp}{\mathcal{Z}(\psi)}
\newcommand{\mle}{\hat{\theta}_{n}}
\newcommand{\mcmle}{\hat{\theta}_{n,m}^{\mathrm{IS}}}
\newcommand{\nce}{\hat{\theta}_{n,m}^{\mathrm{NCE}}}
\newcommand{\emle}{\hat{\xi}_{n}}
\newcommand{\emcmle}{\hat{\xi}_{n,m}^{\mathrm{IS}}}
\newcommand{\ence}{\hat{\xi}_{n,m}^{\mathrm{NCE}}}

\newcommand{\setX}{\mathcal{X}}
\newcommand{\as}{\mbox{\mbox{a.s.}}}
\newcommand{\cvm}{\text{\ensuremath{\underset{m\rightarrow\infty}{\rightarrow}}}}
\newcommand{\cvn}{\text{\ensuremath{\underset{n\rightarrow\infty}{\rightarrow}}}}
\newcommand{\cvd}{\stackrel{\mathcal{D}}{\rightarrow}}
\newcommand{\cvdm}{\mathrel{\mathop{\rightarrow}_{m\rightarrow\infty}^{\mathcal{D}}}}

\newcommand{\refdist}{\mathbb{P}_{\psi}}
\newcommand{\Pt}{\mathbb{P}_{\theta}}

\newcommand{\trueth}{\theta^{\star}}
\newcommand{\truee}{\xi^{\star}}

\newcommand{\glNCE}{\Psi^{\mathrm{NCE}}}
\newcommand{\glIS}{\Psi^{\mathrm{IS}}}
\newcommand{\glISk}{\Psi_{k}^{\mathrm{IS}}}

\title{Noise contrastive estimation: asymptotic properties, formal comparison
with MCMC-MLE}
\author{Lionel Riou-Durand and Nicolas Chopin (ENSAE-CREST)}

\makeatother

\providecommand{\corollaryname}{Corollary}
\providecommand{\theoremname}{Theorem}

\begin{document}

\title{Noise contrastive estimation: asymptotics, comparison 
with MC-MLE}

\author{Lionel Riou-Durand and Nicolas Chopin (ENSAE-CREST)}
\maketitle
\begin{abstract}
A statistical model is said to be un-normalised when its
likelihood function involves an intractable normalising constant.
Two popular methods for parameter inference for these
models are MC-MLE (Monte Carlo maximum likelihood estimation),
and NCE (noise contrastive estimation); both methods rely on simulating
artificial data-points to approximate the normalising constant. While the 
asymptotics of MC-MLE have been established
under general hypotheses \citep{geyer_1994}, this is not so for NCE.
We establish consistency and asymptotic normality of NCE estimators under mild
assumptions. We compare NCE and MC-MLE under several asymptotic
regimes. In particular, we show that, when $m\rightarrow\infty$ while $n$ is
fixed ($m$ and $n$ being respectively the number of artificial data-points, and
actual data-points), the two estimators are asymptotically equivalent.
Conversely, we prove that, when the artificial data-points are IID, and when
$n\rightarrow\infty$ while $m/n$ converges to a positive constant, the
asymptotic variance of a NCE estimator is always smaller than the asymptotic
variance of the corresponding MC-MLE estimator. We illustrate the variance
reduction brought by NCE through a numerical study. 
\end{abstract}

\section{Introduction}

Consider a 
set of probability densities $\{f_{\theta}:\theta\in\Theta\}$ with
respect to some measure $\mu$, defined on a space $\mathcal{X}$, such that:
\[
f_{\theta}(x)=\frac{h_{\theta}(x)}{\Zt}
\]
where $h_\theta$ is non-negative, and $\Zt$ is a normalising constant,
$\Zt=\int_{\mathcal{X}} h_{\theta}(x)\mu(\dx)$. A model based on such 
a family of densities is said to be 
un-normalised if function $h_\theta$ may be computed point-wise, but 
$\Zt$ is not available (i.e. it may not be computed in a reasonable CPU time). 

Un-normalised models arise in several areas of machine learning and
Statistics, such as deep learning \citep{deep_learning},
computer vision \citep{computer_vision}, image
segmentation \citep{image_segmentation}, social network modelling \citep{caimo_friel},
directional data modelling \citep{directional_data_modelling}, among others. In most
applications, data-points are assumed to be IID (independent and identically
distributed); see however e.g. \citet{mnih_teh} or \citet{poisson_transform}
for applications of non-IID un-normalised models. In that spirit, 
we consider an un-normalised model of IID variables $Y_1,\ldots,Y_n$, with 
log-likelihood (divided by $n$): 
\begin{equation}
\ell_{n}(\theta)=\frac{1}{n}\sum_{i=1}^{n}\log h_{\theta}(y_{i})-\log\Zt.\label{eq:loglik_original}
\end{equation}
The fact that $\Zt$ is intractable precludes standard maximum likelihood
estimation. 

\citet{geyer_1994} wrote a seminal paper on un-normalised models,
in which he proposed to estimate $\theta$ by maximising function
\begin{equation}
\lIS(\theta)=\frac{1}{n}\sum_{i=1}^{n}\log\frac{h_{\theta}(y_{i})}{h_{\psi}(y_{i})}-\log\left\{ \frac{1}{m}\sum_{j=1}^{m}\frac{h_{\theta}(x_{j})}{h_{\psi}(x_{j})}\right\} \label{eq:lIS_original}
\end{equation}
where the $x_{j}$'s are $m$ artificial data-points generated from a user-chosen 
distribution $\refdist$ with density $f_{\psi}(x)=h_{\psi}(x)/\Zp$.
The empirical average inside the second log is a consistent (as $m\rightarrow\infty$)
importance sampling estimate of $\Zt/\Zp$. Function $\lIS$ is thus an
approximation of the log-likelihood ratio $\ell_{n}(\theta)-\ell_{n}(\psi)$, whose maximiser is the MLE.

In many applications, the easiest way to sample from $\refdist$ is to use 
MCMC (Markov chain Monte Carlo). \citet{geyer_1994} established
the asymptotic properties of the MC-MLE estimates under general
conditions; in particular that the $x_{j}$'s are realisations
of an ergodic process. This is remarkable, given that most of the theory
on M-estimation (i.e. estimation obtained by maximising functions)
is restricted to IID data.

More recently, \citet{nce_2012} proposed an alternative approach
to parameter estimation of un-normalised models, called noise contrastive
estimation (NCE). It also relies on simulating artificial data-points
$x_{1},\ldots,x_{m}$ from distribution $\refdist$. 
The method consists in maximising the
likelihood of a logistic classifier, where actual
(resp. artificial) data-points are assigned label 1 (resp. 0). With
symbols: 
\begin{equation}
\lNCE(\theta,\nu)=\sum_{i=1}^{n}\log q_{\theta,\nu}(y_{i})+\sum_{i=1}^{m}\log\left\{ 1-q_{\theta,\nu}(x_{i})\right\} \label{eq:lNCE}
\end{equation}
where $q_{\theta,\nu}(x)$, the probability of label $1$ for a value
$x$, is defined through odd-ratio function: 
\[
\log\left\{ \frac{q_{\theta,\nu}(x)}{1-q_{\theta,\nu}(x)}\right\} =\log\left\{ \frac{h_{\theta}(x)}{h_{\psi}(x)}\right\} +\nu+\log\left(\frac{n}{m}\right).
\]
The NCE estimator of $\theta$ is obtained by maximising function
$\lNCE(\theta,\nu)$ with respect to both $\theta\in\Theta$ and $\nu\in\R$.
In particular, when the considered model is exponential,
i.e. when $h_{\theta}(x)=\exp\big\{ \theta^{T}S(x)\big\} $, for some statistic
$S$, $\lNCE$ is the log-likelihood of a standard
logistic regression, with covariate $S(x)$. In that case, implementing
NCE is particularly straightforward. 

This paper has two objectives: first, to establish the asymptotic
properties of NCE when the artificial data-points are generated from an ergodic process (typically a MCMC sampler) in order to show that NCE is as widely applicable as MC-MLE; second, to compare the statistical efficiency of both
methods. 

As a preliminary step, we replace the original log-likelihood by its Poisson
transform \citep{poisson_transform}: 
\begin{equation}
\ell_{n}(\theta,\nu)=\frac{1}{n}\sum_{i=1}^{n}\log\left\{
\frac{h_{\theta}(y_{i})}{h_{\psi}(y_{i})}\right\}
+\nu-e^{\nu}\times\frac{\Zt}{\Zp} . \label{eq:loglik_extended}
\end{equation}
This function is the log-likelihood (up to a linear transformation) of the Poisson process with 
intensity $h_\theta(y) + \nu$. It produces exactly the same MLE as the original
likelihood: 
$(\widehat{\theta}_{n},\widehat{\nu}_{n})$
maximises \eqref{eq:loglik_extended} if and only if $\mle$ maximises
\eqref{eq:loglik_original} and $\widehat{\nu}_{n}=\log\left\{\Zp/\Z(\mle)\right\}$.

In the same way, we replace the MC-MLE log-likelihood by function
\begin{equation}
\lIS(\theta,\nu)=\frac{1}{n}\sum_{i=1}^{n}\log\left\{
\frac{h_{\theta}(y_{i})}{h_{\psi}(y_{i})}\right\}
+\nu-\frac{e^{\nu}}{m}\sum_{j=1}^{m}\frac{h_{\theta}(x_{j})}{h_{\psi}(x_{j})}
\label{eq:lIS_extended}
\end{equation}
which has the same maximiser
(with respect to $\theta$) as function \eqref{eq:lIS_original}.

We thus obtain three objective functions defined with respect to the
same parameter space, $\Theta\times\R$. This will greatly facilitate
our analysis. 
%
The paper is organised as follows. In Section \ref{sec:Set-up-and-notations},
we introduce the set up and notations. In Section \ref{sec:n_fixed}, we study
the behaviour of the NCE estimator as $m\rightarrow\infty$ (while $n$ is kept
fixed). We prove that the NCE estimator converges to the MLE at the same $m^{-1/2}$ 
rate as the MC-MLE estimator, and the difference between the two estimators
converges faster, at rate $m^{-1}$. 
In Section \ref{sec:both_n_and_m}, we let both $m$ and $n$ go to infinity while  
$m/n \rightarrow \tau>0$.  We obtain asymptotic variances
for both estimators, which admit a simple and interpretable decomposition.
Using this decomposition, we are able to establish that when the artificial
data-points are IID, the asymptotic variance of NCE is always smaller than the
asymptotic variance of MC-MLE (for the same computational budget). Section
\ref{sec:Numerical-study} assesses this variance reduction in a numerical
example. Section \ref{sec:Conclusion} discusses the practical implications of
our results. All proofs are delegated to the appendix.

\section{Set-up and notations\label{sec:Set-up-and-notations}}

Let $\Theta$ be an open subset of $\mathbb{R}^d$. We consider a parametric
statistical model $\{\Pt^{\otimes n} : \theta\in\Theta\}$,
corresponding to $n$ IID data-points lying in space $\setX\subset\R^{k}$. We
assume that the model is identifiable, and
equipped with some dominating measure $\mu$, inducing the log-likelihood
\eqref{eq:loglik_original}. From now on,
we work directly with the ``extended'' version of approximate and
exact log-likelihoods, i.e. functions \eqref{eq:lNCE}, \eqref{eq:loglik_extended}
and \eqref{eq:lIS_extended}, which are functions of extended parameter
$\xi=(\theta,\nu)$, with $\xi\in\Xi=\Theta\times\R$. When convenient,
we also write $\ln(\xi)$ for $\ln(\theta,\nu)$ and so on. 
An open ball in $\Xi$, centered on $\xi$ and of radius $\epsilon$, 
is denoted $B(\xi,\epsilon)$. We may also use this notation for balls 
in $\Theta$. 

The point of this paper is to study and compare point estimates $\emcmle$ and
$\ence$, which maximise functions \eqref{eq:lIS_extended} and \eqref{eq:lNCE}.
For the sake of generality, we allow these estimators to be approximate
maximisers; i.e. we will refer to $\emcmle$ as an approximate MC-MLE if
\begin{equation}
\lIS(\emcmle)\geq\sup_{\xi\in\Xi}\lIS(\xi)-\smallo(1)\quad\as\label{eq:def_approx_max}
\end{equation}
and with a similar definition for $\ence$. The meaning of symbol $\smallo(1)$ in \eqref{eq:def_approx_max}
depends on the asymptotic regime: in Section \ref{sec:n_fixed}, 
$n$ is kept fixed, while $m\rightarrow \infty$, hence 
$\smallo(1)$ means ``\textit{converges to zero as $m\rightarrow\infty$}''.
In Section \ref{sec:both_n_and_m}, both $m$ and $n$ go to infinity, 
and the meaning of $\smallo(1)$ must be adapted accordingly.

In both asymptotic regimes, the main assumption regarding the sampling process is
as follows. 
\begin{description}
\item [{(X1)}] The artificial data-points are realisations of a $\refdist-$ergodic process
$(X_{j})_{j\ge1}$. 
\end{description}
By $\refdist-$ergodicity, we mean that the following law of large number
holds: 
\[
\frac{1}{m}\sum_{j=1}^{m}\varphi(X_{j})\cvm\mathbb{E}_{\psi}\left[\varphi(X)\right]=\int_{\setX}\varphi(x)f_{\psi}(x)\mu(\dx)
\]
for any measurable, real-valued function $\varphi$ such that
$\mathbb{E}_{\psi}\left[|\varphi(X)|\right]<+\infty$. 

Assumption (X1) is mild. For instance, if the $X_{j}$'s are
generated by a MCMC algorithm, this is
equivalent to assuming that the simulated chain is aperiodic and irreducible,
which is true for all practical MCMC samplers; see e.g. \citet{gareth_rosenthal}.

Finally, note that, 
although notation $\refdist$ suggests that the distribution of the
artificial data-points belongs to the considered parametric model, this is not
compulsory. The only required assumption is that the model is dominated by
$\refdist$ (i.e. $\Pt\ll\refdist$ for every $\theta\in\Theta$).

\section{Asymptotics of the Monte Carlo error\label{sec:n_fixed}}

In this section, the analysis is conditional on the observed data: $n$ and
$y_1,...,y_n$ are fixed. 
The only source of randomness is the Monte Carlo
error, and the quantity we seek to estimate is the (intractable) MLE.  
This regime was first studied for MC-MLE by \citet{geyer_1994}.  For
convenience, we suppose that the MLE exists and is unique; or equivalently that
$\emle=(\mle,\widehat{\nu}_n)$ is the unique maximiser of $\ln$.

\subsection{Consistency\label{subsec:Consistency}}

We are able to prove NCE consistency (towards the MLE) using the same approach
as \citet{geyer_1994} for MC-MLE. Our consistency result relies on the
following assumptions:
\begin{description}
\item [{(C1)}] The random sequence $\big(\ence\big)_{m\geq1}$ is an approximate NCE estimator, which belongs to a compact set almost surely.
\item [{(H1)}] The maps $\theta\mapsto h_{\theta}(x)$ are: 
\begin{enumerate}
\item lower semi-continuous at each $\theta\in\Theta$, except for $x$
in a $\mathbb{P}_{\psi}$-null set that may depend on $\theta$;
\item upper semi-continuous, for any $x$ not in a $\mathbb{P}_{\psi}$-null
set (that does not depend on $\theta$), and for all $x=y_{i}$, $i=1,\ldots,n$.
\end{enumerate}
\end{description}

\begin{thm}
\label{thm:nce_consistency}Under assumptions (X1), (C1) and (H1), almost surely: $\ence\cvm\emle$.
\end{thm}
This result is strongly linked to Theorems 1 and 4 of \citet{geyer_1994},
which state that $\mcmle\rightarrow\mle$ as $m\rightarrow\infty$ under essentially
the same assumptions. These assumptions are very mild: they basically require
continuity of the maps $\theta\mapsto h_{\theta}(x)$, without any integrability
condition. It is noteworthy that Theorem 1 does not require $\Theta$ to be a
subset of $\mathbb{R}^d$: it holds as soon as $\Theta$ is
a separable metric space.

\subsection{Asymptotic normality, comparison with MCMC-MLE\label{subsec:Rate-of-convergence}}

In order to compare the Monte Carlo error of MC-MLE and NCE estimators, we make the following extra assumptions: 
\begin{description}
\item [{(H2)}] The maps $\theta\mapsto h_{\theta}(x)$ are twice continuously
differentiable in a neighborhood of $\mle$ for $\refdist-$almost
every $x$, and for $x=y_{i}$, $i=1,\ldots,n$. The Hessian matrix
 $\mathbf{H}=\nabla^{2}\ell_{n}(\mle)$
is invertible. Moreover, for some $\varepsilon>0$
\[
\int_{\setX}a_{\varepsilon}(x) \underset{\theta\in B(\mle,\varepsilon)}{\text{sup}}h_{\theta}(x) \mu(\dx)<+\infty
\]
where $a_{\varepsilon}(x)=1+\underset{\theta\in B(\mle,\varepsilon)}{\text{sup}}\|\nabla_{\theta}\log h_{\theta}(x)\|^{2}+\underset{\theta\in B(\mle,\varepsilon)}{\text{sup}}\|\nabla_{\theta}^{2}\log h_{\theta}(x)\|$.
\item [{(G1)}] Estimators $\emcmle$ and $\ence$ converge to
$\emle$ almost surely, and are such that 
\[
\nabla\lIS(\emcmle)=o\left(m^{-1}\right),\qquad\nabla\lNCE(\ence)=o\left(m^{-1}\right).
\]
\item [{(I1)}] For some $\varepsilon>0$ the following integrability condition holds:
\[
\mathbb{E}_{\psi}\left[b_{\varepsilon}(X)\underset{\theta\in B(\mle,\varepsilon)}{\text{sup}}\left(\frac{h_{\theta}(X)}{h_{\psi}(X)}\right)^{2}\right]<+\infty\]
where $b_{\varepsilon}(x)=1+\underset{\theta\in B(\widehat{\theta}_n,\varepsilon)}{\text{sup}}\left\Vert \nabla_{\theta}\log h_{\theta}(x)\right\Vert$.
\end{description}
Measurability of the suprema in
(H2) and (I1) is ensured by the lower semi-continuity of the two first
differentials in a neighbourhood of $\mle$. Assumption (H2) is a regularity condition that ensures in particular
that the partition function $\theta\mapsto\mathcal{Z}(\theta)=\int_{\mathcal{X}}h_{\theta}(x)\mu(\dx)$
is twice differentiable under the integral sign, in a neighbourhood
of $\mle$. Following Theorem \ref{thm:nce_consistency}, Assumption (G1) is trivial as soon as Assumptions (C1) and (H1)
hold and $\emcmle$ and $\ence$ are exact maximisers; in that case
the gradients are zero. Integrability Assumption
(I1) is the critical assumption. It is essentially a (locally uniform) second moment condition on the importance weights, with $\mathbb{P}_{\mle}$ as the target distribution.

\begin{thm}
\label{thm:equivalence_n_fixed}Under assumptions (X1), (H2), (G1) and
(I1): 
\begin{equation}
m\left(\ence-\emcmle\right)\cvm n \left(- \mathcal{H}(\emle)\right)^{-1}v(\emle)\quad\as\label{eq:cmp_nce_is}
\end{equation}
where $\mathcal{H}(\xi)=\nabla_{\xi}^{2}\ln(\xi)$, and $v(\xi)$ is defined as follows: let
$g_{\xi}(x)=\log h_{\theta}(x)+\nu$, then
\[
v(\xi)=\frac{1}{n}\sum_{i=1}^{n}\nabla_{\xi}g_{\xi}(y_{i})\left(\frac{\exp\{g_{\xi}(y_{i})\}}{h_{\psi}(y_{i})}\right)-\mathbb{E}_{\psi}\left[\nabla_{\xi}g_{\xi}(X)\left(\frac{\exp\{g_{\xi}(X)\}}{h_{\psi}(X)}\right)^{2}\right].
\]
\end{thm}

Before discussing the implications of Theorem 2, it is important to consider \citet{geyer_1994}'s result about asymptotic normality of MC-MLE, which relies on the following assumption:

\begin{description}
\item [{(N)}]For some covariance matrix $\mathbf{A}$ we have:
\[
\sqrt{m}\nabla\lIS(\mle)\cvdm \mathcal{N}_d\left(\mathbf{0}_d,\mathbf{A}\right)
\]
\end{description}

As noticed by \citet{geyer_1994}, asymptotics of MC-MLE are quite similar to
the asymptotics of maximum likelihood, and it can be shown that under
assumptions (X1), (H2), (G1) and (N), 
\[
\sqrt{m}\left(\mcmle-\mle\right)\cvdm \mathcal{N}_d\left(\mathbf{0}_d,\mathbf{H}^{-1}\mathbf{A}\mathbf{H}^{-1}\right).
\]

Theorem 2 shows that the difference between the two point estimates
is $\bigO(m^{-1})$, which is negligible relative to the $\bigOP(m^{-1/2})$
rate of convergence to $\mle$. This proves that, when $n$ is fixed, both approaches are asymptotically equivalent when it comes to approximate the MLE. In particular, Slutsky's lemma implies asymptotic normality of the NCE estimator with the same asymptotic variance as for MC-MLE.

Assumptions (H2) and (I1) admit a much simpler formulation when the model belongs to an exponential family. This is the point of the following Proposition.

\begin{prop}

If the parametric model is exponential, i.e. if $h_{\theta}(x)=\exp\big\{\theta^TS(x)\big\}$ for some statistic $S$, then assumptions (H2) and (I1) are equivalent to the following assumptions (H2-exp) and (I1-exp):

\begin{description}
\item [{(H2-exp)}] The Hessian matrix of the log-likelihood $\mathbf{H}=\nabla^{2}\ln(\mle)$
is invertible. 
\item [{(I1-exp)}] The MLE $\mle$ lies in the interior of $\Theta_{\psi}=\left\{\theta\in\Theta:\mathbb{E}_{\psi}\left[\big(\frac{h_{\theta}(X)}{h_{\psi}(X)}\big)^{2}\right]<+\infty\right\}$. 
\end{description}

\end{prop}

The set $\Theta_{\psi}$ is convex whenever $\Theta$ is. In particular, this is
true when $\Theta$ coincides with the natural space of parameters, defined as
$\widetilde{\Theta}=\{\theta\in\R^d:\int_{\setX}\exp\big\{\theta^TS(x)\big\}\mu(\dx)<+\infty\}$.
If $\refdist\in\left\{\Pt:\theta\in\Theta\right\}$, then (I1-exp) holds as soon
as
$2\mle - \psi$ lies in the interior of $\widetilde{\Theta}$.

\begin{rem}
Condition (N) requires a $\sqrt{m}$-CLT (central limit theorem) for the function $\varphi:x\mapsto\big(\nabla_{\theta}\log h_{\theta}\big)(h_{\theta}/h_{\psi})(x)$ at $\theta=\mle$. There has been an extensive literature on CLT's for Markov Chains, see e.g. \citet{gareth_rosenthal} for a review. In particular, if $(X_j)_{j\ge1}$ is a geometrically ergodic Markov Chain with stationary distribution $\mathbb{P}_{\psi}$, then assumption (N) holds if for some $\delta>0$, $\varphi\in\mathbb{L}_{2+\delta}(\refdist)$. This assumption is very similar to assumption (I1), especially when the model is exponential.
\end{rem}

In practice, implications of Theorem 2 must be considered cautiously, as the
norm of the limit in \eqref{eq:cmp_nce_is} will typically increase with $n$.
For several well-known un-normalised models (e.g. Ising models, Exponential
Random Graph Models), $n$ is equal to one, in which case NCE and MC-MLE will
always produce very close estimates. For other models however, it is known that
the two estimators may behave differently, especially when the number of actual
data-points is big and when simulations have a high computational cost (see
\citet{nce_2012}).

To investigate to which extent both approaches provide a good approximation of
the true parameter value in these models, we will require both $m$ and $n$
to go to infinity. As it turns out, this will also make it possible
to do finer comparison between $\emcmle$ and $\ence$ (and thus between
$\nce$ and $\mcmle$). This is the point of the next section.

\section{Asymptotics of the overall error\label{sec:both_n_and_m}}

We now assume that observations $y_{i}$ are realisations of IID random
variables $Y_{i}$, with
probability density $f_{\trueth}$, for some true parameter
$\theta^{\star}\in\Theta$, while the artificial data-points  $(X_j)_{j\ge1}$
remain generated from a $\refdist$-ergodic process. We also assume that
$(Y_{i})_{i\ge1}$ and $(X_j)_{j\ge1}$ are independent sequences; this regime
was first studied for NCE by \citet{nce_2012}, although the $X_j$'s were
assumed IID in that paper. 

This asymptotic regime has some drawbacks: it assumes that the model is well
specified, and that $\refdist$ is chosen independently from the data, which is
rarely true in practice. Nevertheless, allowing both $m$ and $n$ to go to
infinity turns out to provide a better understanding of the asymptotic
behaviours of NCE and MC-MLE, at least for situations where the number of
actual data-points may be large.

We assume implicitly that $m=m_n$ is a non-decreasing sequence of positive
integers going to infinity when $n$ does, while
$m_n/n\rightarrow\tau\in(0,+\infty)$. Every limit when $n$ goes to infinity
should be understood accordingly. Finally,
$\xi^{\star}=(\theta^{\star},\nu^{\star})$ stands for the true extended
parameter, where
$\nu^{\star}=\log\left\{\mathcal{Z}(\psi)/\mathcal{Z}(\theta^{\star})\right\}$.
 
\subsection{Consistency}

Our results concerning the overall consistency (to $\truee$, as both $m$
and $n\rightarrow\infty$) of MC-MLE and NCE
rely on the following assumptions: 
\begin{description}
\item [{(C2)}] The random sequences $\big(\emcmle\big)_{n\geq1}$ and $\big(\ence\big)_{n\geq1}$ are approximate MC-MLE and NCE estimators, and belong to a compact set almost surely.
\item [{(H3)}] The maps $\theta\mapsto h_{\theta}(x)$ are continuous for
$\refdist$-almost every $x$, and for any $\theta\in\Theta$ there is some $\varepsilon>0$ such that \[ \int_{\setX}\underset{\phi\in B(\theta,\varepsilon)}{\sup}\left(\log\frac{h_{\phi}(x)}{h_{\theta^{\star}}(x)}\right)_{+}h_{\theta^{\star}}(x)\mu(\dx)<+\infty.\]
\end{description}

\begin{thm}
Under assumptions (X1), (C2) and (H3), both estimators $\emcmle$ and $\ence$ converge almost surely to $\truee$
as $n,m\rightarrow\infty$, while $m/n\rightarrow\tau$. \label{thm:consistency}
\end{thm}

Our proofs of NCE and MC-MLE consistency are mainly inspired from
\citet{wald_1949}'s famous proof of MLE consistency, for which the same
integrability condition (H3) is required. It is noteworthy that, under this
regime, MC-MLE and NCE consistency essentially rely on the same assumptions as
MLE consistency. It should be noted that Theorem 3 does not require $\Theta$ to
be a subset of $\mathbb{R}^d$: it holds whenever $\Theta$ is a metric space. 

\begin{prop}
If the parametric model is exponential, i.e. if $h_{\theta}(x)=\exp\big\{\theta^TS(x)\big\}$ for some measurable statistic $S$, then assumption (H3) always holds.
\end{prop}

\subsection{Asymptotic normality }

To ensure the asymptotic normality of both NCE and MC-MLE estimates, we make the following assumption.
\begin{description}
\item[{(X2)}] The sequence $(X_j)_{j\ge1}$ is a Harris ergodic Markov chain (that is, aperiodic, $\phi$-irreducible and positive
Harris recurrent; for definitions see \citet{meyn_tweedie}), with stationary distribution $\refdist$. 

The Markov kernel associated with the chain $(X_j)_{j\ge1}$, noted $P(x,\dy)$, is reversible (satisfies detailed balance) with respect to $\refdist$, that is
\begin{equation}\label{eq:detailed_balance}
\refdist(\dx)P(x,\dy)=\refdist(\dy)P(y,\dx).
\end{equation}
Moreover, the chain $(X_j)_{j\ge1}$ is \textit{geometrically ergodic}, i.e. there is some $\rho\in(0,1)$ and a positive measurable function $M$ such that for $\refdist$-almost every $x$
\begin{equation}
\|P^n(x, .) - \refdist(.)\|_{\mathrm{TV}} \le M(x)\rho^n\label{eq:geometric_ergodicity}
\end{equation}
where $P^n(x; dy)$ denote the $n$-step Markov transition
kernel corresponding to $P$, and $\|.\|_{\mathrm{TV}}$ stands for the total variation norm.
\end{description}
 Under (X2), for any measurable, real-valued function $\varphi$ such that $\mathbb{E}_{\psi}[\varphi^2]<\infty$, then a $\sqrt{m}$-CLT holds, i.e.
\begin{equation}\sqrt{m}\left(\frac{1}{m}\sum_{j=1}^m\varphi(X_j)-\mathbb{E}_{\psi}\left[\varphi(X)\right]\right)\cvd\mathcal{N}(0,\sigma_{\varphi}^2)\label{eq:clt}\end{equation}
where
$$\sigma_{\varphi}^2=\mathbb{V}_{\psi}(\varphi(X))+2\sum_{i=1}^{\infty}\mathrm{Cov}(\varphi(X_0),\varphi(X_i)).$$
 In the equation above, $\mathrm{Cov}(\varphi(X_0),\varphi(X_i))$ stands for the $i$-th lag autocovariance of the chain at stationarity; that is with respect to the distribution defined by $X_0\sim\refdist$ and $X_{i+1}|X_i\sim P(X_i,.)$. The sequence of artificial data-points $(X_j)_{j\ge1}$ is not assumed stationary. Since the chain is Harris recurrent, \eqref{eq:clt} holds whenever $X_1=x$ for any $x\in\mathcal{X}$ (see e.g. \citet{gareth_rosenthal}, especially Theorem 4 and Proposition 29).
 
 For convenience, we choose to assume that the kernel is reversible (which is true for any Metropolis-Hastings algorithm), but the reversibility assumption \eqref{eq:detailed_balance} is not compulsory, and may be replaced by slightly stronger integrability assumptions (see e.g. \citet{gareth_rosenthal}); in particular, if reversibility is not assumed then \eqref{eq:clt} holds whenever $\varphi\in\mathbb{L}_{2+\delta}(\mathbb{P}_{\psi)}$. The critical assumption is geometric ergodicity.
 
Geometric ergodicity is obviously stronger than assumption (X1) which only
requires a law of large numbers to hold. Nevertheless, geometric ergodicity
remains a state of the art condition to ensure CLT's for Markov chains (see
e.g. \citet{gareth_rosenthal} and \citet{bradley_2005}), while it can often be
checked for practical MCMC samplers. We thus present assumption (X2) as a sharp
and practical condition for ensuring CLT's when the artificial data-points are
generated from a MCMC sampler, while it also covers the IID case without loss
of generality.

Our asymptotic normality results rely on the following assumptions:
\begin{description}
\item [{(H4)}] The maps $\theta\mapsto h_{\theta}(x)$ are twice continuously differentiable in a neighborhood of $\theta^{\star}$ for $\refdist$-almost every $x$; the Fisher Information  $\mathbf{I}(\theta)=\mathbb{V}_{\theta}\big(\nabla_{\theta}\log h_{\theta}(Y)\big)$ is invertible at $\theta=\theta^{\star}$; and for some $\varepsilon>0$
\[
\int_{\setX}c_{\varepsilon}(x) \underset{\theta\in B(\trueth,\varepsilon)}{\text{sup}}h_{\theta}(x) \mu(\dx)<\infty
\]
where $c_{\varepsilon}(x)=1+\underset{\theta\in B(\trueth,\varepsilon)}{\text{sup}}\|\nabla_{\theta}\log h_{\theta}(x)\|^{2}+\underset{\theta\in B(\trueth,\varepsilon)}{\text{sup}}\|\nabla_{\theta}^{2}\log h_{\theta}(x)\|$.
\item [{(G2)}] Estimators $\emcmle$ and $\ence$ converge in probability to
$\truee$, and are such that 
\[
\nabla\lIS(\emcmle)=\smalloP\left(n^{-1/2}\right),\qquad\nabla\lNCE(\ence)=\smalloP\left(n^{-1/2}\right).
\]
\item [{(I3)}] At $\theta=\theta^{\star}$, the following integrability condition holds:
\[
\mathbb{E}_{\psi}\left[d_{\theta}(X)\left(\frac{h_{\theta}(X)}{h_{\psi}(X)}\right)^{2}\right]<\infty\]
where $d_{\theta}(x)=1+\left\Vert \nabla_{\theta}\log h_{\theta}(x)\right\Vert^2$.
\end{description}

\begin{thm}
Under assumptions (X2), (H4)  and (G2), we have 
\[
\sqrt{n}\left(\ence-\truee\right)\cvd  \mathcal{N}_{d+1}\left(0,\mathbf{V}_{\tau}^{\mathrm{NCE}}(\truee)\right)\]
where
\begin{align*}
\mathbf{V}_{\tau}^{\mathrm{NCE}}(\xi)&=\mathbf{J}_{\tau}(\xi)^{-1}\left\{\boldsymbol{\boldsymbol{\Sigma}}_{\tau}(\xi)+\tau^{-1}\boldsymbol{\boldsymbol{\Gamma}}_{\tau}(\xi)\right\}\mathbf{J}_{\tau}(\xi)^{-1},\\
\mathbf{J}_{\tau}(\xi)&=\mathbb{E}_{\theta}\bigg[(\nabla_{\xi}\nabla_{\xi}^Tg_{\xi})\bigg(\frac{\tau f_{\psi}}{\tau f_{\psi}+f_{\theta}}\bigg)(Y)\bigg],\\
\boldsymbol{\boldsymbol{\Sigma}}_{\tau}(\xi)&=\mathbb{V}_{\theta}\bigg((\nabla_{\xi}g_{\xi})\bigg(\frac{\tau f_{\psi}}{\tau f_{\psi}+f_{\theta}}\bigg)(Y)\bigg),\\
\boldsymbol{\boldsymbol{\Gamma}}_{\tau}(\xi)&=\mathbb{V}_{\psi}\left(\varphi_{\xi}^{\mathrm{NCE}}(X)\right)+2\sum_{i=1}^{+\infty}\mathrm{Cov}\Big(\varphi_{\xi}^{\mathrm{NCE}}(X_0),\varphi_{\xi}^{\mathrm{NCE}}(X_i)\Big),\\
\varphi_{\xi}^{\mathrm{NCE}}(x)&=(\nabla_{\xi}g_{\xi})\frac{f_{\theta}}{f_{\psi}}\bigg(\frac{\tau f_{\psi}}{\tau f_{\psi}+f_{\theta}}\bigg)(x).
\end{align*}

Moreover, under assumptions (X2), (H4), (G2) and (I3), we have
\[
\sqrt{n}\left(\emcmle-\truee\right)\cvd  \mathcal{N}_{d+1}\left(0,\mathbf{V}_{\tau}^{\mathrm{IS}}(\truee)\right)\]
where \begin{align*}
\mathbf{V}_{\tau}^{\mathrm{IS}}(\xi)&=\mathbf{J}(\xi)^{-1}\left\{\boldsymbol{\Sigma}(\xi)+\tau^{-1}\boldsymbol{\Gamma}(\xi)\right\}\mathbf{J}(\xi)^{-1},\\
\mathbf{J}(\xi)&=\mathbb{E}_{\theta}\left[\nabla_{\xi}\nabla_{\xi}^Tg_{\xi}(Y)\right],\\
\boldsymbol{\Sigma}(\xi)&=\mathbb{V}_{\theta}\Big(\nabla_{\xi}g_{\xi}(Y)\Big),\\
\boldsymbol{\boldsymbol{\Gamma}}(\xi)&=\mathbb{V}_{\psi}\left(\varphi_{\xi}^{\mathrm{IS}}(X)\right)+2\sum_{i=1}^{+\infty}\mathrm{Cov}\Big(\varphi_{\xi}^{\mathrm{IS}}(X_0),\varphi_{\xi}^{\mathrm{IS}}(X_i)\Big),\\
\varphi_{\xi}^{\mathrm{IS}}(x)&=(\nabla_{\xi}g_{\xi})\frac{f_{\theta}}{f_{\psi}}(x).
\end{align*}\label{thm:asymptotic_normality}
\end{thm}
It is noteworthy that second moment condition (I3) is needed for establishing
MC-MLE asymptotic normality, but not for NCE. This shows that, under
the considered regime, NCE is more robust (to $\refdist$) than MC-MLE. It turns out that, when
the artificial data-points are IID, a finer result can be proven. This is the
point of next section. 

Assumptions (H4) and (I3) admit a simpler formulation when the model is
exponential, as shown by the following proposition.

\begin{prop} If the parametric model is exponential, i.e. if
	$h_{\theta}(x)=\exp\big\{\theta^TS(x)\big\}$ for some statistic $S$,
	then assumptions (H4) and (I3) are equivalent to the following
	assumptions (H4-exp) and (I3-exp):

\begin{description} \item [{(H4-exp)}] The Fisher Information
	$\mathbf{I}(\theta)=\mathbb{V}_{\theta}\big(\nabla_{\theta}\log
	h_{\theta}(Y)\big)$ is invertible at $\theta=\theta^{\star}$.  \item
	[{(I3-exp)}] The true parameter $\theta^{\star}$ belongs to the
	interior of
	$\Theta_{\psi}=\left\{\theta:\mathbb{E}_{\psi}\left[\big(\frac{h_{\theta}(X)}{h_{\psi}(X)}\big)^{2}\right]<\infty\right\}$.
	\end{description}

In particular, if $\refdist\in\left\{\Pt\right\}_{\theta\in\Theta}$, then
(I3-exp) holds as soon as $2\trueth - \psi$ belongs to the interior of
$\widetilde{\Theta}=\left\{\theta\in\R^d:\int_{\setX}\exp\big\{\theta^TS(x)\big\}\mu(\dx)<\infty\right\}$.
\end{prop}

\subsection{Comparison of asymptotic variances}
\begin{thm} If the artificial data-points $(X_j)_{j\ge1}$ are IID, then under assumptions (H4) and (I3),
$\mathbf{V}_{\tau}^{\mathrm{IS}}(\truee)\succcurlyeq\mathbf{V}_{\tau}^{\mathrm{NCE}}(\truee)$, i.e. $\mathbf{V}_{\tau}^{\mathrm{IS}}(\truee)-\mathbf{V}_{\tau}^{\mathrm{NCE}}(\truee)$ is a positive semi-definite matrix.\label{thm:variance_comparison}
\end{thm}

Theorem 5 shows that, asymptotically, when $m/n\rightarrow\tau>0$, and when the artificial data-points are IID, the variance of a NCE estimator is always lower than the variance of the corresponding MC-MLE estimator. This inequality is with respect to the Loewner partial order on symmetric matrices. To our knowledge, this is the first theoretical result proving that NCE dominates MC-MLE in terms of mean square error. We failed however to extend this result to correlated Markov chains.

This inequality holds for any fixed ratio $\tau\in(0,+\infty)$, and any given
sampling distribution $\refdist$, but the sharpness of the bound remains
unknown. Typically, the bigger is $\tau$, the closer the two variances will be,
as the ratio $\tau f_{\psi}/\tau f_{\psi}+ f_{\theta^{\star}}$ gets closer to
one. It is also the case when the sampling distribution $\refdist$ is close to
the true data distribution $\mathbb{P}_{\theta^{\star}}$. \citet{geyer_1994}
noticed that MC-MLE performs better when $\refdist$ is close to
$\mathbb{P}_{\theta^{\star}}$. Next proposition shows that when
$\refdist=\mathbb{P}_{\theta^{\star}}$, both variances can be related to the
variance of the MLE.

\begin{prop}
If the artificial data-points are IID sampled from $\refdist=\mathbb{P}_{\theta^{\star}}$, then under assumptions (H4) and (I3) we have
$$\mathbf{V}_{\tau}^{\mathrm{NCE}}(\truee)=\mathbf{V}_{\tau}^{\mathrm{IS}}(\truee)=(1+\tau^{-1})\mathbf{V}^{\mathrm{MLE}}(\truee)$$
where $\mathbf{V}^{\mathrm{MLE}}(\xi)=\mathbf{J}(\xi)^{-1}\boldsymbol{\Sigma}(\xi)\mathbf{J}(\xi)^{-1}$. \label{prop_var_mle}
\end{prop}
It is straightforward to check that, under the usual conditions ensuring
asymptotic normality of the MLE, the extended maximiser of the Poisson
Transform $\ln$ is also asymptotically normal with variance
$\mathbf{V}^{\mathrm{MLE}}(\truee)$. This proposition shows what we can expect
from NCE and MC-MLE in a ideal scenario where the sampling distribution is the
same as the true data distribution.

\section{Numerical example\label{sec:Numerical-study}}

This section presents a numerical example that illustrates how the 
variance reduction brought by NCE may vary according to
the sampling distribution $\refdist$ and the ratio $\tau$. 

We consider observations IID distributed from the multivariate Gaussian distribution $\mathcal{N}_p(\mu,\Sigma)$ truncated to $]0,+\infty[^p$; that is $Y_1,...,Y_n$ are IID with the following probability density with respect to Lebesgue's measure: 
\[
	f_{\mu,\Sigma}(x)=\frac{1}{\mathcal{Z}(\mu,\Sigma)}\exp\left\{-\frac{1}{2}(x-\mu)^{T}\Sigma^{-1}(x-\mu)\right\}\mathds{1}_{]0,+\infty[^p}(x)
\] 
where
$$\mathcal{Z}(\mu,\Sigma)=(2\pi)^{p/2}|\Sigma|^{1/2}\mathbb{P}\left(W\in]0,+\infty[^p\right),\qquad W\sim \mathcal{N}_p(\mu,\Sigma).$$
The probability $\mathbb{P}\left(W\in]0,+\infty[^p\right)$ is intractable for
almost every $(\mu,\Sigma)$. Numerical approximations of such probabilities
quickly become inefficient when $p$ increases. 
We generate IID artificial data-points from density 
$f_{\mu,\Sigma}$ with $\mu=\mathbf{0}_p$ and $\Sigma=\lambda\mathbf{I}_p$ for
some $\lambda>0$.  

It is well known that (truncated) Gaussian densities form an 
exponential family under the following parametrisation: for a given
$\mu\in\mathbb{R}^p$ and $\Sigma\in\mathbb{S}_p^{++}$ (the set of positive 
definite matrices of size $p$), define
$\theta=(\Sigma^{-1}\mu,\mathrm{triu}(-(1/2)\Sigma^{-1}))$, and
$S(x)=(x,\mathrm{triu}(xx^{T}))$, where $\mathrm{triu}(.)$ is the upper
triangular part. 
This parametrisation is minimal and the natural parameter space is a convex
open subset of $\mathbb{R}^q$ where $q=p+p(p+1)/2$. Indeed, under the
exponential formulation, we have $\Theta=\Theta_1\times\Theta_2$ where
$\Theta_1=\mathbb{R}^p$ and $\Theta_2$ is on open cone of
$\mathbb{R}^{p(p+1)/2}$, in bijection with $\mathbb{S}_p^{++}$ through the
function $\mathrm{triu}(.)$. 

The observations are sampled IID from $\Pt$ for some true parameter
$\theta=\theta^{\star}$, corresponding to 
$$\mu^{\star}=\begin{pmatrix}
1  \\
-1 \\
0.5
\end{pmatrix},\qquad\Sigma^{\star}=\begin{pmatrix}
1 & 0.5 & 1\\
0.5 & 1.5 & 0.3 \\
1 & 0.3 & 2
\end{pmatrix},$$
in the usual Gaussian parametrisation. The sample size is fixed to $n=1000$,
while $m$ is chosen such that the ratio $m/n$ is equal to
$\tau\in\{1,5,20,100\}$. The distribution $\refdist$ is chosen as stated above
for $\lambda\in[1.5,20]$. 

Figure \ref{fig:both_estimates_vs_MLE} plots estimates and
confidence intervals of the mean square error ratio (mean square error of 
the estimator divided by the asymptotic variance of the MLE) 
of both estimators (NCE and MC-MLE), based on 1000 independent replications. 
(Regarding the numerator of this ratio, note that the variance of the MLE may
be estimated by performing noise contrastive estimation with
$\refdist=\mathbb{P}_{\theta^*}$, see Proposition \ref{prop_var_mle}.) 

\begin{figure}[!tbp]
  \centering
   \begin{minipage}[c]{.48\linewidth}
      \includegraphics[width=\linewidth]{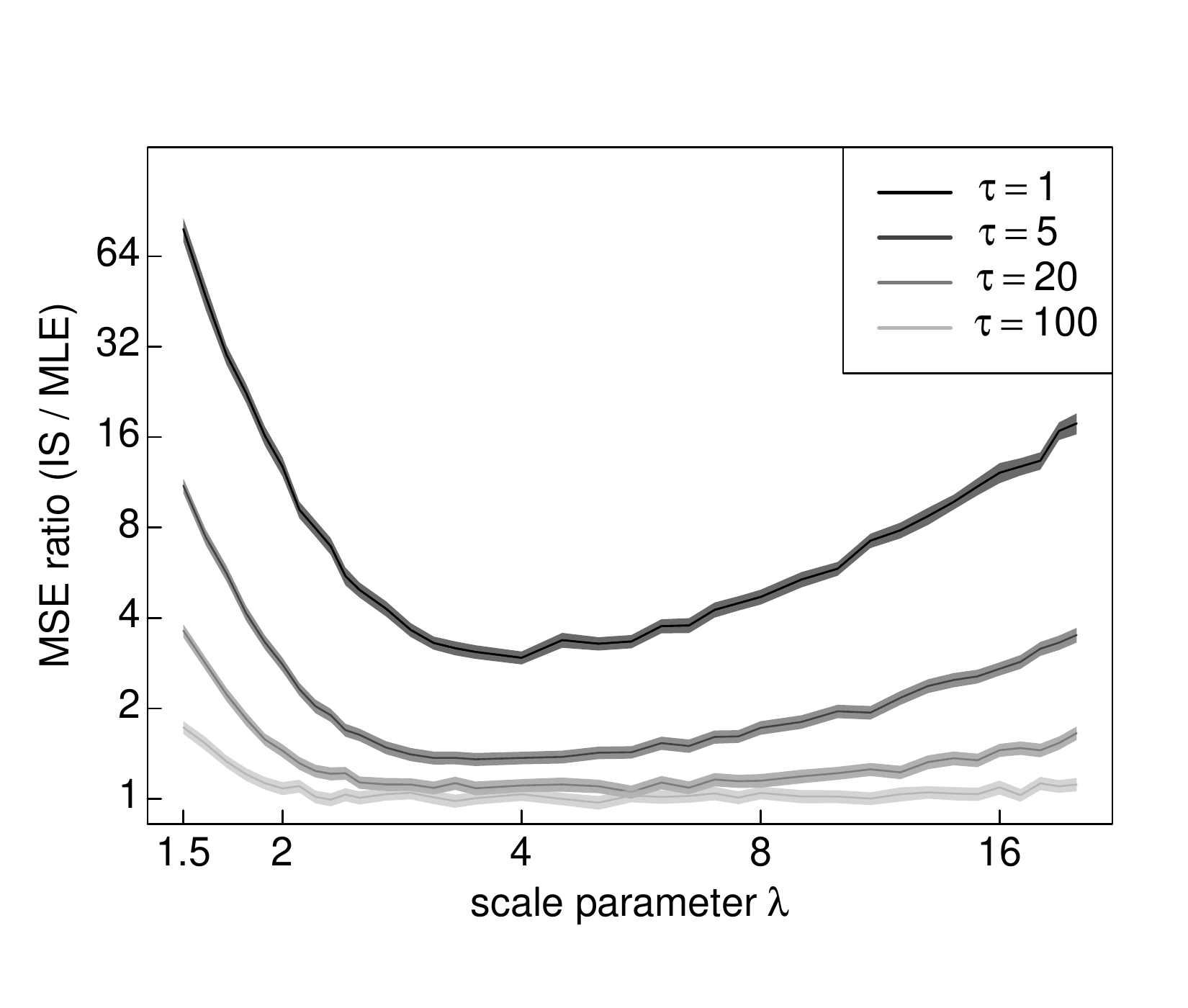}
   \end{minipage} \hfill
   \begin{minipage}[c]{.48\linewidth}
      \includegraphics[width=\linewidth]{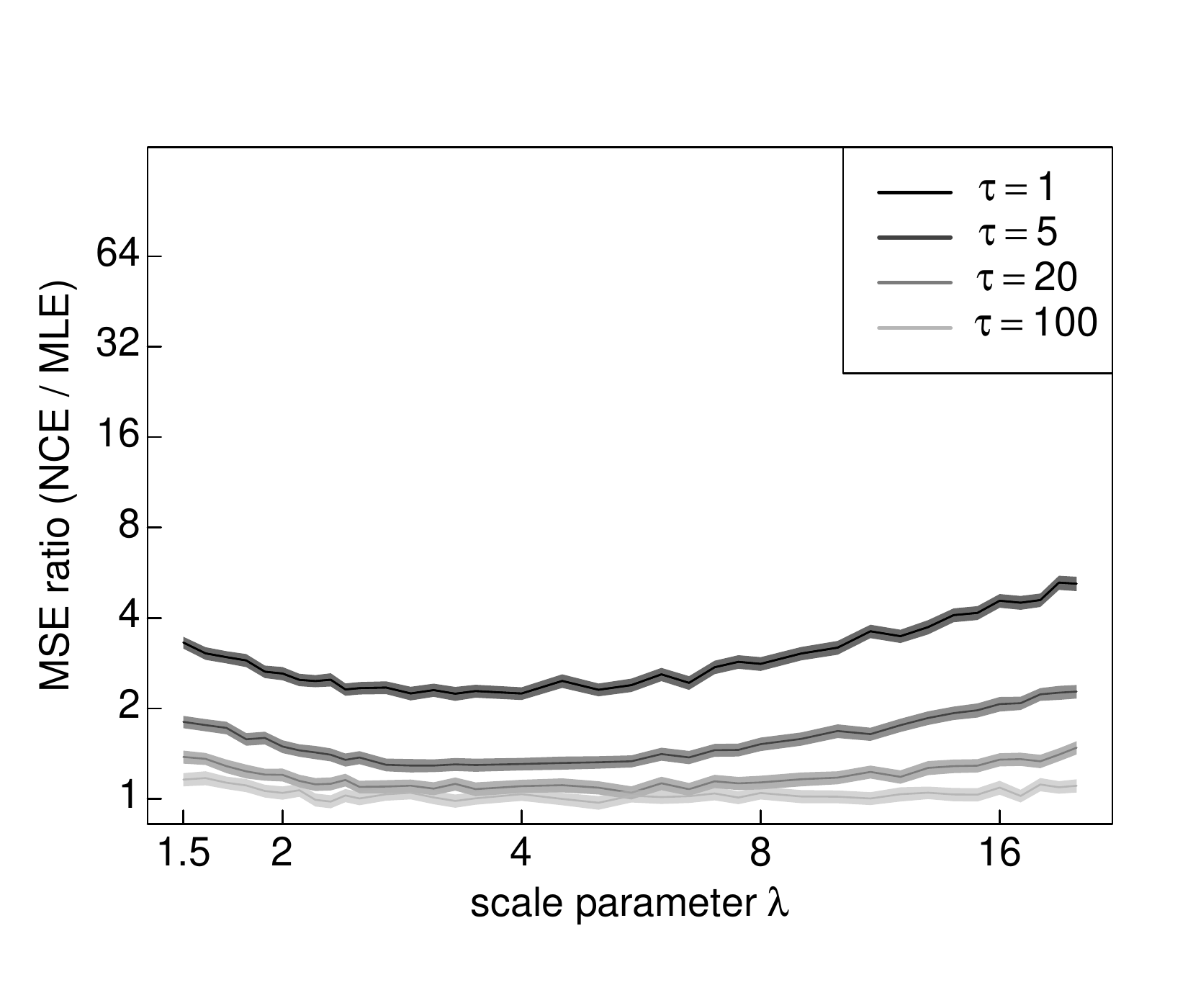}
   \end{minipage}
   \captionof{figure}{Estimates and confidence intervals of the Mean Square
   Error ratios of MC-MLE (left) and NCE (right), compared to the MLE. The MSE
   ratio depends both on the variance of the proposal distribution $\lambda$
   and the number of artificial data-points $m=\tau\times n$ ($n=1000$). A
   log-scale is used for both axes.}
   \label{fig:both_estimates_vs_MLE}
\end{figure}

To facilitate the direct comparison between NCE and MC-MLE, we also plot in
Figure \ref{fig:MCMLEvsNCE} estimates and confidence intervals of the MSE ratio
of MC-MCLE over NCE. As expected from Theorem \ref{thm:variance_comparison}, 
this ratio is always higher than one; it becomes larger and larger as $\tau$
decreases, or as $\lambda$ moves away from its optimal value (around 4). 
This suggests that NSE is more robust than MC-MLE to a poor choice for the 
reference distribution. 

\begin{figure}[!tbp]
	\begin{center}
      \begin{minipage}[c]{.48\linewidth}
      \includegraphics[width=\linewidth]{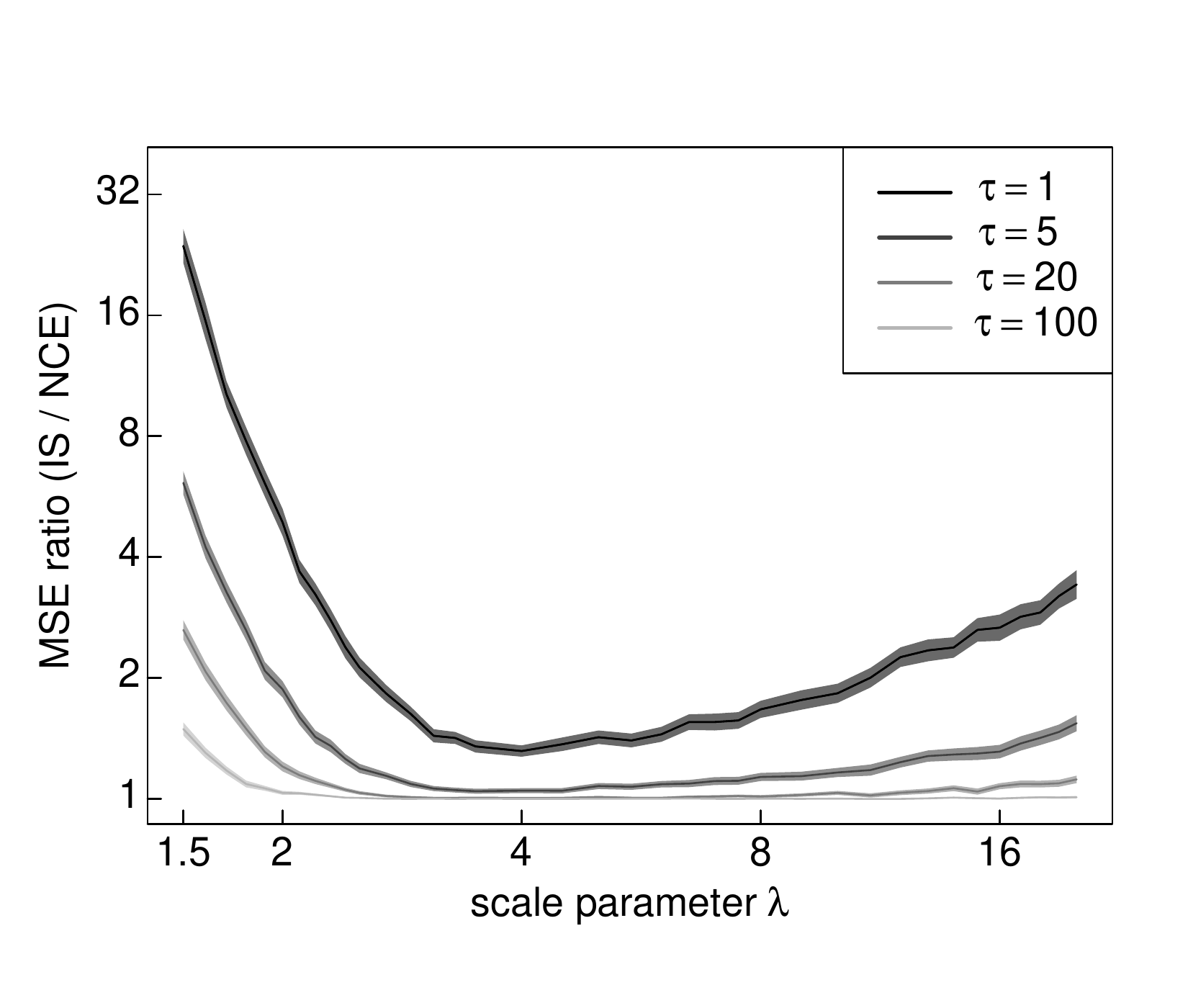}
   \end{minipage}
\end{center}
      \captionof{figure}{Estimates and confidence intervals of the Mean Square
      Error ratios of MC-MLE, compared to the NCE. The MSE ratio depends both
      on the variance of the proposal distribution $\lambda$ and the number of
      artificial data-points $m=\tau\times n$ ($n=1000$). A log-scale is used
      for both axis.}
      \label{fig:MCMLEvsNCE}
\end{figure}

Finally, we discuss a technical difficulty related to the constrained nature of
the parameter space $\Theta$. In principle, both the NCE and the MC-MLE
estimators should be obtained through constrained optimisation (i.e.  as
maximisers of their respective objective functions over $\Theta$). However, it
is much easier (here, and in many cases) to perform an unconstrained
optimisation (over $\R^q$). We must check then that the so obtained solution
fulfils the constraint that defines $\Theta$ (here, that the solution
corresponds to a matrix $\Sigma$ which is definite positive). Figure
\ref{fig:prob_exists} plots estimates and confidence intervals of the
probability that both estimators belong to $\Theta$. We see that NCE (when
implemented without constraints) is much more likely to produce estimates that
belong to $\Theta$. 

Note also that when the considered model is an exponential family (as in this
case), both functions $\lIS$ and $\lNCE$ are convex. This implies that, when
the unconstrained maximiser of these functions do not fulfil the constraint
that defines $\Theta$, then the constrained maximiser does not exist. (Any
solution of the constrained optimisation program lies on the boundary of the
constrained set.)

\begin{figure}[!tbp]
  \centering
   \begin{minipage}[c]{.48\linewidth}
      \includegraphics[width=\linewidth]{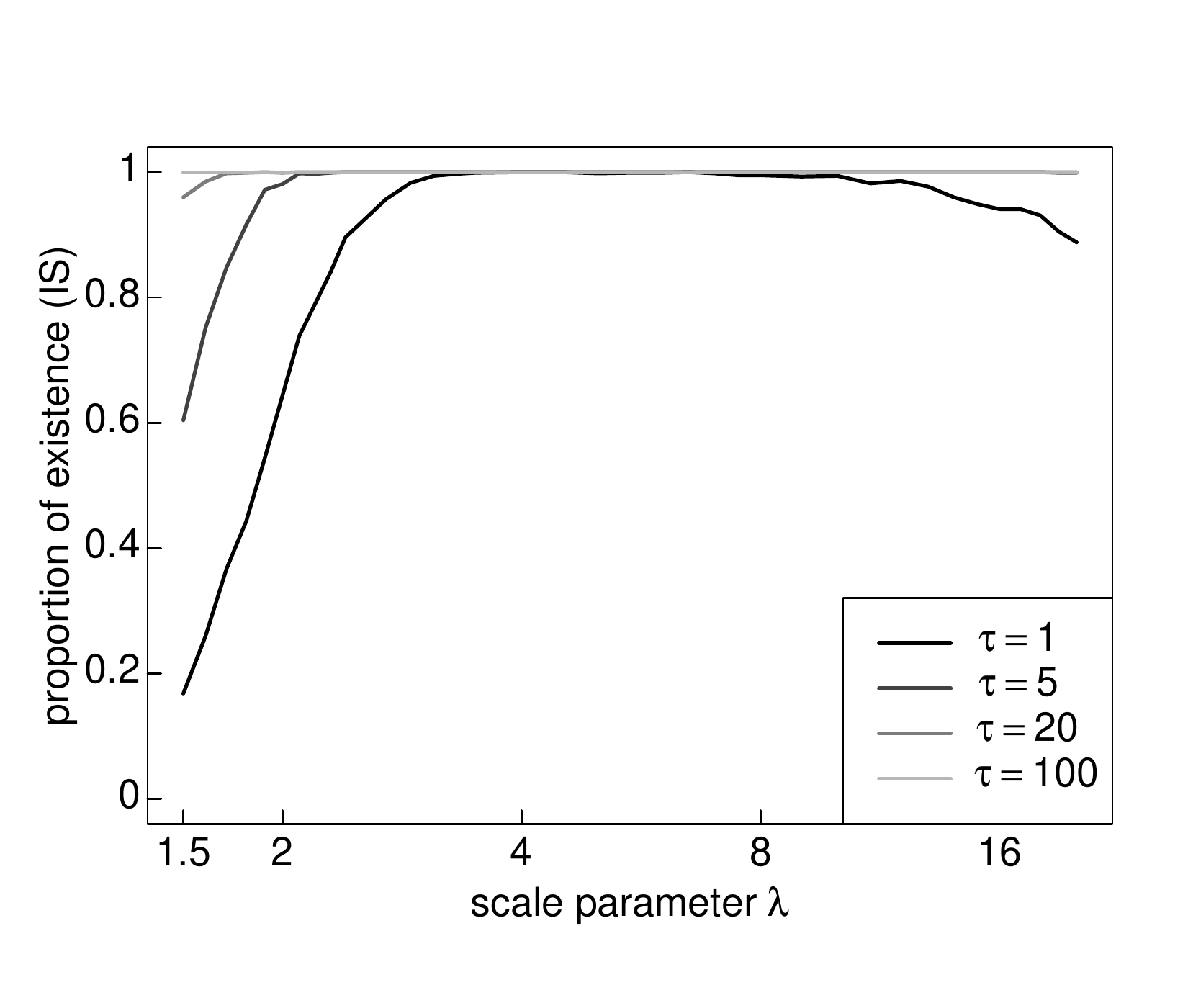}
   \end{minipage} \hfill
   \begin{minipage}[c]{.48\linewidth}
      \includegraphics[width=\linewidth]{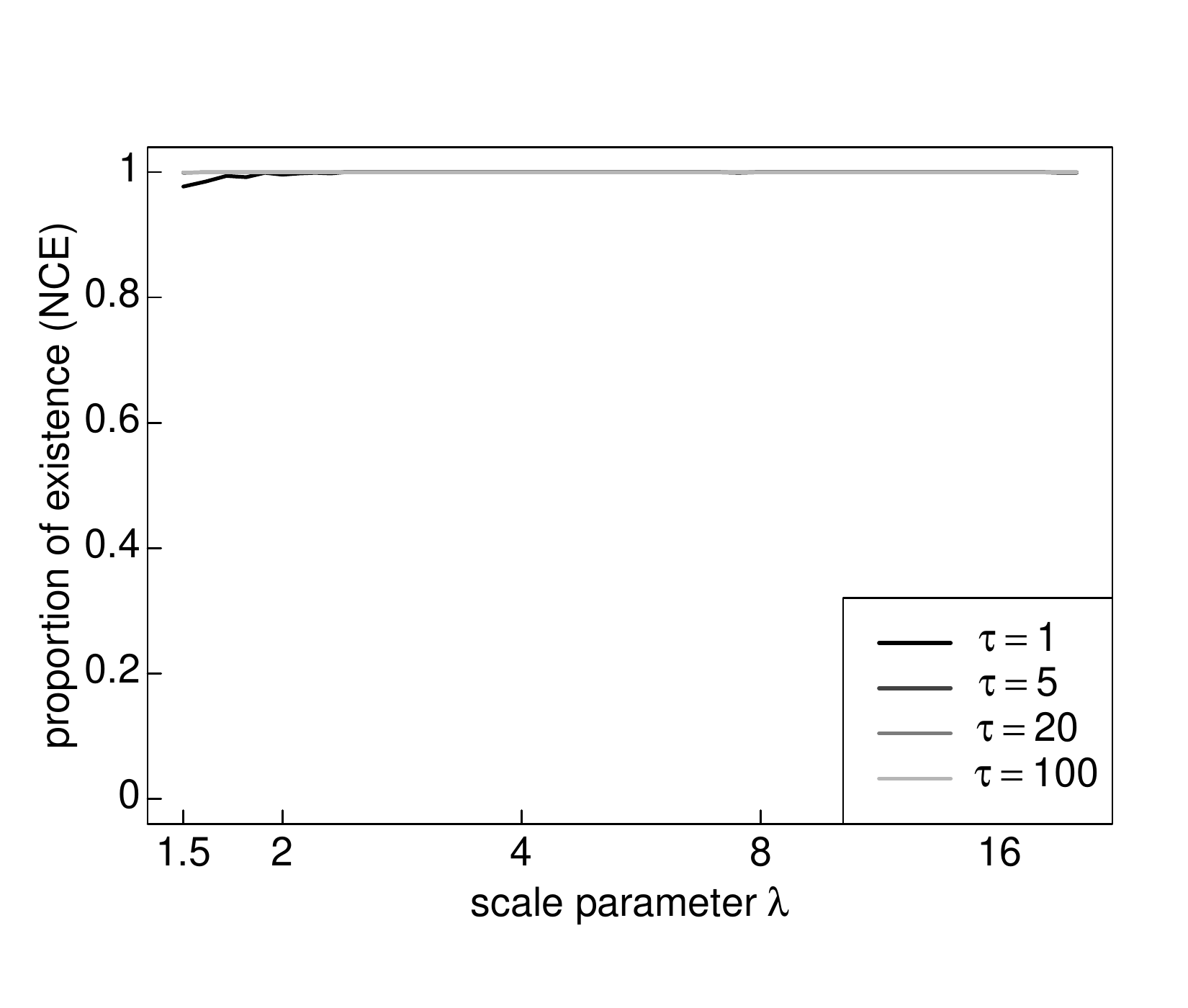}
   \end{minipage}
      \captionof{figure}{Estimates and confidence intervals of the probability
      of existence of MC-MLE (left) and NCE (right) estimators. For a fixed
      $n=1000$, the probability of belonging to $\Theta$ is lower for MC-MLE,
      especially for small values of the variance of the proposal distribution
      $\lambda$ and the number of artificial data-points $m=\tau\times n$. A log-scale is used for both axis.
      }
      \label{fig:prob_exists} 
\end{figure}

\section{Conclusion\label{sec:Conclusion}}

The three practical conclusions we draw from our results are that: (a) NCE is
as widely applicable as MCMC-MLE (including when the $X_j'$s are generated
using MCMC); (b) NCE and MC-MLE are asymptotically equivalent (as $m\rightarrow
\infty$) when $n$ is fixed; (c) NCE may provide lower-variance estimates than
MC-MLE when $n$ is large (provided that $m=\bigO(n)$).  The variance reduction
seems to be more important when the ratio $\tau = m/n$ is small, or when the
reference distribution (for generating the $X_j$'s) is poorly chosen. 
Note that we proved (c) under the assumption that the $X_j$'s are IID, 
but we conjecture it also holds when they are generated using 
MCMC. Proving this conjecture may be an interesting avenue 
for future research.

As mentioned in the introduction, another advantage of NCE is its ease of 
implementation. In particular, when the considered model is exponential, 
NCE boils down to performing a standard logistic regression. 
For all these reasons, it seems reasonable to recommend NCE as the default
method to perform inference for un-normalised models. 

\section*{Acknowledgements}

The  research  of  the  first  author  is  funded  by  a  GENES  doctoral
scholarship.   The  research  of the  second  author  is  partially  supported
by  a  grant  from  the  French  National  Research  Agency (ANR) as part of
the Investissements d'Avenir program (ANR-11-LABEX-0047).  
We are grateful to Bernard Delyon for
letting us include in the supplement an English translation of 
some technical results (and their proofs) on ergodic processes that he 
derived in lecture notes. 

\appendix

\section{Proofs}

\begin{subsection}{Technical lemmas}
The following lemmas are prerequisites for the proofs of our main theorems.
Most of them are classical results, but for the sake of completeness, we
provide the proofs of these lemmas in the supplement. All these lemma apply 
to a $\refdist$-ergodic sequence of random variables, 
$(X_j)_{j\ge1}$.

First lemma is a slightly disguised version of the law of large numbers, combined with the monotone convergence of a sequence of test functions.

\begin{lem}\label{lem:LLN_monotone}
Let $(f_m)_{m\ge1}$ be a non-decreasing sequence of measurable, non negative real-valued functions converging pointwise towards $f$. Then we have:
\[\frac{1}{m}\sum_{j=1}^mf_{m}(X_j)\underset{m\rightarrow+\infty}{\overset{a.s.}{\longrightarrow}}\mathbb{E}_{\psi}[f(X)] . 
\]
This result holds whether the expectation is finite or infinite.
\end{lem}

Second lemma is a natural generalisation of Lemma \ref{lem:LLN_monotone} to dominated convergence.

\begin{lem}\label{lem:LLN_dominated}
Let $(f_m)_{m\ge1}$, $f$ and $g$ be measurable, real-valued functions, such that $(f_m)_{m\ge1}$ converges pointwise towards $f$; for any $m\ge1$, $|f_m|\le g$ ; and $\mathbb{E}_{\psi}[g(X)]<+\infty$. Then we have:
\[\frac{1}{m}\sum_{j=1}^mf_{m}(X_j)\underset{m\rightarrow+\infty}{\overset{a.s.}{\longrightarrow}}\mathbb{E}_{\psi}[f(X)].
\]
\end{lem}

Third lemma is a generalisation of Lemma \ref{lem:LLN_monotone} to the
degenerate case where the expectation is infinite. In that case, Lemma
\ref{lem:LLN_infinite} shows that the monotonicity assumption is unnecessary.

\begin{lem}\label{lem:LLN_infinite}
Let $(f_m)_{m\ge1}$, $f$ and $g$ be measurable, real-valued functions, such
that $(f_m)_{m\ge1}$ converges pointwise towards $f$; $g$ is non negative,
$\mathbb{E}_{\psi}[g(X)]<+\infty$; for any $m\ge1$, $f_m\le g$; and
$\mathbb{E}_{\psi}[f(X)_{-}]=+\infty$ where $f_{-}$ stands for the negative
part of $f$. Then we have:
\[\frac{1}{m}\sum_{j=1}^mf_{m}(X_j)\underset{m\rightarrow+\infty}{\overset{a.s.}{\longrightarrow}}-\infty.
\]
\end{lem}

Fourth lemma is a uniform law of large numbers. It is well known in the IID
case. This result does not actually require the independence assumption. We
present a generalisation of this result to ergodic processes. The proof is due
to Bernard Delyon, who made it available in an unpublished course in French (\citet{delyon}). We
present an English translation of the proof in the supplement.

\begin{lem} Let $K$ a compact subset of $\mathbb{R}^{d}$; $(\theta,x)\mapsto
	\varphi(\theta,x)$ a measurable function defined on
	$K\times\mathcal{X}$ whose values lie on $\mathbb{R}^p$; and suppose
	that the maps $\theta\mapsto \varphi(\theta,x)$ are continuous for
	$\refdist$-almost every $x$. Moreover, suppose that
\[\mathbb{E}_{\psi}\bigg[\underset{\theta\in K}{\text{sup}}\ \|\varphi(\theta,X)\|\bigg]<+\infty.\]
Then the function $\theta\mapsto\mathbb{E}_{\psi}\big[\varphi(\theta,X)\big]$ defined on $K$ is continuous, and we have
\[\underset{\theta\in K}{\text{sup}}\ \left\|\frac{1}{m}\sum_{j=1}^m\varphi(\theta,X_j)-\mathbb{E}_{\psi}\left[\varphi(\theta,X)\right]\right\|\overset{a.s.}{\underset{m\rightarrow+\infty}{\longrightarrow}}0.\]
Consequently, if there is a random sequence $(\widetilde{\theta}_m)_{m\ge1}$ converging almost surely to some parameter $\widetilde{\theta}\in \Theta$. Then we have
$$ \left\|\frac{1}{m}\sum_{j=1}^m\varphi(\widetilde{\theta}_m,X_j)-\mathbb{E}_{\psi}\left[\varphi(\widetilde{\theta},X)\right]\right\|\cvm0\quad\as$$ \label{lem:ULLN_theorem}
\end{lem}

Fifth lemma is also a well known result. It is often used to prove the weak convergence (usually asymptotic normality) of Z-estimators.

\begin{lem}
Define any probability space $(\Omega,\mathcal{F},\mathbb{P})$, and let $(\ln(\theta,\omega))_{n\ge1}$ be measurable real-valued functions defined on $\mathbb{R}^d\times\Omega$. Let $\theta^{\star}\in\mathbb{R}^d$ and $\varepsilon>0$ such that for any $n\ge1$ and for $\mathbb{P}$-almost every $\omega\in\Omega$ the map $\theta\mapsto \ln(\theta,\omega)$ is $C^2$ on $B(\theta^{\star},\varepsilon)$. Let $(\widehat{\theta}_n)_{n\ge1}$ be a random sequence converging in probability to $\theta^{\star}$. Suppose also that:

\begin{description}
\item [{(a)}] $\{\nabla_{\theta}\ln(\theta)\}_{|_{\theta=\widehat{\theta}_n}}=o_{\mathbb{P}}(n^{-1/2})$,
\item [{(b)}] $\underset{\theta\in B(\theta^{\star},\varepsilon)}{\sup}\|\nabla_{\theta}^2 \ln(\theta)-\mathcal{H}(\theta)\|\overset{\mathbb{P}}{\longrightarrow}0$, for some $\mathbb{R}^{d\times d}$ valued function $\mathcal{H}$ continuous at $\theta^{\star}$, such that $\mathcal{H}(\theta^{\star})$ is full rank,
\item [{(c)}] $\sqrt{n}\{\nabla_{\theta}\ln(\theta)\}_{|_{\theta=\theta^{\star}}}\cvd Z$, for some random vector $Z$.
\end{description}

Then
$$\sqrt{n}(\widehat{\theta}_n-\theta^{\star})+\mathcal{H}(\theta^{\star})^{-1}\sqrt{n}\{\nabla_{\theta}\ln(\theta)\}_{|_{\theta=\theta^{\star}}}\overset{\mathbb{P}}{\longrightarrow}0_{\mathbb{R}^d},$$
and, consequently
$$\sqrt{n}(\widehat{\theta}_n-\theta^{\star})\cvd-\mathcal{H}(\theta^{\star})^{-1}Z.$$\label{lem:asymptotic_normality}
\end{lem} 

Sixth lemma is a technical tool required for proving asymptotic normality of NCE. It is particularly straightforward to prove in the IID case. We present a generalisation of this result to reversible, geometrically ergodic Markov chains.

\begin{lem}
 Assume that (X2) holds. Let $(f_n)_{n\ge1}$, $f$ and $g$ be measurable, real-valued functions, such that $(f_n)_{n\ge1}$ converges pointwise towards $f$; for any $n\ge1$, $|f_n|\le g$; and $\mathbb{E}_{\psi}[g(X)^2]<\infty$. Then we have
$$\sqrt{n}\left(\frac{1}{n}\sum_{i=1}^n\big\{f_n(X_i)-f(X_i)\big\}-\mathbb{E}\big[f_n(X)-f(X)\big]\right)\overset{\mathbb{P}}{\longrightarrow}0,$$
and, consequently
$$\sqrt{n}\left(\frac{1}{n}\sum_{i=1}^nf_n(X_i)-\mathbb{E}[f_n(X)]\right)\cvd\mathcal{N}\big(0,\sigma_f^2),$$
where $\sigma_f^2=\mathbb{V}_{\psi}(f(X))+2\sum_{i=1}^{+\infty}\mathrm{Cov}(f(X_0),f(X_i))<+\infty$.\label{lem:CLT_dominated}
\end{lem}

\end{subsection}

\subsection{Proof of Theorem \ref{thm:nce_consistency}\label{subsec:Proof-consistent}}
A standard approach to establish consistency of M-estimators is to prove some Glivenko-Cantelli
result (uniform convergence), but, to the best of our knowledge, no
such result exists under the general assumption that the underlying
random variables (the $X_{j}$'s in our case) are generated from an ergodic process.
Instead, we follow \citet{geyer_1994}'s approach, which relies on establishing that function $-\lNCE$
epiconverges to $-\ell_{n}$. Epiconvergence is essentially the most general notion
of convergence for functions that ensures the convergence of minimisers;
for a succint introduction to epiconvergence, see Appendix A of \citet{geyer_1994}
and Chapter 7 of \citet{variational}.

We follow closely \citet{geyer_1994}. In particular, Theorem 4 of \citet{geyer_1994} shows that: if a sequence of functions $\ell_{n,m}$ hypoconverges to some function $\ell_{n}$ which has a unique maximiser $\mle$ and if a random sequence $(\widehat{\theta}_{n,m})_{m\ge1}$ is an approximate maximiser of $\ell_{n,m}$ which belongs to a compact set almost surely, then $\widehat{\theta}_{n,m}$ converges to $\mle$ almost surely. Consequently, to prove Theorem \ref{thm:nce_consistency}, we only have to prove that
$\lNCE$ hypoconverges to $\ell_{n}$ (i.e. that $-\lNCE$ epiconverges
to $-\ell_{n}$); that is 
\begin{equation}
    \ell_{n}(\theta,\nu)\le\underset{B\in\mathcal{N}(\theta,\nu)}{\text{inf}}\underset{m\rightarrow+\infty}{\text{lim inf}}\underset{(\phi,\mu)\in B}{\text{sup}}\left\{ \lNCE(\phi,\mu)\right\} \label{inequality_1_geyer}
\end{equation}
\begin{equation}
\ell_{n}(\theta,\nu)\ge\underset{B\in\mathcal{N}(\theta,\nu)}{\text{inf}}\underset{m\rightarrow+\infty}{\text{lim sup}}\underset{(\phi,\mu)\in B}{\text{sup}}\left\{ \lNCE(\phi,\mu)\right\} \label{inequality_2_geyer}
\end{equation}
where $\mathcal{N}(\theta,\nu)$ denotes the set of neighborhoods
of the point $(\theta,\nu)$. 

Since $\Xi=\Theta\times\mathbb{R}$ is a separable metric space, there
exists a countable base $\mathcal{B}=\{B_{1},B_{2},...\}$ for the
considered topology. For any point $(\theta,\nu)$ define the countable
base of neighborhoods $\mathcal{N}_{c}(\theta,\nu)=\mathcal{B}\cap\mathcal{N}(\theta,\nu)$
which can replace $\mathcal{N}(\theta,\nu)$ in the infima of the
preceding inequalities. Choose a countable dense subset $\Gamma_{c}=\{(\theta_{1},\nu_{1}),(\theta_{2},\nu_{2}),...\}$
as follows. For each $k$ let $(\theta_{k},\nu_{k})$ be a point of
$B_{k}$ such that: 
\[
\ell_{n}(\theta_{k},\nu_{k})\ge\underset{(\phi,\mu)\in B_{k}}{\text{sup}}\left\{ \ell_{n}(\phi,\mu)\right\} -\frac{1}{k}.
\]

The proof is very similar to Theorem 1 of \citet{geyer_1994}. However, in this slightly different proof, we will need 

\begin{equation}
    \underset{m\rightarrow+\infty}{\text{lim}}\ \left[\frac{1}{m}\sum_{j=1}^{m}\log\left\{\left(1+e^{\nu}\frac{nh_{\theta}(X_{j})}{mh_{\psi}(X_{j})}\right)^{\frac{m}{n}}\right\}\right]=\mathbb{E}_{\psi}\left[e^{\nu}\frac{h_{\theta}(X)}{h_{\psi}(X)}\right]=e^{\nu}\frac{\mathcal{Z}(\theta)}{\mathcal{Z}(\psi)}\label{simultaneous_limit_1}
\end{equation}
and 
\begin{equation}
\underset{m\rightarrow+\infty}{\text{lim}}\ \frac{1}{m}\sum_{j=1}^{m}\log\left(1+\frac{n}{m}\underset{(\phi,\mu)\in B}{\text{inf}}\ \left[e^{\mu}\frac{h_{\phi}(X_{j})}{h_{\psi}(X_{j})}\right]\right)^{\frac{m}{n}}=\mathbb{E}_{\psi}\left[\underset{(\phi,\mu)\in B}{\text{inf}}\ \left\{ e^{\mu}\frac{h_{\phi}(X)}{h_{\psi}(X)}\right\}\right] \label{simultaneous_limit_2}
\end{equation}
to hold simultaneously with probability one for any $(\theta,\nu)\in\Gamma_{c}$
and any $B\in\mathcal{B}$. For any fixed $(\theta,\nu)$, Lemma \ref{lem:LLN_monotone} applies to the maps $x\mapsto(1+\frac{x}{m})^m$, and since any countable union
of null sets is still a null set, convergence holds simultaneously for every element of $\Gamma_c$ and $\mathcal{B}$ with probability one. One may note that infima in the last
equation are measurable under (H1) (in that
case, an infima over any set $B\in\mathcal{B}$ can be replaced by an infima
over the countable dense subset $B\cap\Gamma_{c}$).

Proving inequality \eqref{inequality_1_geyer} is
straightforward: 
\[
\forall B\in\mathcal{B},\ \ \forall(\theta,\nu)\in B\cap\Gamma_{c},\hspace{0.5cm}\ell_{n}(\theta,\nu)=\lim_{m\rightarrow+\infty}\lNCE(\theta,\nu)\le\underset{m\rightarrow+\infty}{\text{lim inf}}\ \underset{(\phi,\mu)\in B}{\text{sup}}\ \left\{ \lNCE(\phi,\mu)\right\} 
\]
and thus
\[
\underset{B\in\mathcal{N}_{c}(\theta,\nu)}{\text{inf}}\ \underset{(\phi,\mu)\in B\cap\Gamma_{c}}{\text{sup}}\ \left\{ \ell_{n}(\phi,\mu)\right\} \le\underset{B\in\mathcal{N}_{c}(\theta,\nu)}{\text{inf}}\ \underset{m\rightarrow+\infty}{\text{lim inf}}\ \underset{(\phi,\mu)\in B}{\text{sup}}\ \left\{ \lNCE(\phi,\mu)\right\} .
\]
\citep{geyer_1994} proved that $\theta\mapsto\mathcal{Z}(\theta)$ is lower semi-continuous (cf Theorem 1).
This result directly implies that $(\theta,\nu)\mapsto\ell_{n}(\theta,\nu)$
is upper semi-continuous as a sum of upper semi-continuous functions.
Thus the left hand side is equal to $l(\theta,\nu)$ by construction
of $\Gamma_{c}$. 

The proof of the second inequality also follows closely \citet{geyer_1994}:
\begin{align*}
\underset{B\in\mathcal{N}(\theta,\nu)}{\text{inf}}\  \underset{m\rightarrow+\infty}{\text{lim sup}}\ \underset{(\phi,\mu)\in B}{\text{sup}}&\left\{ \lNCE(\phi,\mu)\right\} \\ \le&\underset{B\in\mathcal{N}(\theta,\nu)}{\text{inf}}\ \Bigg\{\underset{(\phi,\mu)\in B}{\text{sup}}\ \left[\frac{1}{n}\sum_{i=1}^{n}\log\left\{\frac{h_{\phi}(y_{i})}{h_{\psi}(y_{i})}\right\}+\mu\right]\\
  &-\underset{m\rightarrow+\infty}{\text{lim inf}}\ \underset{(\phi,\mu)\in B}{\text{inf}}\ \left[\frac{1}{n}\sum_{i=1}^{n}\log\left\{1+\frac{n}{m}e^{\mu}\frac{h_{\phi}(y_{i})}{h_{\psi}(y_{i})}\right\}\right]\\
  &-\underset{m\rightarrow+\infty}{\text{lim inf}}\ \frac{1}{m}\sum_{j=1}^{m}\log\left(1+\frac{n}{m}\underset{(\phi,\mu)\in B}{\text{inf}}\ \left[e^{\mu}\frac{h_{\phi}(X_{j})}{h_{\psi}(X_{j})}\right]\right)^{\frac{m}{n}}\Bigg\}\\
 =&\frac{1}{n}\sum_{i=1}^{n}\log\left\{\frac{h_{\theta}(y_{i})}{h_{\psi}(y_{i})}\right\}+\nu-\underset{B\in\mathcal{N}(\theta,\nu)}{\text{sup}}\ \mathbb{E}_{\psi}\left[\underset{(\phi,\mu)\in B}{\text{inf}}\ \left\{ e^{\mu}\frac{h_{\phi}(X)}{h_{\psi}(X)}\right\}\right].
\end{align*}
The inequality follows directly from superadditivity of the supremum (and subadditivity
of the infimum) and the continuity and monotonicity of the maps $x\mapsto\log(1+\frac{n}{m}x)^{\frac{m}{n}}$.
The last equality holds because the infimum over $\mathcal{N}(\theta,\nu)$
can be replaced by the infimum over the countable set $\mathcal{A}_{c}(\theta,\nu)$: the set of open balls centered on $(\theta,\nu)$ of radius $k^{-1}$,
$k\ge1$, which means the infimum is also the limit of a
decreasing sequence, which can be splitted into three terms. The second term converges deterministically to zero, while convergences \eqref{simultaneous_limit_1} and \eqref{simultaneous_limit_2} apply for the first and third terms.

To conclude, apply the monotone convergence
theorem to the remaining term:
\begin{align*}
\underset{B\in\mathcal{A}_{c}(\theta,\nu)}{\text{sup}}\ \mathbb{E}_{\psi}\bigg[\underset{(\phi,\mu)\in B}{\text{inf}}\ \Big\{ e^{\mu}\frac{h_{\phi}(X)}{h_{\psi}(X)}\Big\}\bigg] & =\mathbb{E}_{\psi}\bigg[\underset{B\in\mathcal{A}_{c}(\theta,\nu)}{\text{sup}}\ \underset{(\phi,\mu)\in B}{\text{inf}}\ \Big\{ e^{\mu}\frac{h_{\phi}(X)}{h_{\psi}(X)}\Big\}\bigg]\\
 & =\mathbb{E}_{\psi}\bigg[e^{\nu}\frac{h_{\theta}(X)}{h_{\psi}(X)}\bigg]=e^{\nu}\frac{\mathcal{Z}(\theta)}{\mathcal{Z}(\psi)}.
\end{align*}

\subsection{Proof of Theorem \ref{thm:equivalence_n_fixed}}

Define 
$g_{\xi}(x)=\log h_{\theta}(x)+\nu$, and the following gradients (dropping $n$ and $m$ in the
notation for convenience):
\begin{multline*}
\glNCE(\xi)=\nabla\lNCE(\xi)=\frac{1}{n}\sum_{i=1}^{n}\nabla_{\xi}g_{\xi}(y_{i})\bigg(\frac{mh_{\psi}(y_{i})}{mh_{\psi}(X_{j})+n\exp\{g_{\xi}(y_{i})\}}\bigg)\\
-\frac{1}{m}\sum_{j=1}^{m}\nabla_{\xi}g_{\xi}(X_{j})\bigg(\frac{m\exp\{g_{\xi}(X_{j})\}}{mh_{\psi}(X_{j})+n\exp\{g_{\xi}(X_{j})\}}\bigg),
\end{multline*}
\[
\glIS(\xi)=\nabla\lIS(\xi)=\frac{1}{n}\sum_{i=1}^{n}\nabla_{\xi}g_{\xi}(y_{i})-\frac{1}{m}\sum_{j=1}^{m}\nabla_{\xi}g_{\xi}(X_{j})\bigg(\frac{\exp\{g_{\xi}(X_{j})\}}{h_{\psi}(X_{j})}\bigg).\]

By Taylor-Lagrange, for any component $k$, $1\leq k\leq d+1$, there
exists (a random variable) $\xi_{m}^{(k)}\in[\emcmle;\ence]$ such that 
\[
\glISk(\emcmle)=\glISk(\ence)+\left\{ \nabla\glISk(\xi_{m}^{(k)})\right\} ^{T}\left(\emcmle-\ence\right)
\]
where $\glISk(\xi)$ denotes the $k-$th component of $\glIS(\xi)$,
and $[\emcmle;\ence]$ denotes the line segment in $\mathbb{R}^{d+1}$
which joins $\emcmle$ and $\ence$. 

By assumption (G1), the left hand side is $\smalloP(m^{-1})$. The matrix form yields: 
\[
\smalloP\big(m^{-1}\big)=\glIS(\ence)+\mathbf{H}_{m}^{\mathrm{IS}}\left(\emcmle-\ence\right),\qquad \mathbf{H}_{m}^{\mathrm{IS}}=\left(\begin{array}{c}
\left\{ \nabla\Psi_{1}^{\mathrm{IS}}(\xi_{m}^{(1)})\right\} ^{T}\\
\vdots\\
\left\{ \nabla\Psi_{d+1}^{\mathrm{IS}}(\xi_{m}^{(d+1)})\right\} ^{T}
\end{array}\right).
\]
Let us prove first the convergence of the Hessian matrix. Lemma \ref{lem:ULLN_theorem} can be applied to each row component
of the following matrix-valued function, the uniform norm of which
is $\refdist$-integrable under (H2): 
\[
\varphi_h:(\xi,x)\mapsto\bigg(\frac{1}{n}\sum_{i=1}^{n}\nabla_{\xi}^{2}g_{\xi}(y_{i})\bigg)-\bigg(\nabla_{\xi}^{2}g_{\xi}(x)+\nabla_{\xi}g_{\xi}(x)\left\{ \nabla_{\xi}g_{\xi}(x)\right\} {}^{T}\bigg)\bigg(\frac{\exp\{g_{\xi}(x)\}}{h_{\psi}(x)}\bigg).
\]
Convergences of the $d+1$ rows of
$\mathbf{H}_{m}^{\mathrm{IS}}$ can be combined to get the following result: 
\[
\left\Vert \mathbf{H}_{m}^{\mathrm{IS}}-\mathcal{H}(\emle)\right\Vert \cvm0\qquad\as
\]
where
\[
\mathcal{H}(\xi)=\mathbb{E}_{\psi}\big[\varphi_h(\xi,X)\big]=\nabla_{\xi}^{2}\ell_{n}(\xi).
\]

It turns out that $\mathcal{H}(\emle)$ is invertible as soon as (H2) holds.
This is the point of the following lemma. This implies in particular that $\mathbf{H}_{m}^{\mathrm{IS}}$ is eventually invertible with probability one.

\begin{lem}  Assume (H2) holds. At the point $\xi=\emle$, the Hessian matrix of the \textit{Poisson Transform} $\nabla_{\xi}^2l_n(\xi)$ is negative definite if and only if the Hessian of the log-likelihood $\nabla_{\theta}^2l_n(\theta)$ is definite negative.\label{lem:full_rank_hessian}\end{lem}

The proof of Lemma \ref{lem:full_rank_hessian} follows from a direct block
matrix computation (using Schur's complement). For the sake of completeness, we
present a proof in the supplement.

Now, let us prove the convergence of the gradient. By assumption (G1),
we can write $\glIS(\ence)=\Delta_m+o\big(m^{-1}\big)$,
where:
\begin{align*}
\Delta_m=&  \ \glIS(\ence)-\glNCE(\ence)\\=& \Bigg\{\frac{1}{n}\sum_{i=1}^{n}\nabla_{\xi}g_{\xi}(y_{i})\Big(\frac{n\exp\{g_{\xi}(y_{i})\}}{mh_{\psi}(y_{i})+n\exp\{g_{\xi}(y_{i})\}}\Big)\\
&- \frac{1}{m}\sum_{j=1}^{m}\nabla_{\xi}g_{\xi}(X_{j})\Big(\frac{\exp\{g_{\xi}(X_{j})\}}{h_{\psi}(X_{j})}\Big)\Big(\frac{n\exp\{g_{\xi}(X_{j})\}}{mh_{\psi}(X_{j})+n\exp\{g_{\xi}(X_{j})\}}\Big)\Bigg\}_{\big|_{\xi=\ence}}
\end{align*}
hence 
\begin{align*}
\frac m n \Delta_m= & \Bigg\{\frac{1}{n}\sum_{i=1}^{n}\nabla_{\xi}g_{\xi}(y_{i})\Big(\frac{\exp\{g_{\xi}(y_{i})\}}{h_{\psi}(y_{i})}\Big)\Big(1-\frac{n\exp\{g_{\xi}(y_{i})\}}{mh_{\psi}(y_{i})+n\exp\{g_{\xi}(y_{i})\}}\Big)\\
&-  \frac{1}{m}\sum_{j=1}^{m}\nabla_{\xi}g_{\xi}(X_{j})\Big(\frac{\exp\{g_{\xi}(X_{j})\}}{h_{\psi}(X_{j})}\Big)^{2}\Big(1-\frac{n\exp\{g_{\xi}(X_{j})\}}{mh_{\psi}(X_{j})+n\exp\{g_{\xi}(X_{j})\}}\Big)\Bigg\}_{\big|_{\xi=\ence}}\\
= & \Bigg\{\frac{1}{n}\sum_{i=1}^{n}\nabla_{\xi}g_{\xi}(y_{i})\Big(\frac{\exp\{g_{\xi}(y_{i})\}}{h_{\psi}(y_{i})}\Big)-\frac{1}{m}\sum_{j=1}^{m}\nabla_{\xi}g_{\xi}(X_{j})\Big(\frac{\exp\{g_{\xi}(X_{j})\}}{h_{\psi}(X_{j})}\Big)^{2}\\
&-  \frac{1}{n}\sum_{i=1}^{n}\nabla_{\xi}g_{\xi}(y_{i})\Big(\frac{\exp\{g_{\xi}(y_{i})\}}{h_{\psi}(y_{i})}\Big)\Big(\frac{n\exp\{g_{\xi}(y_{i})\}}{mh_{\psi}(y_{i})+n\exp\{g_{\xi}(y_{i})\}}\Big)\\
&+  \frac{1}{m}\sum_{j=1}^{m}\nabla_{\xi}g_{\xi}(X_{j})\Big(\frac{\exp\{g_{\xi}(X_{j})\}}{h_{\psi}(X_{j})}\Big)^{2}\Big(\frac{n\exp\{g_{\xi}(X_{j})\}}{mh_{\psi}(X_{j})+n\exp\{g_{\xi}(X_{j})\}}\Big)\Bigg\}_{\big|_{\xi=\ence}}.
\end{align*}

The two last terms of the right hand side are residuals for which
we want to bound the uniform norm over the ball $B(\mle,\varepsilon)$. The sup norm of the second term is eventually bounded by:
\[\frac{1}{m}\underset{\xi\in B(\emle,\varepsilon)}{\text{sup}}\ \sum_{i=1}^{n}\|\nabla_{\xi}g_{\xi}(y_{i})\|\left(\frac{\exp\{g_{\xi}(y_{i})\}}{h_{\psi}(y_{i})}\right)^{2}\cvm 0.\]

The sup norm of the third term is eventually bounded by
$\frac{1}{m}\sum_{j=1}^{m}f_{m}(X_{j})$ where \[f_{m}(x)=\underset{\xi\in B(\emle,\varepsilon)}{\text{sup}}\ \|\nabla_{\xi}g_{\xi}(x)\|\left(\frac{\exp\{g_{\xi}(x)\}}{h_{\psi}(x)}\right)^{2}\left(\frac{n\exp\{g_{\xi}(x)\}}{mh_{\psi}(x)+n\exp\{g_{\xi}(x)\}}\right)\]
and Lemma \ref{lem:LLN_dominated} applies under (I1) to the sequence $(f_{m})_{m\ge1}$
converging pointwise towards 0, and dominated by the integrable function $g(x)=\underset{\xi\in B(\emle,\varepsilon)}{\text{sup}}\ \|\nabla_{\xi}g_{\xi}(x)\|\big(\frac{\exp\{g_{\xi}(x)\}}{h_{\psi}(x)}\big)^{2}.$

The limit of $(m/n)\Delta_m$ is thus dictated by the behaviour of the first term. We apply Lemma \ref{lem:ULLN_theorem}
to the following vector-valued function, whose uniform norm is integrable
under (I1) and under the continuity of the deterministic part assumed in (H2): 
\[
\varphi_g:(\xi,x)\mapsto\left(\frac{1}{n}\sum_{i=1}^{n}\nabla_{\xi}g_{\xi}(y_{i})\frac{\exp\{g_{\xi}(y_{i})\}}{h_{\psi}(y_{i})}\right)-\nabla_{\xi}g_{\xi}(x)\left(\frac{\exp\{g_{\xi}(x)\}}{h_{\psi}(x)}\right)^{2}.
\]

Lemma \ref{lem:ULLN_theorem} yields $(m/n)\Delta_m\underset{m\rightarrow+\infty}{\longrightarrow}v(\emle)$
a.s. where 
\[
v(\xi)=\frac{1}{n}\sum_{i=1}^{n}\nabla_{\xi}g_{\xi}(y_{i})\left(\frac{\exp\{g_{\xi}(y_{i})\}}{h_{\psi}(y_{i})}\right)-\mathbb{E}_{\psi}\left[\nabla_{\xi}g_{\xi}(X)\left(\frac{\exp\{g_{\xi}(X)\}}{h_{\psi}(X)}\right)^{2}\right].
\]

Combination of these facts ensure that on a set of probability one,
we have eventually: 
\[
\frac{m}{n}\ \left(\emcmle-\ence\right)  =o(1)+\left(-\mathbf{H}_{m}^{\mathrm{IS}}\right)^{-1}\left(\frac{m}{n}\Delta_m+o(1)\right)
 \cvm\Big(-\mathcal{H}(\emle)\Big)^{-1}v(\emle).
\]

\subsection{Proof of Theorem \ref{thm:consistency}}

The proof of MC-MLE consistency under the considered regime is a very straightforward adaptation of Wald's proof of consistency for the MLE. We thus choose to present in appendix only the proof of NCE consistency, which is slightly more technical, although the sketch is similar. For the sake of completeness, a proof of MC-MLE consistency is presented in the supplement.

\subsubsection{NCE consistency}

For convenience, we choose to analyse a slightly different objective function  (sharing the same maximiser with $\lNCE$), defined as:
\begin{equation}
M_{n}^{\mathrm{NCE}}(\theta,\nu)=\frac{1}{n}\sum_{i=1}^n\Big\{\varphi_{(\theta,\nu)}(Y_i)-\zeta_{(\theta,\nu)}^{(n)}(Y_i)\Big\}-\Big(\frac{m_n}{n}\Big)\times\frac{1}{m_n}\sum_{j=1}^{m_n}\zeta_{(\theta,\nu)}^{(n)}(X_j)\label{eq:NCE_objective}
\end{equation}
where $\varphi_{(\theta,\nu)}(x)=\log\left\{\frac{e^{\nu}h_{\theta}(x)}{e^{\nu^{\star}}h_{\theta^{\star}}(x)}\right\}$ and $\zeta_{(\theta,\nu)}^{(n)}(x)=\log\left\{\frac{\frac{m_n}{n} h_{\psi}(x)+e^{\nu}h_{\theta}(x)}{\frac{m_n}{n} h_{\psi}(x)+e^{\nu^{\star}}h_{\theta^{\star}}(x)}\right\}$.

We begin our proof with the following lemma.

\begin{lem} For any fixed $(\theta,\nu)$, almost surely,
$M_{n}^{\mathrm{NCE}}(\theta,\nu)\cvn\mathcal{M}_{\tau}^{\mathrm{NCE}}(\theta,\nu)$, where:
$$    \mathcal{M}_{\tau}^{\mathrm{NCE}}(\theta,\nu)=\mathbb{E}_{\theta^{\star}}\Bigg[\log\bigg\{\frac{e^{\nu}h_{\theta}}{e^{\nu^{\star}}h_{\theta^{\star}}}\bigg\}-\log\bigg\{\frac{\tau h_{\psi}+e^{\nu}h_{\theta}}{\tau h_{\psi}+e^{\nu^{\star}}h_{\theta^{\star}}}\bigg\}\Bigg]
    -\tau \mathbb{E}_{\psi}\Bigg[\log\bigg\{\frac{\tau h_{\psi}+e^{\nu}h_{\theta}}{\tau h_{\psi}+e^{\nu^{\star}}h_{\theta^{\star}}}\bigg\}\Bigg]$$
Moreover, $(\theta^{\star},\nu^{\star})$ is the unique maximiser of $\mathcal{M}_{\tau}^{\mathrm{NCE}}(\theta,\nu)$.\label{lem:pointwise_cv_nce}\end{lem}

\begin{proof}

For any fixed $(\theta,\nu)$, the sequence $\zeta_{(\theta,\nu)}^{(n)}$ is eventually dominated (by a $\refdist$-integrable function), since for any $c>0$ (in particular for $c=\tau\pm\varepsilon$) we have by Jensen's inequality:
\begin{equation}
\mathbb{E}_{\psi}\Bigg[\log\bigg\{\frac{c h_{\psi}+e^{\nu}h_{\theta}}{c h_{\psi}+e^{\nu^{\star}}h_{\theta^{\star}}}\bigg\}\Bigg]\ge\mathbb{E}_{\psi}\Bigg[\log\bigg\{\frac{f_{\psi}}{f_{\psi}+\frac{1}{c}f_{\theta^{\star}}}\bigg\}\Bigg]\ge-\log\Big(1+\frac{1}{c}\Big)\label{eq:bla}
\end{equation}
\begin{equation}
\mathbb{E}_{\psi}\Bigg[\log\bigg\{\frac{c h_{\psi}+e^{\nu}h_{\theta}}{c h_{\psi}+e^{\nu^{\star}}h_{\theta^{\star}}}\bigg\}\Bigg]\le\mathbb{E}_{\psi}\Bigg[\log\bigg\{\frac{f_{\psi}+\frac{e^{\nu}\mathcal{Z}(\theta)}{c\mathcal{Z}(\psi)}f_{\theta}}{f_{\psi}}\bigg\}\Bigg]\le\log\Big(1+\frac{e^{\nu}\mathcal{Z}(\theta)}{c\mathcal{Z}(\psi)}\Big)\label{eq:blo}
\end{equation}

Moreover, $\zeta_{(\theta,\nu)}^{(n)}$ converges pointwise to $\zeta_{(\theta,\nu)}^{\infty}(x)=\log\left\{\frac{\tau h_{\psi}(x)+e^{\nu}h_{\theta}(x)}{\tau h_{\psi}(x)+e^{\nu^{\star}}h_{\theta^{\star}}(x)}\right\}$, thus Lemma \ref{lem:LLN_dominated} applies: the second empirical average in \eqref{eq:NCE_objective} converges almost surely to $\mathbb{E}_{\psi}\big[\zeta_{(\theta,\nu)}^{\infty}(X)\big]$.
  
Now, the sequence $\big\{\varphi_{(\theta,\nu)}-\zeta_{(\theta,\nu)}^{(n)}\big\}$ is upper bounded by the positive part of $\varphi_{(\theta,\nu)}$ which is $\mathbb{P}_{\theta^{\star}}$-integrable. 
 In particular, if $\mathbb{E}_{\theta^{\star}}\big[\big(\varphi_{(\theta,\nu)}-\zeta_{(\theta,\nu)}^{\infty}\big)_{-}\big]=+\infty$, then Lemma \ref{lem:LLN_infinite} applies and the first empirical average in \eqref{eq:NCE_objective} converges towards $-\infty$.
 
Conversely, suppose that $\mathbb{E}_{\theta^{\star}}\big[\big(\varphi_{(\theta,\nu)}-\zeta_{(\theta,\nu)}^{\infty}\big)_{-}\big]<\infty$. The law of large numbers would apply directly if the sequence $m_n/n$ was exactly equal to $\tau$. To handle this technical issue, we can consider the two following inequalities. Note that for any $a\ge b>0$:
$$\log\bigg\{\frac{a h_{\psi}(x)+e^{\nu}h_{\theta}(x)}{a h_{\psi}(x)+e^{\nu^{\star}}h_{\theta^{\star}}(x)}\bigg\}\le\log\bigg\{\frac{\frac{a}{b} bh_{\psi}(x)+\frac{a}{b}e^{\nu}h_{\theta}(x)}{b h_{\psi}(x)+e^{\nu^{\star}}h_{\theta^{\star}}(x)}\bigg\}$$
$$\log\bigg\{\frac{bh_{\psi}(x)+e^{\nu}h_{\theta}(x)}{b h_{\psi}(x)+e^{\nu^{\star}}h_{\theta^{\star}}(x)}\bigg\}\le\log\bigg\{\frac{ah_{\psi}(x)+e^{\nu}h_{\theta}(x)}{\frac{b}{a} ah_{\psi}(x)+\frac{b}{a}e^{\nu^{\star}}h_{\theta^{\star}}(x)}\bigg\}$$

This yields a useful uniform bound for any $a,b>0$:
\begin{equation}
\bigg|\log\bigg\{\frac{a h_{\psi}(x)+e^{\nu}h_{\theta}(x)}{a h_{\psi}(x)+e^{\nu^{\star}}h_{\theta^{\star}}(x)}\bigg\}-\log\bigg\{\frac{b h_{\psi}(x)+e^{\nu}h_{\theta}(x)}{b h_{\psi}(x)+e^{\nu^{\star}}h_{\theta^{\star}}(x)}\bigg\}\bigg|\le\Big|\log a-\log b\Big|\label{uniform_bound}
\end{equation}

Thus, if $\mathbb{E}_{\theta^{\star}}\Big[\big(\varphi_{(\theta,\nu)}-\zeta_{(\theta,\nu)}^{\infty}\big)_{-}\Big]<+\infty$, then the uniform bound \eqref{uniform_bound} also ensures that: 
 $$\mathbb{E}_{\theta^{\star}}\Bigg[\bigg(\varphi_{(\theta,\nu)}-\log\bigg\{\frac{c h_{\psi}+e^{\nu}h_{\theta}}{c h_{\psi}+e^{\nu^{\star}}h_{\theta^{\star}}}\bigg\}\bigg)_{-}\Bigg]<+\infty$$
 for any positive $c>0$. The sequence can now be easily dominated and Lemma \ref{lem:LLN_dominated} applies; the first empirical average in \eqref{eq:NCE_objective} converges to $\mathbb{E}_{\theta^{\star}}\big[\varphi_{(\theta,\nu)}(Y)-\zeta_{(\theta,\nu)}^{\infty}(Y)\big]$.

Finally, let us prove that $(\theta^{\star},\nu^{\star})$ is the unique maximiser of $\mathcal{M}_{\tau}^{\mathrm{NCE}}$.
We have:
\begin{align*}
    \mathcal{M}_{\tau}^{\mathrm{NCE}}(\theta,\nu)&=\frac{1}{\mathcal{Z(\psi)}}\Bigg[\int_{\mathcal{X}}-\log\bigg\{\frac{e^{\nu^{\star}}h_{\theta^{\star}}(x)}{e^{\nu}h_{\theta}(x)}\bigg\}e^{\nu^{\star}}h_{\theta^{\star}}(x)\\
    &\hspace{1.6cm}+\log\bigg\{\frac{\tau h_{\psi}(x)+e^{\nu^{\star}}h_{\theta^{\star}}(x)}{\tau h_{\psi}(x)+e^{\nu}h_{\theta}(x)}\bigg\}\big(\tau h_{\psi}(x)+e^{\nu^{\star}}h_{\theta^{\star}}(x)\big)\lambda(\dx)\Bigg]\\
    &\le\frac{1}{\mathcal{Z(\psi)}}\Bigg[\int_{\mathcal{X}}-\log\bigg\{\frac{e^{\nu^{\star}}h_{\theta^{\star}}(x)}{e^{\nu}h_{\theta}(x)}\bigg\}e^{\nu^{\star}}h_{\theta^{\star}}(x)\\
    &\hspace{1.6cm}+\log\bigg\{\frac{\tau h_{\psi}(x)}{\tau h_{\psi}(x)}\bigg\} \tau h_{\psi}(x)+\log\bigg\{\frac{e^{\nu^{\star}}h_{\theta^{\star}}(x)}{e^{\nu}h_{\theta}(x)}\bigg\}e^{\nu^{\star}}h_{\theta^{\star}}(x)\lambda(\dx)\Bigg]\\
    &=0
\end{align*}
by the log-sum inequality, which applies with equality if and only if $e^{\nu}h_{\theta}(x)=e^{\nu^{\star}}h_{\theta^{\star}}(x)$ for $\mathbb{P_{\theta^{\star}}}$ almost every $x$.
This occurs if and only if $\nu$ and $\theta$ are chosen such that
$f_{\theta^{\star}}(x)=\frac{e^{\nu}}{\mathcal{Z(\psi)}}h_{\theta}(x)$. The
model being identifiable, there is only one choice for both the unnormalized
density and the normalizing constant; $\theta=\theta^{\star}$ and
$\nu=\nu^{\star}$.

\end{proof}

We now prove that the NCE estimator converges almost surely to this unique maximiser. Let $\eta>0$, and define $K_{\eta}=\{\xi\in K : d(\xi,\xi^{\star})\ge\eta\}$ where $K$ is the compact set defined in (C2).

Under (H3), monotone convergence ensures that for any $\xi\in K_{\eta}$: 
$$\lim_{\varepsilon\downarrow0}\ \mathbb{E}_{\theta^{\star}}\bigg[\underset{\beta\in B(\xi,\varepsilon)}{\sup}\Big(\varphi_{\beta}(Y)-\zeta_{\beta}^{\infty}(Y)\Big)\bigg]=\mathbb{E}_{\theta^{\star}}\Big[\varphi_{\xi}(Y)-\zeta_{\xi}^{\infty}(Y)\Big]$$
and
$$\lim_{\varepsilon\downarrow0}\ \mathbb{E}_{\psi}\bigg[\underset{\beta\in B(\xi,\varepsilon)}{\inf}\zeta_{\beta}^{\infty}(X)\bigg]= \mathbb{E}_{\psi}\Big[\zeta_{\xi}^{\infty}(X)\Big].$$
Indeed, since maps $\theta\mapsto h_{\theta}(x)$ are continuous, the two previous expectations (on the left hand side) are respectively bounded from above for $\varepsilon$ small enough, and bounded from below for any $\varepsilon$. 

 Thus, for any $\xi\in K_{\eta}$ and any $\gamma>0$ we can find $\varepsilon_{\xi}>0$ such that simultaneously:
 $$\mathbb{E}_{\theta^{\star}}\bigg[\underset{\beta\in
 B(\xi,\varepsilon_{\xi})}{\sup}\Big(\varphi_{\beta}(Y)-\zeta_{\beta}^{\infty}(Y)\Big)\bigg]\le\mathbb{E}_{\theta^{\star}}\Big[\varphi_{\xi}(Y)-\zeta_{\xi}^{\infty}(Y)\Big]+\frac{\gamma}{2}$$
 $$\mathbb{E}_{\psi}\bigg[\underset{\beta\in
 B(\xi,\varepsilon_{\xi})}{\inf}\zeta_{\beta}^{\infty}(X)\bigg]\ge
 \mathbb{E}_{\psi}\Big[\zeta_{\xi}^{\infty}(X)\Big]-\frac{\gamma}{2\tau}.$$
The compactness assumption ensures that there is a finite set $\{\xi_1,...,\xi_p\}\subset K_{\eta}$ such that $K_{\eta}\subset\bigcup_{k=1}^pB(\xi_k,\varepsilon_{\xi_k})$. This yields the following inequality:
\begin{multline*}
    \underset{\xi\in K_{\eta}}{\sup} M_{n}^{\mathrm{NCE}}(\xi)\le\underset{k=1,...,p}{\max}\Bigg\{\frac{1}{n}\sum_{i=1}^n\underset{q\ge n}{\sup}\ \underset{\beta\in B(\xi_k,\varepsilon_{\xi_k})}{\sup}\  \Big(\varphi_{\beta}(Y_i)-\zeta_{\beta}^{(q)}(Y_i)\Big)\\
    -\Big(\frac{m_n}{n}\Big)\times\frac{1}{m_n}\sum_{j=1}^{m_n}\underset{\beta\in B(\xi_k,\varepsilon_{\xi_k})}{\inf}\  \zeta_{\beta}^{(n)}(X_j)\Bigg\}
\end{multline*}
Choose any $x$ for which the map $\theta \mapsto h_{\theta}(x)$ is continuous, and any $\xi\in K_{\eta}$. From the definition of $\zeta_{\beta}^{(n)}$, the following convergence is trivial:
$$\underset{\beta\in B(\xi,\varepsilon_{\xi})}{\inf}\  \Big(\zeta_{\beta}^{(n)}(x)\Big)\underset{n\rightarrow+\infty}{\longrightarrow}\underset{\beta\in B(\xi,\varepsilon_{\xi})}{\inf}\ \Big( \zeta_{\beta}^{\infty}(x)\Big).$$
Moreover, using inequalities \eqref{eq:bla} et \eqref{eq:blo}, one can easily show that the sequence $\Big\{\underset{\beta\in B(\xi,\varepsilon_{\xi})}{\inf}\  \zeta_{\beta}^{(n)}\Big\}$ is dominated (by a $\refdist$-integrable function). Lemma \ref{lem:LLN_dominated} applies: $$\frac{1}{m_n}\sum_{j=1}^{m_n}\underset{\beta\in B(\xi_k,\varepsilon_{\xi_k})}{\inf}\  \zeta_{\beta}^{(n)}(X_j)\underset{n\rightarrow+\infty}{\longrightarrow}\mathbb{E}_{\psi}\left[\underset{\beta\in B(\xi,\varepsilon_{\xi})}{\inf}\zeta_{\beta}^{\infty}(X)\right]\qquad\as$$

Now, subadditivity of the supremum and inequality \eqref{uniform_bound} yield
\begin{multline*}
\bigg|\underset{\beta\in B(\xi,\varepsilon_{\xi})}{\sup} \Big(\varphi_{\beta}(x)-\zeta_{\beta}^{(n)}(x)\Big)-\underset{\beta\in B(\xi,\varepsilon_{\xi})}{\sup} \Big(\varphi_{\beta}(x)-\zeta_{\beta}^{(\infty)}(x)\Big)\bigg|\\
\le\underset{\beta\in B(\xi,\varepsilon_{\xi})}{\sup}\Big| \zeta_{\beta}^{(n)}(x)-\zeta_{\beta}^{\infty}(x)\Big|\le\Big|\log\frac{m_n}{n}-\log\tau\Big|\underset{n\rightarrow+\infty}{\longrightarrow}0
\end{multline*}
while monotonicity ensures that
 $$\underset{\beta\in B(\xi,\varepsilon_{\xi})}{\sup}\Big(\varphi_{\beta}-\zeta_{\beta}^{\infty}\Big)\le\underset{q\ge n}{\sup}\ \underset{\beta\in B(\xi,\varepsilon_{\xi})}{\sup}  \Big(\varphi_{\beta}-\zeta_{\beta}^{(q)}\Big)\le\underset{\beta\in B(\xi,\varepsilon_{\xi})}{\sup}  \Big(\varphi_{\beta}\Big)_{+}.$$
 In the last inequality, the right hand side is $\mathbb{P}_{\theta^{\star}}$-integrable under (H3), and the sequence (in the middle) converges pointwise towards its lower bound whose negative part has either finite or infinite expectation. In both cases, either Lemma \ref{lem:LLN_dominated} or Lemma \ref{lem:LLN_infinite} can be applied and ensures that, almost surely:
$$\frac{1}{n}\sum_{i=1}^n\underset{q\ge n}{\sup}\ \underset{\beta\in B(\xi_k,\varepsilon_{\xi_k})}{\sup}  \Big(\varphi_{\beta}(Y_i)-\zeta_{\beta}^{(q)}(Y_i)\Big)\underset{n\rightarrow+\infty}{\longrightarrow}\mathbb{E}_{\theta^{\star}}\bigg[\underset{\beta\in B(\xi,\varepsilon_{\xi})}{\sup}\Big(\varphi_{\beta}(Y)-\zeta_{\beta}^{\infty}(Y)\Big)\bigg]$$
Combining these convergences simultaneously on a finite set, we get almost surely:
\begin{align*}
    \underset{n\rightarrow+\infty}{\lim \sup} \underset{\xi\in K_{\eta}}{\sup} M_{n}^{\mathrm{NCE}}(\xi)\le&\underset{k=1,...,p}{\max}\Bigg\{\mathbb{E}_{\theta^{\star}}\bigg[\underset{\beta\in B(\xi_k,\varepsilon_{\xi_k})}{\sup}\Big(\varphi_{\beta}(Y)-\zeta_{\beta}^{\infty}(Y)\Big)\bigg]\\
    &-\tau\mathbb{E}_{\psi}\bigg[\underset{\beta\in B(\xi_k,\varepsilon_{\xi_k})}{\inf}\zeta_{\beta}^{\infty}(X)\bigg]\Bigg\}\\
    \le& \underset{\xi\in K_{\eta}}{\sup} \mathcal{M}_{\tau}^{\mathrm{NCE}}(\xi)+\gamma
\end{align*}
This leads to the following inequality since $\gamma$ is arbitrary small:
\begin{equation}
\underset{n\rightarrow+\infty}{\lim \sup} \underset{\xi\in K_{\eta}}{\sup} M_{n}^{\mathrm{NCE}}(\xi)\le\underset{\xi\in K_{\eta}}{\sup} \mathcal{M}_{\tau}^{\mathrm{NCE}}(\xi)\qquad\as\label{eq:heart_NCE}
\end{equation}

This last inequality is the heart of the proof. To conclude, we need only to show that the right hand side is negative, this is the aim of the following lemma.

\begin{lem}
 Under (H3), the map $\xi\mapsto\mathcal{M}_{\tau}^{\mathrm{NCE}}(\xi)$ is upper semi continuous.\label{lem:usc_nce}
\end{lem} 

The proof of Lemma \ref{lem:usc_nce} is straightforward. For the sake of completeness, we present a proof in the supplement.

Since an upper semi continuous function achieves its maximum on any compact set, this lemma proves in particular that $\underset{\xi\in K_{\eta}}{\sup} \mathcal{M}_{\tau}^{\mathrm{NCE}}(\xi)<0$.

Thus inequality \eqref{eq:heart_NCE} implies that we can always
find some $\alpha<0$ such that eventually $\underset{\xi\in K_{\eta}}{\sup}M_{n}^{\mathrm{NCE}}(\xi)<\alpha$, while (C2) implies that $M_{n}^{\mathrm{NCE}}(\emcmle)\ge\underset{\xi\in\Xi}{\sup}\ M_{n}^{\mathrm{NCE}}(\xi)-\delta_{n}$
where $\delta_{n}\rightarrow0$, and where
\[
\underset{\xi\in\Xi}{\sup}\ M_{n}^{\mathrm{NCE}}(\xi)\ge M_{n}^{\mathrm{NCE}}(\truee)\overset{a.s.}{\underset{n\rightarrow+\infty}{\longrightarrow}}\mathcal{M}^{\mathrm{NCE}}(\truee)=0.
\]
Combination of these facts show that with probability one we have eventually: 
\[
M_{n}^{\mathrm{NCE}}(\emcmle)>\alpha>\underset{\xi\in K_{\eta}}{\sup}\ M_{n}^{\mathrm{NCE}}(\xi).
\]
This is enough to prove strong consistency. Indeed, with probability
one, $\ence$ eventually escapes from $K_{\eta}$ (otherwise there would be a contradiction with the inequality above). Since the sequence belongs to $K$ by assumption, the sequence has no choice but to stay eventually in the ball of radius $\eta$. Thus with probability one, for any $\eta>0$, we have eventually $d(\ence,\truee)<\eta$. This is the definition
of almost sure convergence.

\subsection{Proof of Theorem 4}

The proof of MC-MLE asymptotic normality is entirely classical. We choose to present in appendix only the proof of NCE asymptotic normality, which follows the same sketch but is slightly more technical. For the sake of completeness, a proof of MC-MLE asymptotic normality is presented in the supplement.

 \subsubsection{NCE asymptotic normality}

Let $G_n^{\mathrm{NCE}}(\xi)=\nabla_{\xi}\lNCE(\xi)$ and $\mathbf{H}_n^{\mathrm{NCE}}(\xi)=\nabla_{\xi}^2\lNCE(\xi)$. We have:
\begin{align*}
    G_n^{\mathrm{NCE}}(\xi)&=\frac{1}{n}\sum_{i=1}^{n}\nabla_{\xi}g_{\xi}(Y_i)\bigg(\frac{m_n h_{\psi}}{m_n h_{\psi}+n\exp\{g_{\xi}\}}\bigg)(Y_i)\\
&-\frac{1}{m_n}\sum_{j=1}^{m_n}\nabla_{\xi}g_{\xi}(X_j)\frac{\exp\{g_{\xi}(X_j)\}}{h_{\psi}(X_j)}\bigg(\frac{m_n h_{\psi}}{m_n h_{\psi}+n\exp\{g_{\xi}\}}\bigg)(X_j)\end{align*}
\begin{align*}
\mathbf{H}_n^{\mathrm{NCE}}(\xi)&=\frac{1}{n}\sum_{i=1}^n\nabla_{\xi}^2g_{\xi}(Y_i)\bigg(\frac{m_n h_{\psi}}{m_n h_{\psi}+n \exp\{g_{\xi}\}}\bigg)(Y_i)\\
&-\frac{1}{n}\sum_{i=1}^n\nabla_{\xi}\nabla_{\xi}^Tg_{\xi}(Y_i)\bigg(\frac{m_n h_{\psi}n \exp\{g_{\xi}\}}{(m_n h_{\psi}+n \exp\{g_{\xi}\})^2}\bigg)(Y_i)\\
&-\frac{1}{m_n}\sum_{j=1}^{m_n}\Big\{(\nabla_{\xi}^2+\nabla_{\xi}\nabla_{\xi}^T)g_{\xi}(X_j)\Big\}\frac{\exp\{g_{\xi}(X_j)\}}{h_{\psi}(X_j)}\bigg(\frac{m_n h_{\psi}}{m_n h_{\psi}+n \exp\{g_{\xi}\}}\bigg)(X_j)\\
&+\frac{1}{m_n}\sum_{j=1}^{m_n}\nabla_{\xi}\nabla_{\xi}^Tg_{\xi}(X_j)\frac{\exp\{g_{\xi}(X_j)\}}{h_{\psi}(X_j)}\bigg(\frac{m_n h_{\psi}n \exp\{g_{\xi}\}}{(m_n h_{\psi}+n \exp\{g_{\xi}\})^2}\bigg)(X_j)
\end{align*}
We firstly show that the study can be reduced to the following random sequences:
\begin{align*}
    G_n^{\tau}(\xi)&=\frac{1}{n}\sum_{i=1}^{n}\nabla_{\xi}g_{\xi}(Y_i)\bigg(\frac{\tau h_{\psi}}{\tau h_{\psi}+\exp\{g_{\xi}\}}\bigg)(Y_i)\\
&-\frac{1}{m_n}\sum_{j=1}^{m_n}\nabla_{\xi}g_{\xi}(X_j)\frac{\exp\{g_{\xi}(X_j)\}}{h_{\psi}(X_j)}\bigg(\frac{\tau h_{\psi}}{\tau h_{\psi}+\exp\{g_{\xi}\}}\bigg)(X_j)\end{align*}
\begin{align*}
\mathbf{H}_n^{\tau}(\xi)&=\frac{1}{n}\sum_{i=1}^n\nabla_{\xi}^2g_{\xi}(Y_i)\bigg(\frac{\tau h_{\psi}}{\tau h_{\psi}+ \exp\{g_{\xi}\}}\bigg)(Y_i)\\
&-\frac{1}{n}\sum_{i=1}^n\nabla_{\xi}\nabla_{\xi}^Tg_{\xi}(Y_i)\bigg(\frac{\tau h_{\psi} \exp\{g_{\xi}\}}{(\tau h_{\psi}+ \exp\{g_{\xi}\})^2}\bigg)(Y_i)\\
&-\frac{1}{m_n}\sum_{j=1}^{m_n}\Big\{(\nabla_{\xi}^2+\nabla_{\xi}\nabla_{\xi}^T)g_{\xi}(X_j)\Big\}\frac{\exp\{g_{\xi}(X_j)\}}{h_{\psi}(X_j)}\bigg(\frac{\tau h_{\psi}}{\tau h_{\psi}+ \exp\{g_{\xi}\}}\bigg)(X_j)\\
&+\frac{1}{m_n}\sum_{j=1}^{m_n}\nabla_{\xi}\nabla_{\xi}^Tg_{\xi}(X_j)\frac{\exp\{g_{\xi}(X_j)\}}{h_{\psi}(X_j)}\bigg(\frac{\tau h_{\psi} \exp\{g_{\xi}\}}{(\tau h_{\psi}+ \exp\{g_{\xi}\})^2}\bigg)(X_j)
\end{align*}
To do so, we show that almost surely $\underset{\xi\in B(\xi^{\star},\varepsilon)}{\sup}\|\mathbf{H}_n^{\mathrm{NCE}}(\xi)-\mathbf{H}_n^{\tau}(\xi)\|\cvn0$.

Splitting the uniform norm into four parts yields:
\begin{align*}
\underset{\xi\in B(\xi^{\star},\varepsilon)}{\sup}\Big\|\mathbf{H}_n^{\mathrm{NCE}}(\xi)&-\mathbf{H}_n^{\tau}(\xi)\Big\|\le\frac{1}{n}\sum_{i=1}^n\underset{\xi\in B(\xi^{\star},\varepsilon)}{\sup}\Big\|\nabla_{\xi}^2g_{\xi}(Y_i)\Big\|\eta_n^{\tau}(Y_i)\\
&+\frac{1}{n}\sum_{i=1}^n\underset{\xi\in B(\xi^{\star},\varepsilon)}{\sup}\Big\|\nabla_{\xi}\nabla_{\xi}^Tg_{\xi}(Y_i)\Big\|\boldsymbol{\Gamma}_n^{\tau}(Y_i)\\
&+\frac{1}{m_n}\sum_{j=1}^{m_n}\underset{\xi\in B(\xi^{\star},\varepsilon)}{\sup}\Big\|(\nabla_{\xi}^2+\nabla_{\xi}\nabla_{\xi}^T)g_{\xi}(X_j)\Big\|\frac{\exp\{g_{\xi}(X_j)\}}{h_{\psi}(X_j)}\eta_n^{\tau}(X_j)\\
&+\frac{1}{m_n}\sum_{j=1}^{m_n}\underset{\xi\in B(\xi^{\star},\varepsilon)}{\sup}\Big\|\nabla_{\xi}\nabla_{\xi}^Tg_{\xi}(X_j)\Big\|\frac{\exp\{g_{\xi}(X_j)\}}{h_{\psi}(X_j)}\gamma_n^{\tau}(X_j)\numberthis\label{eq:hess_norm_split}
\end{align*}
where the sequences of functions $$\eta_n^{\tau}=\underset{\xi\in B(\xi^{\star},\varepsilon)}{\sup}\Big|\frac{m_n h_{\psi}}{m_n h_{\psi}+n \exp\{g_{\xi}\}}-\frac{\tau h_{\psi}}{\tau h_{\psi}+ \exp\{g_{\xi}\}}\Big|$$ and $$\gamma_n^{\tau}=\underset{\xi\in B(\xi^{\star},\varepsilon)}{\sup}\Big|\frac{m_n h_{\psi}n \exp\{g_{\xi}\}}{(m_n h_{\psi}+n \exp\{g_{\xi}\})^2}-\frac{\tau h_{\psi} \exp\{g_{\xi}\}}{(\tau h_{\psi}+ \exp\{g_{\xi}\})^2}\Big|$$ 
are both upper bounded by $1$ and converge pointwise (for any $x\in\mathcal{X}$) to $0$ (use the continuity of $\xi\mapsto g_{\xi}(x)$).

Lemma \ref{lem:LLN_dominated} applies to each empirical average in \eqref{eq:hess_norm_split} (every integrability condition holds under (H4)). The sum converges to $0$ almost surely.

Now, we prove that $\forall
a\in\mathbb{R}^{d+1}\hspace{0.3cm}a^T\sqrt{n}\big(G_n^{\mathrm{NCE}}(\xi^{\star})-G_n^{\tau}(\xi^{\star})\big)\overset{\mathbb{P}}{\longrightarrow}0$.

Define $\eta_{\theta,\tau}^{(n)}=\frac{m_n f_{\psi}}{m_n f_{\psi}+nf_{\theta}}-\frac{\tau f_{\psi}}{\tau f_{\psi}+f_{\theta}}$. At the point $\xi=\xi^{\star}$ we have:
\begin{multline*}
    \sqrt{n}\big(G_n^{\mathrm{NCE}}(\xi^{\star})-G_n^{\tau}(\xi^{\star})\big)=\sqrt{n}\Bigg\{\frac{1}{n}\sum_{i=1}^{n}\big(\nabla_{\xi}g_{\xi}\big)\eta_{\theta,\tau}^{(n)}(Y_i)-\mathbb{E}_{\theta}\bigg[\big(\nabla_{\xi}g_{\xi}\big)\eta_{\theta,\tau}^{(n)}(Y)\bigg]\Bigg\}_{|_{\xi=\xi^{\star}}}\\
-\sqrt{\frac{n}{m_n}}\times\sqrt{m_n}\Bigg\{\frac{1}{m_n}\sum_{j=1}^{m_n}\big(\nabla_{\xi}g_{\xi}\big)\frac{f_{\theta}}{f_{\psi}}\eta_{\theta,\tau}^{(n)}(X_j)-\mathbb{E}_{\theta}\bigg[\big(\nabla_{\xi}g_{\xi}\big)\eta_{\theta,\tau}^{(n)}(Y)\bigg]\Bigg\}_{|_{\xi=\xi^{\star}}}
\end{multline*}
The sequence $\big|\eta_{\theta,\tau}^{(n)}\big|$ is upper bounded by $1$ and converges pointwise towards $0$. Moreover, for any $c>\tau$, the sequence $\big|\eta_{\theta,\tau}^{(n)}\big|$ is also eventually upper bounded by $2\frac{cf_{\psi}}{cf_{\psi}+f_{\theta}}$. This ensures that both second moment conditions required holds under (H4) since: 
\begin{align*}
    \int_{\setX}\|\nabla_{\xi}g_{\xi}\|^2\left(\frac{f_{\theta}}{f_{\psi}}\right)^2\left(\frac{cf_{\psi}}{cf_{\psi}+f_{\theta}}\right)^2f_{\psi}\dmu&=c\int_{\setX}\|\nabla_{\xi}g_{\xi}\|^2\left(\frac{cf_{\psi}}{cf_{\psi}+f_{\theta}}\right)\left(\frac{f_{\theta}}{cf_{\psi}+f_{\theta}}\right)f_{\theta}\dmu\\
    &\le c\times\mathbb{E}_{\theta}\left[\|\nabla_{\xi}g_{\xi}\|^2\right]<+\infty \numberthis \label{eq:second_moment}
\end{align*}
We can thus apply Lemma \ref{lem:CLT_dominated}:
\begin{align*}
    \sqrt{n}\Bigg\{\frac{1}{n}\sum_{i=1}^{n}\big(a^T\nabla_{\xi}g_{\xi}\big)\eta_{\theta,\tau}^{(n)}(Y_i)-\mathbb{E}_{\theta}\bigg[\big(a^T\nabla_{\xi}g_{\xi}\big)\eta_{\theta,\tau}^{(n)}(Y)\bigg]\Bigg\}_{|_{\xi=\xi^{\star}}}&\overset{\mathbb{P}}{\longrightarrow}0\\
    \sqrt{m_n}\Bigg\{\frac{1}{m_n}\sum_{j=1}^{m_n}\big(a^T\nabla_{\xi}g_{\xi}\big)\frac{f_{\theta}}{f_{\psi}}\eta_{\theta,\tau}^{(n)}(X_j)-\mathbb{E}_{\theta}\bigg[\big(a^T\nabla_{\xi}g_{\xi}\big)\eta_{\theta,\tau}^{(n)}(Y)\bigg]\Bigg\}_{|_{\xi=\xi^{\star}}}&\overset{\mathbb{P}}{\longrightarrow}0
\end{align*}
Finally, Cramér-Wold's device applies: $\sqrt{n}\Big(G_n^{\mathrm{NCE}}(\xi^{\star})-G_n^{\tau}(\xi^{\star})\Big)\overset{\mathbb{P}}{\longrightarrow}0_{\mathbb{R}^{d+1}}$.

Now, we can work directly with $G_n^{\tau}$ and $\mathbf{H}_n^{\tau}$, which is much easier. Indeed, Lemma \ref{lem:ULLN_theorem} yields $\underset{\xi\in B(\xi^{\star},\varepsilon)}{\sup}\|\mathbf{H}_n^{\tau}(\xi)-\mathbf{H}_{\tau}(\xi)\|\cvn0$ almost surely, where:
\begin{align*}
\mathbf{H}_{\tau}(\xi)&=\mathbb{E}_{\theta^{\star}}\bigg[\nabla_{\xi}^2g_{\xi}(Y)\bigg(\frac{\tau h_{\psi}}{\tau h_{\psi}+ \exp\{g_{\xi}\}}\bigg)(Y)\bigg]\\
&-\mathbb{E}_{\theta^{\star}}\bigg[\nabla_{\xi}\nabla_{\xi}^Tg_{\xi}(Y)\bigg(\frac{\tau h_{\psi} \exp\{g_{\xi}\}}{(\tau h_{\psi}+ \exp\{g_{\xi}\})^2}\bigg)(Y)\bigg]\\
&-\mathbb{E}_{\psi}\bigg[\Big\{(\nabla_{\xi}^2+\nabla_{\xi}\nabla_{\xi}^T)g_{\xi}(X)\Big\}\frac{\exp\{g_{\xi}(X)\}}{h_{\psi}(X)}\bigg(\frac{\tau h_{\psi}}{\tau h_{\psi}+ \exp\{g_{\xi}\}}\bigg)(X)\bigg]\\
&+\mathbb{E}_{\psi}\bigg[\nabla_{\xi}\nabla_{\xi}^Tg_{\xi}(X)\frac{\exp\{g_{\xi}(X)\}}{h_{\psi}(X)}\bigg(\frac{\tau h_{\psi} \exp\{g_{\xi}\}}{(\tau h_{\psi}+ \exp\{g_{\xi}\})^2}\bigg)(X)\bigg]
\end{align*}
The only condition required is that the supremum norm of each integrand is integrable, which is satisfied under (H4) (bound the ratios by one).

Note also that, at the point $\xi=\xi^{\star}$, functions $\mathbf{H}_{\tau}$ and $-\mathbf{J}_{\tau}$ coincide, where:
$$\mathbf{J}_{\tau}(\xi)=\mathbb{E}_{\theta}\bigg[(\nabla_{\xi}\nabla_{\xi}^Tg_{\xi})\Big(\frac{\tau f_{\psi}}{\tau f_{\psi}+f_{\theta}}\Big)(Y)\bigg]$$
A quick block matrix calculation shows that Schur's complement in $-\mathbf{J}_{\tau}(\xi)$ is proportional to:
$$\mathbf{I}_{\tau}(\theta)=\mathbb{V}_{X\sim \mathbb{Q}_{\tau}}\Big(\nabla_{\theta}\log h_{\theta}(X)\Big)$$
where $\mathbb{Q}_{\tau}$ refers to the probability measure whose density with respect to $\mu$ is defined as $q_{\tau}(x)\propto\frac{\tau f_{\psi}(x) f_{\theta}(x)}{\tau f_{\psi}(x) + f_{\theta}(x)}$. Note that $\mathbb{P}_{\theta}\ll\mathbb{Q}_{\tau}$ since the model is dominated by $\refdist$.

In particular, $\mathbf{J}_{\tau}(\truee)$ is invertible if and only if $\mathbf{I}_{\tau}(\theta^{\star})$ is invertible. Since $\mathbf{I}_{\tau}(\theta)$ is a covariance matrix, if it is not full rank, then $\nabla_{\theta}\log h_{\theta}(X)$ belongs to a hyperplane $\mathbb{Q}_{\tau}$-almost surely (and thus $\mathbb{P}_{\theta}$-almost surely). This contradicts assumption (H4) since the Fisher Information is full rank. Thus $\mathbf{I}_{\tau}(\theta^{\star})$ and $\mathbf{J}_{\tau}(\truee)$ are both invertible.

Now, we prove the weak convergence of the gradient:
\begin{align*}
    \sqrt{n}G_n^{\tau}(\xi^{\star})&=\sqrt{n}\Bigg(\frac{1}{n}\sum_{i=1}^n(\nabla_{\xi}g_{\xi})\Big(\frac{\tau f_{\psi}}{\tau f_{\psi}+f_{\theta}}\Big)(Y_i)-\mathbb{E}_{\theta}\bigg[(\nabla_{\xi}g_{\xi})\Big(\frac{\tau f_{\psi}}{\tau f_{\psi}+f_{\theta}}\Big)(Y)\bigg]\Bigg)_{|_{\xi=\xi^{\star}}}\\
    -\sqrt{\frac{n}{m_n}}&\sqrt{m_n}\Bigg(\frac{1}{m_n}\sum_{j=1}^{m_n}(\nabla_{\xi}g_{\xi})\frac{f_{\theta}}{f_{\psi}}\Big(\frac{\tau f_{\psi}}{\tau f_{\psi}+f_{\theta}}\Big)(X_j)-\mathbb{E}_{\theta}\bigg[(\nabla_{\xi}g_{\xi})\Big(\frac{\tau f_{\psi}}{\tau f_{\psi}+f_{\theta}}\Big)(Y)\bigg]\Bigg)_{|_{\xi=\xi^{\star}}}
\end{align*}
Slutky's lemma applies as follows; second moment conditions hold under (H4) (use inequality \eqref{eq:second_moment}).
$$\sqrt{n}G_n^{\tau}(\xi^{\star})\cvd \mathcal{N}\Big(0_{\mathbb{R}^{d+1}}, \boldsymbol{\boldsymbol{\Sigma}}_{\tau}(\xi^{\star})+\tau^{-1}\boldsymbol{\boldsymbol{\Gamma}}_{\tau}(\xi^{\star})\Big)$$
where \begin{align*}
    \boldsymbol{\boldsymbol{\Sigma}}_{\tau}(\xi)&=\mathbb{V}_{\theta}\left((\nabla_{\xi}g_{\xi})\Big(\frac{\tau f_{\psi}}{\tau f_{\psi}+f_{\theta}}\Big)(Y)\right),\\
    \boldsymbol{\boldsymbol{\Gamma}}_{\tau}(\xi)&=\mathbb{V}_{\psi}\left(\varphi_{\xi}^{\mathrm{NCE}}(X)\right)+2\sum_{i=1}^{+\infty}\mathrm{Cov}\Big(\varphi_{\xi}^{\mathrm{NCE}}(X_0),\varphi_{\xi}^{\mathrm{NCE}}(X_i)\Big),\\
    \varphi_{\xi}^{\mathrm{NCE}}&=(\nabla_{\xi}g_{\xi})\frac{f_{\theta}}{f_{\psi}}\Big(\frac{\tau f_{\psi}}{\tau f_{\psi}+f_{\theta}}\Big).
\end{align*}
Finally, Lemma \ref{lem:asymptotic_normality} applies:
$$\sqrt{n}\left(\ence-\xi^{\star}\right)\cvd\mathcal{N}_{d+1}\Big(0, \mathbf{V}_{\tau}^{\mathrm{NCE}}(\truee)\Big)$$
where $\mathbf{V}_{\tau}^{\mathrm{NCE}}(\xi)=\mathbf{J}_{\tau}(\xi)^{-1}\left\{\boldsymbol{\boldsymbol{\Sigma}}_{\tau}(\xi)+\tau^{-1}\boldsymbol{\boldsymbol{\Gamma}}_{\tau}(\xi)\right\}\mathbf{J}_{\tau}(\xi)^{-1}$.

\subsection{Proof of Theorem 5}

For convenience, we will use some shorthand notations. Define the real-valued
measurable functions $Q=f_{\theta}/f_{\psi}$ and $R=\tau
f_{\psi}/(\tau f_{\psi}+f_{\theta})$. Note that we have the relationship
$QR=\tau(1-R)$. In the following, assume that $\xi=\xi^{\star}$, and for any
measurable function $\varphi$, note that $\mathbb
{E}_{\theta}[\varphi]$ stands for the expectation of $\varphi(X)$ where
$X\sim\Pt$, and that $\nabla\nabla^T g_{\xi}$ stands for the measurable
matrix-valued function
$x\mapsto\nabla_{\xi}g_{\xi}(x)(\nabla_{\xi}g_{\xi}(x))^T$. We begin with the
following computations:
\begin{align*}
 \mathbf{J}(\xi)&=\mathbb{E}_{\theta}\Big[\nabla\nabla^Tg_{\xi}\Big],\\
     \boldsymbol{\Sigma}(\xi)&=\mathbb{E}_{\theta}\Big[\nabla\nabla^Tg_{\xi}\Big]-\mathbb{E}_{\theta}\Big[\nabla g_{\xi}\Big]\mathbb{E}_{\theta}\Big[\nabla^Tg_{\xi}\Big],\\
     \boldsymbol{\Gamma}(\xi)&=\mathbb{E}_{\psi}\Big[\nabla\nabla^Tg_{\xi}Q^2\Big]-\mathbb{E}_{\psi}\Big[\nabla g_{\xi}Q\Big]\mathbb{E}_{\psi}\Big[\nabla^Tg_{\xi}Q\Big]\\
     &=\mathbb{E}_{\theta}\Big[\nabla\nabla^Tg_{\xi}(R^{-1}-1)\Big]\times\tau-\mathbb{E}_{\theta}\Big[\nabla g_{\xi}\Big]\mathbb{E}_{\theta}\Big[\nabla^Tg_{\xi}\Big],\\
     \mathbf{J}_{\tau}(\xi)&=\mathbb{E}_{\theta}\Big[\nabla\nabla^Tg_{\xi}R\Big],\\
     \boldsymbol{\Sigma}_{\tau}(\xi)&=\mathbb{E}_{\theta}\Big[\nabla\nabla^Tg_{\xi}R^2\Big]-\mathbb{E}_{\theta}\Big[\nabla g_{\xi}R\Big]\mathbb{E}_{\theta}\Big[\nabla^Tg_{\xi}R\Big],\\
     \boldsymbol{\Gamma}_{\tau}(\xi)&=\mathbb{E}_{\psi}\Big[\nabla\nabla^Tg_{\xi}Q^2R^2\Big]-\mathbb{E}_{\psi}\Big[\nabla g_{\xi}QR\Big]\mathbb{E}_{\psi}\Big[\nabla^Tg_{\xi}QR\Big]\\
     &=\mathbb{E}_{\theta}\Big[\nabla\nabla^Tg_{\xi}R(1-R)\Big]\times\tau-\mathbb{E}_{\theta}\Big[\nabla g_{\xi}R\Big]\mathbb{E}_{\theta}\Big[\nabla^Tg_{\xi}R\Big]. \end{align*}
     
Fortunately, the expression of the asymptotic variances simplify, 
as shown by the following lemma. 
     
    \begin{lem} Let $Z$ be any real-valued, non-negative measurable function such that $\mathbb{E}_{\theta}\big[\nabla\nabla^Tg_{\xi}Z\big]$ is finite and invertible. Then:
$$\mathbf{M}:=\mathbb{E}_{\theta}\big[\nabla\nabla^Tg_{\xi}Z\big]^{-1}\mathbb{E}_{\theta}\big[\nabla g_{\xi}Z\big]\mathbb{E}_{\theta}\big[\nabla^T g_{\xi}Z\big]\mathbb{E}_{\theta}\big[\nabla\nabla^Tg_{\xi}Z\big]^{-1}=\left(\begin{array}{cc}
0_{\mathbb{R}^{d\times d}} & 0_{\mathbb{R}^{d}}  \\
0_{\mathbb{R}^{d}}^T & 1  \end{array} \right).$$\label{lem:empty_matrix}
\end{lem}

The proof of Lemma \ref{lem:empty_matrix} follows from a direct block matrix computation. For the sake of completeness, we present a proof in the supplement.

Let $\mathbf{M}$ be defined as in Lemma \ref{lem:empty_matrix}, matrix calculations yield
     \begin{align*}
     \mathbf{J}(\xi)^{-1}\boldsymbol{\Sigma}(\xi)\mathbf{J}(\xi)^{-1}&=\mathbb{E}_{\theta}\Big[\nabla\nabla^Tg_{\xi}\Big]^{-1}-\mathbf{M},\\
\mathbf{J}_{\tau}(\xi)^{-1}\boldsymbol{\Sigma}_{\tau}(\xi)\mathbf{J}_{\tau}(\xi)^{-1}&=\mathbb{E}_{\theta}\Big[\nabla\nabla^Tg_{\xi}R\Big]^{-1}\mathbb{E}_{\theta}\Big[\nabla\nabla^Tg_{\xi}R^2\Big]\mathbb{E}_{\theta}\Big[\nabla\nabla^Tg_{\xi}R\Big]^{-1}-\mathbf{M},\\
\mathbf{J}(\xi)^{-1}\boldsymbol{\Gamma}(\xi)\mathbf{J}(\xi)^{-1}&=\tau\mathbb{E}_{\theta}\Big[\nabla\nabla^Tg_{\xi}\Big]^{-1}\mathbb{E}_{\theta}\Big[\nabla\nabla^Tg_{\xi}(R^{-1}-1)\Big]\mathbb{E}_{\theta}\Big[\nabla\nabla^Tg_{\xi}\Big]^{-1}-\mathbf{M},\\
\mathbf{J}_{\tau}(\xi)^{-1}\boldsymbol{\Gamma}_{\tau}(\xi)\mathbf{J}_{\tau}(\xi)^{-1}&=\tau\mathbb{E}_{\theta}\Big[\nabla\nabla^Tg_{\xi}R\Big]^{-1}\mathbb{E}_{\theta}\Big[\nabla\nabla^Tg_{\xi}R(1-R)\Big]\mathbb{E}_{\theta}\Big[\nabla\nabla^Tg_{\xi}R\Big]^{-1}-\mathbf{M}.\end{align*}

Summing up these expressions we finally get
$$
\mathbf{V}_{\tau}^{\mathrm{IS}}(\xi)=\mathbb{E}_{\theta}\Big[\nabla\nabla^Tg_{\xi}\Big]^{-1}\mathbb{E}_{\theta}\Big[\nabla\nabla^Tg_{\xi}R^{-1}\Big]\mathbb{E}_{\theta}\Big[\nabla\nabla^Tg_{\xi}\Big]^{-1}-(1+\tau^{-1})\mathbf{M},$$
$$\mathbf{V}_{\tau}^{\mathrm{NCE}}(\xi)=\mathbb{E}_{\theta}\Big[\nabla\nabla^Tg_{\xi}R\Big]^{-1}-(1+\tau^{-1})\mathbf{M}.$$

Now, to compare these variances, the idea is the following: $(x,y)\mapsto
x^2/y$ is a convex function on $\mathbb{R}^2$, which means Jensen's inequality
ensures that for any random variables $X,Y$ such that the following
expectations exist we have 
$\mathbb{E}[X^2/Y]\ge\mathbb{E}[X]^2/\mathbb{E}[Y]$. Here the variances are matrices, but it turns out that it is possible to use a generalization of Jensen's inequality to the Loewner partial order on matrices.  We introduce the following notations:
\begin{align*}
    \mathbb{M}_{n,m}&\text{ is the set of $n\times m$ matrices},\\
     \mathbb{S}_n&\text{ is the set of $n\times n$ symmetric matrices},\\
    \mathbb{S}_n^+&\text{ is the set of ($n\times n$ symmetric) positive semi-definite matrices},\\
     \mathbb{S}_n^{++}&\text{ is the set of ($n\times n$ symmetric) positive definite matrices},\\
    \mathcal{R}(A)&\text{ is the range of $A$},\\
     \Delta_{n,m}&=\Big\{(A,B)\in \mathbb{S}_n^{+}\times \mathbb{M}_{n,m}\ :\ \mathcal{R}(B)\subset\mathcal{R}(A)\Big\},\\
A^{\dagger}&\text{ denotes the Moore-Penrose pseudo-inverse of }A,\\
\succcurlyeq&\text{ denotes the Loewner partial order }(A_1\succcurlyeq A_2 \text{ iff } A_1-A_2\in \mathbb{S}_n^+).
\end{align*}

\begin{lem}
 Let A,B be random matrices such that $(A,B)\in \Delta_{n,m}$ with probability one for some positive integers $n,m$. Let $\varphi:(A,B)\mapsto B^TA^{\dagger}B$ defined on $\Delta_{n,m}$.
Then $\mathbb{E}[\varphi(A,B)]\succcurlyeq \varphi(\mathbb{E}[A],\mathbb{E}[B])$ provided that the three expectations exist.\label{lem:jensen_matrices}
\end{lem} 

\begin{proof}
We just have to prove that $f$ is convex with respect to $\succcurlyeq$, i.e. that for any $\lambda\in[0,1]$, and any $(A_1,B_1),(A_2,B_2) \in \Delta_{n,m}$ we have $\lambda \varphi(A_1,B_1)+(1-\lambda) \varphi(A_2,B_2)\succcurlyeq \varphi(\lambda
(A_1,B_1)+(1-\lambda)(A_2,B_2))$. Indeed, if this convex relationship on matrices is satisfied then for any $x\in\mathbb{R}^m$, the real-valued map $q:(A,B)\mapsto x^T\varphi(A,B)x$ is necessarily convex on $\Delta_{n,m}$. Consequently, Jensen's inequality applies, i.e. for any random $(A,B)\in \Delta_{n,m}$ a.s. and any $x\in\mathbb{R}^m$ we have 
\[x^T\mathbb{E}[\varphi(A,B)]x=\mathbb{E}[q(A,B)]\ge q(\mathbb{E}[A],\mathbb{E}[B])=x^T\varphi(\mathbb{E}[A],\mathbb{E}[B])x
\] which is the claim of the lemma.

Now, to prove that $\varphi$ is convex with respect to $\succcurlyeq$, we use a
property of the generalized Schur's complement in positive semi-definite
matrices (see \citet{boyd_convex_optimization} p.651): let $A\in
\mathbb{S}_n,B\in\mathbb{M}_{n,m},C\in \mathbb{S}_{m}$, and consider the block
symmetric matrix
\[
	D=\left(\begin{array}{cc}
A & B  \\
B^T & C  \end{array} \right).
\]
Then we have $$D\succcurlyeq0\hspace{0.5cm}\Leftrightarrow \hspace{0.5cm}A\succcurlyeq0\hspace{0.2cm},\hspace{0.2cm}\mathcal{R}(B)\subset
\mathcal{R}(A)\hspace{0.2cm},\hspace{0.2cm}C-B^TA^{\dagger}B\succcurlyeq0.$$

This leads to a straightforward proof of the convexity of $\varphi$. To our knowledge, the following trick is due to \citet{ando_concavity_pdm}, whose original proof was restricted to positive definite matrices. We use the generalized Schur's complement to extend this result to any $(A,B)\in\Delta_{n,m}$: let  $\lambda\in[0,1]$, and $(A_1,B_1),(A_2,B_2) \in \Delta_{n,m}$. The sum of two positive semi definite matrices is positive semi-definite thus we have 
$$\lambda\left(\begin{array}{cc}
A_1 & B_1  \\
B_1^T & B_1^TA_1^{\dagger}B_1  \end{array} \right)+(1-\lambda)\left(\begin{array}{cc}
A_2 & B_2  \\
B_2^T & B_2^TA_2^{\dagger}B_2  \end{array} \right)\succcurlyeq0$$
which is the same as 
\[\left(
		\begin{array}{cc}
\lambda A_1+(1-\lambda)A_2 & \lambda B_1+(1-\lambda)B_2  \\
\lambda B_1^T+(1-\lambda)B_2^T & \lambda
B_1^TA_1^{\dagger}B_1+(1-\lambda)B_2^TA_2^{\dagger}B_2\\
		\end{array} 
\right)\succcurlyeq 0.
\]
 Consequently, the generalised Schur's complement in this last block matrix is also positive semi-definite, i.e. 
 $$\lambda
B_1^TA_1^{\dagger}B_1+(1-\lambda)B_2^TA_2^{\dagger}B_2\succcurlyeq\big(\lambda B_1+(1-\lambda)B_2\big)^T\big[\lambda A_1+(1-\lambda)A_2\big]^{\dagger}\big(\lambda B_1+(1-\lambda)B_2\big)$$
which proves the convexity of $\varphi$ with respect to $\succcurlyeq$, and thus the claim of the lemma.
\end{proof}

Finally, we compare the asymptotic variances of the two estimators. Note that for any $(A,B)\in \mathbb{S}_n\times \mathbb{S}_n^{++}$, and for every $x\in\mathbb{R}^n$, we have 
$$x^TAx\ge0\hspace{0.5cm}\Leftrightarrow\hspace{0.5cm}x^TBABx\ge 0.$$
Indeed, if $A$ is semi definite positive then for some integer $k$ we can find $P\in \mathbb{M}_{k,n}$ such that $A=P^TP$, moreover, $B$ being symmetric we have $x^TBABx=\|PBx\|^2\ge0$. The direct implication is enough since $B^{-1}\in \mathbb{S}_n^{++}$.

Consequently, the relation $\mathbf{V}_{\tau}^{\mathrm{IS}}(\xi)\succcurlyeq\mathbf{V}_{\tau}^{\mathrm{NCE}}(\xi)$ is equivalent to the relation 
\begin{equation}
\mathbb{E}_{\theta}\Big[\nabla\nabla^Tg_{\xi}R^{-1}\Big]\succcurlyeq\mathbb{E}_{\theta}\Big[\nabla\nabla^Tg_{\xi}\Big]\mathbb{E}_{\theta}\Big[\nabla\nabla^Tg_{\xi}R\Big]^{-1}\mathbb{E}_{\theta}\Big[\nabla\nabla^Tg_{\xi}\Big].\label{eq:inequality_matrices}
\end{equation}
Inequality \eqref{eq:inequality_matrices} is a direct application of Lemma \ref{lem:jensen_matrices} (let $B=\nabla\nabla^Tg_{\xi}$, $A=BR$; note that $(A,B)\in\Delta_{d+1,d+1}$ almost surely; and use basic properties of the pseudo-inverse).

\bibliographystyle{apalike}
\bibliography{biblio}

\section{Supplement}
\subsection{Proofs of technical lemmas}
\subsubsection{Proof of Lemma \ref{lem:LLN_monotone}}
 For all $k\in\mathbb{N}$, eventually (for any $m\ge k$) we have
\[\frac{1}{m}\sum_{j=1}^mf_{k}(X_j)\le\frac{1}{m}\sum_{j=1}^mf_{m}(X_j)\le\frac{1}{m}\sum_{j=1}^mf(X_j).\]
Moreover, since $(X_j)_{j\ge1}$ is $\refdist$-ergodic, the law of large numbers
applies (even if the expectations are infinite, since $f_{k}$ and $f$ are non-negative):
\[\frac{1}{m}\sum_{j=1}^mf_{k}(X_j)\underset{m\rightarrow+\infty}{\overset{a.s.}{\longrightarrow}}\mathbb{E}_{\psi}[f_k(X)] \hspace{0.5cm}\text{and}\hspace{0.5cm} \frac{1}{m}\sum_{j=1}^mf(X_j)\underset{m\rightarrow+\infty}{\overset{a.s.}{\longrightarrow}}\mathbb{E}_{\psi}[f(X)].\]
Thus, there is a set of probability one on which for every $k\in\mathbb{N}$,
\[\mathbb{E}_{\psi}[f_k(X)]\le\underset{m\rightarrow+\infty}{\text{lim inf}}\ \frac{1}{m}\sum_{j=1}^mf_{m}(X_j)\le\underset{m\rightarrow+\infty}{\text{lim sup}}\ \frac{1}{m}\sum_{j=1}^mf_{m}(X_j)\le\mathbb{E}_{\psi}[f(X)].\]
Since the inequality holds for any $k\in\mathbb{N}$, it also holds for the supremum over $k$:
\[\underset{k\in\mathbb{N}}{\text{sup}}\ \mathbb{E}_{\psi}[f_k(X)]\le\underset{m\rightarrow+\infty}{\text{lim inf}}\ \frac{1}{m}\sum_{j=1}^mf_{m}(X_j)\le\underset{m\rightarrow+\infty}{\text{lim sup}}\ \frac{1}{m}\sum_{j=1}^mf_{m}(X_j)\le\mathbb{E}_{\psi}[f(X)].
\]
Finally, the monotone convergence theorem yields
\[\underset{k\in\mathbb{N}}{\text{sup}}\ \mathbb{E}_{\psi}[f_k(X)]=\underset{k\rightarrow+\infty}{\text{lim}}\ \mathbb{E}_{\psi}[f_k(X)]=\mathbb{E}_{\psi}\big[\underset{k\rightarrow+\infty}{\text{lim}}\ f_k(X)\big]=\mathbb{E}_{\psi}[f(X)].\]
Consequently, the lower and upper limits are both equal to $\mathbb{E}_{\psi}[f(X)]$ almost surely.

\subsubsection{Proof of Lemma \ref{lem:LLN_dominated}}

Since $(X_j)_{j\ge1}$ is $\refdist$-ergodic and $f$ is dominated by the integrable function $g$, the law of large numbers applies to function $f$. Thus, we just need to prove that 
\[\left|\frac{1}{m}\sum_{j=1}^m\left\{f_m(X_j)-f(X_j)\right\}\right|\underset{m\rightarrow+\infty}{\overset{a.s.}{\longrightarrow}}0.\]
To do so, use the fact that
\begin{align*}
    \left|\frac{1}{m}\sum_{j=1}^m\left\{f_m(X_j)-f(X_j)\right\}\right|&\le\frac{1}{m}\sum_{j=1}^m\left|f_m(X_j)-f(X_j)\right|\\&\le\frac{1}{m}\sum_{j=1}^m\underset{k\ge m}{\text{sup}}\ \big|f_k(X_j)-f(X_j)\big|.
\end{align*}
Define $h_m=2g-\underset{k\ge m}{\text{sup}}\ |f_k-f|$ and note that $(h_m)_{m\ge1}$ is a non-decreasing sequence of non negative functions converging pointwise towards $2g$. Lemma 1 yields
\[\frac{1}{m}\sum_{j=1}^mh_m(X_j)\underset{m\rightarrow+\infty}{\overset{a.s.}{\longrightarrow}}\mathbb{E}_{\psi}[2g(X)].\]
Finally $g$ is integrable, thus the remainder converges almost surely towards zero:
\[\frac{1}{m}\sum_{j=1}^m\underset{k\ge m}{\text{sup}}\ \big|f_k(X_j)-f(X_j)\big|=\frac{2}{m}\sum_{j=1}^mg(X_j)-\frac{1}{m}\sum_{j=1}^mh_m(X_j)\underset{m\rightarrow+\infty}{\overset{a.s.}{\longrightarrow}}0.\]

\subsubsection{Proof of Lemma \ref{lem:LLN_infinite}}

Since $g$ is integrable and $(X_j)_{j\ge1}$ is $\refdist$-ergodic we have
\[\frac{1}{m}\sum_{j=1}^mg(X_j)\underset{m\rightarrow+\infty}{\overset{a.s.}{\longrightarrow}}\mathbb{E}_{\psi}\big[g(X)\big]<+\infty.\]
Thus we only need to show that
\[\frac{1}{m}\sum_{j=1}^m\left\{g(X_j)-f_{m}(X_j)\right\}\underset{m\rightarrow+\infty}{\overset{a.s.}{\longrightarrow}}+\infty.\]
Define $h_m=g-\underset{k\ge m}{\sup}f_{k}$, an increasing sequence of non negative functions converging pointwise to $g-f$. Lemma 1 applies whether $g-f$ is integrable or not:
\[\frac{1}{m}\sum_{j=1}^m\left(g(X_j)-f_{m}(X_j)\right)\ge\frac{1}{m}\sum_{j=1}^mh_m(X_j)\underset{m\rightarrow+\infty}{\overset{a.s.}{\longrightarrow}}\mathbb{E}_{\psi}\big[g(X)-f(X)\big].\]
The following inequality shows that the expectation is indeed infinite:
\[\mathbb{E}_{\psi}\big[g(X)-f(X)\big]=\mathbb{E}_{\psi}\big[\big(g(X)-f(X)_{+}\big)+f(X)_{-}\big]\ge\mathbb{E}_{\psi}\big[f(X)_{-}\big]=+\infty.\]

\subsubsection{Proof of Lemma \ref{lem:ULLN_theorem}}

To begin, note that measurability of the supremum is ensured by the lower semi-continuity of the maps $\theta\mapsto \varphi(\theta,x)$ on a set of probability one that does not depend on $\theta$.

For every $\theta\in K$,
consider the following function:
$$f_{\theta}(\eta)=\mathbb{E}_{\psi}\bigg[\underset{\phi\in B(\theta,\eta)}{\text{sup}}\ \|\varphi(\phi,X)-\varphi(\theta,X)\|\bigg].$$
Dominated convergence implies that $f_{\theta}(\eta)$ converges to zero when $\eta$ goes to zero. This is enough to ensure the continuity of the map $\theta\mapsto\mathbb{E}_{\psi}\big[\varphi(\theta,X)\big]$ because of the following inequality:
$$\underset{\phi\in B(\theta,\eta)}{\text{sup}}\ \|\mathbb{E}_{\psi}\big[\varphi(\phi,X)-\varphi(\theta,X)\big]\|\le f_{\theta}(\eta).$$

Let $\varepsilon>0$. For every $\theta\in K$, we can always find
$\eta_{(\theta,\varepsilon)}>0$ small enough such that
$f_{\theta}(\eta_{(\theta,\varepsilon)})<\varepsilon$. Note that the open balls
centered on $\theta\in K$ of radius $\eta_{(\theta,\varepsilon)}$, form an open
cover of $K$, from which we can extract a finite subcover thanks to the
compactness assumption.
Thus we can build a finite set $\{\phi_1,...,\phi_I\}\subset K$ (centers of the balls) such that
$$K\subset\bigcup_{i=1}^IB_i,\qquad B_i=B(\phi_i,\eta_{(\phi_i,\varepsilon)}).$$
Now, for any $\theta\in K$ define $i_{\theta}$ as the smallest integer $i\in\{1,...,I\}$ such that $\theta\in B_i$, and consider the following equality:
\begin{align*}
    \frac{1}{m}\sum_{j=1}^m\varphi(\theta,X_j)-\mathbb{E}_{\psi}\big[\varphi(\theta,X)\big]&=\frac{1}{m}\sum_{j=1}^m\Big\{\varphi(\theta,X_j)-\varphi(\phi_{i_{\theta}},X_j)\Big\}\\
    &+\frac{1}{m}\sum_{j=1}^m\varphi(\phi_{i_{\theta}},X_j)-\mathbb{E}_{\psi}\big[\varphi(\phi_{i_{\theta}},X)\big]\\
    &+\mathbb{E}_{\psi}\big[\varphi(\phi_{i_{\theta}},X)\big]-\mathbb{E}_{\psi}\big[\varphi(\theta,X)\big]
\end{align*}

The three last terms are functions of $\theta$ for which we want to bound the uniform norm.

The uniform norm of the third term is lower than $\varepsilon$ since $\forall\theta\in K$, $d(\theta,\phi_{i_{\theta}})\le\eta_{(\phi_{i_{\theta}},\varepsilon)}$. The second term converges to zero by the law of large number since $\{\phi_1,...,\phi_J\}$ is finite. Finally, the uniform norm of the second term can be bounded by
$$U_m=\underset{1\le i\le I}{\max}\ \Bigg\{ \frac{1}{m}\sum_{j=1}^m \underset{\theta\in B_i}{\sup}\ \|\varphi(\theta,X_j)-\varphi(\phi_i,X_j)\|\Bigg\}.$$

The supremum are integrable by assumption, thus the law of large numbers applies:
\[
	U_m\underset{m\rightarrow+\infty}{\overset{a.s.}{\longrightarrow}}\underset{1\le
	i\le I}{\max}\ f_{\phi_i}(\eta_{(\phi_i,\varepsilon)})<\varepsilon.
\]

To sum up, we have just proven that for any $\varepsilon>0$, almost surely,
$$\underset{m\rightarrow+\infty}{\text{lim sup}}\ \left\{ \underset{\theta\in K}{\text{sup}}\ \left\|\frac{1}{m}\sum_{j=1}^m\varphi(\theta,X_j)-\mathbb{E}_{\psi}\big[\varphi(\theta,X)\big]\right\|\right\}<2\varepsilon.$$

Since $\varepsilon$ is arbitrary small, we get the first claim of the lemma.

Now, if $\widetilde{\theta}_m\rightarrow\widetilde{\theta}$, we have eventually $\|\widetilde{\theta}_m-\widetilde{\theta}\|\le\varepsilon$ with probability one. This yields the following inequality for $m$ large enough:
\begin{multline*}
\left\|\frac{1}{m}\sum_{j=1}^m\varphi(\widetilde{\theta}_m,X_j)-\mathbb{E}_{\psi}\Big[\varphi(\widetilde{\theta},X)\Big]\right\|\le\underset{\theta\in B (\widetilde{\theta}_m,\varepsilon)}{\text{sup}}\ \left\|\frac{1}{m}\sum_{j=1}^m\varphi(\theta,X_j)-\mathbb{E}_{\psi}\Big[\varphi(\theta,X)\Big]\right\|\\
+\left\|\mathbb{E}_{\psi}\Big[\varphi(\widetilde{\theta}_m,X)\Big]-\mathbb{E}_{\psi}\Big[\varphi(\widetilde{\theta},X)\Big]\right\|.
\end{multline*}

The first term converges to zero since the first claim of the lemma applies to the compact closure of $B(\widetilde{\theta},\varepsilon)$. The continuity of the map $\theta\mapsto\mathbb{E}_{\psi}\big[\varphi(\theta,X)\big]$ ensures that the second term also goes to zero, proving the second claim of the Lemma.

\subsubsection{Proof of Lemma \ref{lem:asymptotic_normality}}

 Let $\varepsilon>0$, and $G_n(\theta,\omega)=\nabla_{\theta}\ln(\theta,\omega$) defined on $B(\theta^{\star},\varepsilon)$. Define also $g_{k}^{(n)}(\theta)$ as the $k$-th component of $G_n(\theta)$. By assumption, for any $\delta>0$,
$$\left\{\omega\in\Omega:\max\Big(\|\widehat{\theta}_n-\theta^{\star}\| \ ,\ \|\sqrt{n}G_n(\widehat{\theta}_n)\| \ ,\ \underset{\theta\in B(\theta^{\star},\varepsilon)}{\sup}\|\nabla_{\theta}^2 \ln(\theta)-\mathcal{H}(\theta)\|\Big)\le\delta\right\}$$
defines a sequence of sets whose probability goes to one.

On any of these sets (for a fixed $\omega$), Taylor Lagrange's theorem ensures that we can find $(\widetilde{\theta}_j^{(n)})_{j=1,...,d}$ on the segment line $[\theta^{\star},\widehat{\theta}_n]$ such that
$$G_n(\widehat{\theta}_n)=G_n(\theta^{\star})+\mathbf{H}_n(\widehat{\theta}_n-\theta^{\star}),\qquad\mathbf{H}_n=\left(\begin{array}{c}
    \big(\nabla_{\theta}g_{1}^{(n)}(\widetilde{\theta}_1^{(n)})\big)^T   \\
     \vdots  \\
     \big(\nabla_{\theta}g_{d}^{(n)}(\widetilde{\theta}_{d}^{(n)})\big)^T
\end{array}\right).$$
In particular, for any $\delta\in]0,\varepsilon[$,
$$\|\mathbf{H}_n-\mathcal{H}(\theta^{\star})\|\le d \underset{\theta\in B(\theta^{\star},\varepsilon)}{\sup}\|\nabla_{\theta}^2 \ln(\theta)-\mathcal{H}(\theta)\|+\sum_{j=1}^d\|\mathcal{H}(\widetilde{\theta}_{j}^{(n)})-\mathcal{H}(\theta^{\star})\|.$$
For any $j=1,...,d$, the distance between $\widetilde{\theta}_{j}^{(n)}$ and $\theta^{\star}$ is at most $\delta$, and $\mathcal{H}$ is continuous, thus $\delta$ can always be chosen small enough such that $\mathbf{H}_n$ is invertible. We thus have:
$$\widehat{\theta}_n-\theta^{\star}=\mathbf{H}_n^{-1}\big\{G_n(\widehat{\theta}_n)-G_n(\theta^{\star})\big\}$$
$$\sqrt{n}(\widehat{\theta}_n-\theta^{\star})+\mathcal{H}(\theta^{\star})^{-1}\sqrt{n}G_n(\theta^{\star})=\mathbf{H}_n^{-1}\sqrt{n}G_n(\widehat{\theta}_n)-\big\{\mathbf{H}_n^{-1}-\mathcal{H}(\theta^{\star})^{-1}\big\}\sqrt{n}G_n(\theta^{\star})$$
The right hand side converges in probability to zero because $G_n(\widehat{\theta}_n)=o_{\mathbb{P}}(n^{-1/2})$ by assumption, and because $\sqrt{n}G_n(\theta^{\star})$ converges in distribution and is thus bounded in probability. The last equality being true on a sequence of sets whose probability goes to one, this implies that the left hand side must also converge to zero in probability.

The last conclusion follows from Slutsky's lemma.

\subsubsection{Proof of Lemma \ref{lem:CLT_dominated}}

Before proving the lemma, we recall a powerful result from \citet{galin_jones}. Under (X2), the chain $(X_j)_{j\ge1}$ is asymptotically
uncorrelated with exponential decay, i.e. there is some $\gamma>0$ such that
$$\rho(n)= \sup\Big\{\text{Corr}(U, V )\ ,\ U \in L^2(\mathcal{F}_1^k
)\ ,\ V \in L^2(\mathcal{F}_{k+n}^{\infty}) \ ,\ k \ge 1\Big\}=\bigO(e^{-\gamma n})$$
where $\mathcal{F}_k^m$ is the sigma-algebra generated by $X_k, . . . , X_m$.

Let $h_n=f_n-f$, and note that $\mathbb{V}_{\psi}(h_n(X_0))\le\mathbb{E}_{\psi}\big[(h_n(X_0))^2\big]\cvn0$ by dominated convergence. Combined with the previous result, this implies that
$$\frac{1}{n}\mathbb{V}\left(\sum_{i=1}^nh_n(X_i)\right)=\mathbb{V}_{\psi}(h_n(X_0))\times\left\{1+2\sum_{i=1}^n\frac{n-i}{n}\text{Corr}(h_n(X_0),h_n(X_{i}))\right\}\cvn0$$
since
$$\left|\sum_{i=1}^n\frac{n-i}{n}\mathrm{Corr}(h_n(X_0),h_n(X_{i}))\right|\le\sum_{i=1}^{+\infty}\rho(i)<+\infty.$$

The first claim of the lemma follows from Chebyshev's inequality, since for any $\varepsilon>0$
\begin{align*}
    \mathbb{P}\left(\frac{1}{n}\sum_{i=1}^nh_n(X_i)-\mathbb{E}\big[h_n(X)\big]\ge\frac{\varepsilon}{\sqrt{n}}\right)&\le\frac{1}{n\varepsilon^2}\mathbb{V}\left(\sum_{i=1}^nh_n(X_i)\right)\cvn0.
\end{align*}

Finally, under (X2) a $\sqrt{n}$-CLT holds for $f$ dominated by $g$
$$\sqrt{n}\left(\frac{1}{n}\sum_{i=1}^nf(X_i)-\mathbb{E}[f(X)]\right)\cvd\mathcal{N}\big(0,\sigma_f^2\big).$$
An application of Slutsky's lemma yields the second claim of the lemma.

\subsection{Proofs of the remaining lemmas}
\subsubsection{Proof of Lemma \ref{lem:full_rank_hessian}}
 Assumption (H2) ensures in particular that the partition function $\theta\mapsto\mathcal{Z}(\theta)$ is differentiable in a neighborhood of $\mle$.
Write the Hessian of the Poisson Transform as the following block matrix:
\[\nabla_{(\theta,\nu)}^2 \ln(\theta,\nu)=\left(\begin{array}{cc}
    A & b \\
     b^T & c
\end{array}\right)\]
where
\begin{align*}
    A&=\nabla_{\theta}^2 \ln(\theta,\nu)=\frac{1}{n}\sum_{i=1}^n\nabla_{\theta}^2\log h_{\theta}(y_i)-e^{\nu}\frac{\mathcal{Z}(\theta)}{\mathcal{Z}(\psi)}\frac{\nabla_{\theta}^2\mathcal{Z}(\theta)}{\mathcal{Z}(\theta)},\\
   b&=\nabla_{\theta} \frac{\partial}{\partial \nu}\ln(\theta,\nu)=-e^{\nu}\frac{\mathcal{Z}(\theta)}{\mathcal{Z}(\psi)}\frac{\nabla_{\theta}\mathcal{Z}(\theta)}{\mathcal{Z}(\theta)},\\
    c&=\frac{\partial^2}{\partial \nu^2} \ln(\theta,\nu)=-e^{\nu}\frac{\mathcal{Z(\theta)}}{\mathcal{Z}(\psi)}<0.
\end{align*}
The Hessian of the Poisson transform is negative definite if and only if Schur's complement of $c$ in the Hessian also is. Use the following equality to compute it: \[\nabla_{\theta}^2\log\mathcal{Z}(\theta)=\frac{\nabla_{\theta}^2\mathcal{Z}(\theta)}{\mathcal{Z}(\theta)}-\frac{\nabla_{\theta}\mathcal{Z}(\theta)\big(\nabla_{\theta}\mathcal{Z}(\theta)\big)^T}{\mathcal{Z}(\theta)^2},\]
\[A-bc^{-1}b^T=\frac{1}{n}\sum_{i=1}^n\nabla_{\theta}^2\log h_{\theta}(y_i)-e^{\nu}\frac{\mathcal{Z}(\theta)}{\mathcal{Z}(\psi)}\nabla_{\theta}^2\log\mathcal{Z}(\theta).\]
 
 At the point $\xi=\emle$, Schur's complement of $c$ is also the Hessian of the log likelihood:
 \[\nabla_{\theta}^2\ln(\theta)=\frac{1}{n}\sum_{i=1}^n\nabla_{\theta}^2\log h_{\theta}(y_i)-\nabla_{\theta}^2\log\mathcal{Z}(\theta),\qquad\Big\{e^{\nu}\frac{\mathcal{Z}(\theta)}{\mathcal{Z}(\psi)}\Big\}_{\big|_{\xi=\emle}}=1.\]
 
 \subsubsection{Proof of Lemma \ref{lem:usc_nce}}

 Let $\xi_n\rightarrow\xi$, we have 
$$\underset{n\rightarrow+\infty}{\lim}\ \underset{k\ge n}{\sup}\ \mathcal{M}_{\tau}^{\mathrm{NCE}}(\xi_k)\le\frac{1}{\mathcal{Z}(\psi)}\underset{n\rightarrow+\infty}{\lim}\left\{\int_{\mathcal{X}}\underset{k\ge n}{\sup}\ \varphi_k(x)\mu(\dx)\right\},$$
where
$$\varphi_k(x)=\log\bigg\{\frac{e^{\nu_k}h_{\theta_k}(x)}{e^{\nu^{\star}}h_{\theta^{\star}}(x)}\bigg\}e^{\nu^{\star}}h_{\theta^{\star}}(x)+\log\bigg\{\frac{\tau h_{\psi}(x)+e^{\nu^{\star}}h_{\theta^{\star}}(x)}{\tau h_{\psi}(x)+e^{\nu_k}h_{\theta_k}(x)}\bigg\}\big(\tau h_{\psi}(x)+e^{\nu^{\star}}h_{\theta^{\star}}(x)\big).$$
The sequence $\big\{\sup_{k\ge n}\ \varphi_k\big\}$ is a decreasing sequence
converging pointwise. It may be bounded from above thanks to the log-sum
inequality, since for any $k$ we have
$$\varphi_k\le\log\bigg\{\frac{e^{\nu_k}h_{\theta_k}}{e^{\nu^{\star}}h_{\theta^{\star}}}\bigg\}e^{\nu^{\star}}h_{\theta^{\star}}+\log\bigg\{\frac{\tau h_{\psi}}{\tau h_{\psi}}\bigg\}\tau h_{\psi}+\log\bigg\{\frac{e^{\nu^{\star}}h_{\theta^{\star}}}{e^{\nu_k}h_{\theta_k}}\bigg\}e^{\nu^{\star}}h_{\theta^{\star}}=0.$$
Monotone convergence theorem applies:
$$\underset{n\rightarrow+\infty}{\lim}\ \underset{k\ge n}{\sup}\ \mathcal{M}_{\tau}^{\mathrm{NCE}}(\xi_k)\le\frac{1}{\mathcal{Z}(\psi)}\int_{\mathcal{X}}\underset{n\rightarrow+\infty}{\lim}\ \varphi_n(x)\mu(\dx)=\mathcal{M}_{\tau}^{\mathrm{NCE}}(\xi).$$

\subsubsection{Proof of Lemma \ref{lem:empty_matrix}}
Without loss of generality, we may suppose that $\mathbb{E}_{\theta}[Z]=1$.
Recall the following expressions:
$$\nabla g_{\xi}=\left(\begin{array}{c}
\nabla g_{\theta}  \\
1  \end{array} \right) \hspace{1cm} \nabla\nabla^Tg_{\xi}=\left(\begin{array}{cc}
\nabla\nabla^Tg_{\theta} & \nabla g_{\theta}  \\
\nabla^Tg_{\theta} & 1  \end{array} \right).$$
We thus have
$$\mathbb{E}_{\theta}\big[\nabla g_{\xi}Z\big]\mathbb{E}_{\theta}\big[\nabla^T g_{\xi}Z\big]=\left(\begin{array}{cc}
\mathbb{E}_{\theta}[\nabla g_{\theta}Z]\mathbb{E}_{\theta}[\nabla^Tg_{\theta}Z] & \mathbb{E}_{\theta}[\nabla g_{\theta}Z]  \\
\mathbb{E}_{\theta}[\nabla^Tg_{\theta}Z] & 1  \end{array} \right),$$
$$\mathbb{E}_{\theta}\big[\nabla\nabla^Tg_{\xi}Z\big]=\left(\begin{array}{cc}
\mathbb{E}_{\theta}[\nabla\nabla^Tg_{\theta}Z] & \mathbb{E}_{\theta}[\nabla g_{\theta}Z]  \\
\mathbb{E}_{\theta}[\nabla^Tg_{\theta}Z] & 1  \end{array} \right).$$
We use the following decomposition
$$\mathbb{E}_{\theta}\big[\nabla g_{\xi}Z\big]\mathbb{E}_{\theta}\big[\nabla^T g_{\xi}Z\big]=\mathbb{E}_{\theta}\big[\nabla\nabla^Tg_{\xi}Z\big]-\left(\begin{array}{cc}
A_Z & 0  \\
0 & 0  \end{array} \right)$$
where Schur's complement $A_Z=\mathbb{E}_{\theta}[\nabla\nabla^Tg_{\theta}Z]-\mathbb{E}_{\theta}[\nabla g_{\theta}Z]\mathbb{E}_{\theta}[\nabla^Tg_{\theta}Z]$ is definite positive.

So we can re-write the matrix $\mathbf{M}$ as:
$$\mathbf{M}=\mathbb{E}_{\theta}\big[\nabla\nabla^Tg_{\xi}Z\big]^{-1}-\mathbb{E}_{\theta}\big[\nabla\nabla^Tg_{\xi}Z\big]^{-1}\left(\begin{array}{cc}
A_Z & 0  \\
0 & 0  \end{array} \right)\mathbb{E}_{\theta}\big[\nabla\nabla^Tg_{\xi}Z\big]^{-1}.$$
Now, on the one hand, an inverse block matrix calculation yields
$$\mathbb{E}_{\theta}\big[\nabla\nabla^Tg_{\xi}Z\big]^{-1}=\left(\begin{array}{cc}
A_Z^{-1} & -A_Z^{-1}\mathbb{E}_{\theta}[\nabla g_{\theta}Z]  \\
-\mathbb{E}_{\theta}[\nabla^T g_{\theta}Z]A_Z^{-1} & 1+\mathbb{E}_{\theta}[\nabla^T g_{\theta}Z]A_Z^{-1}\mathbb{E}_{\theta}[\nabla g_{\theta}Z]  \end{array} \right),$$
while, on the other hand, a direct computation yields
\begin{multline*}
    \mathbb{E}_{\theta}\big[\nabla\nabla^Tg_{\xi}Z\big]^{-1}\left(\begin{array}{cc}
A_Z & 0  \\
0 & 0  \end{array} \right)\mathbb{E}_{\theta}\big[\nabla\nabla^Tg_{\xi}Z\big]^{-1}\\=\left(\begin{array}{cc}
A_Z^{-1} & -A_Z^{-1}\mathbb{E}_{\theta}[\nabla g_{\theta}Z]  \\
-\mathbb{E}_{\theta}[\nabla^T g_{\theta}Z]A_Z^{-1} & \mathbb{E}_{\theta}[\nabla^T g_{\theta}Z]A_Z^{-1}\mathbb{E}_{\theta}[\nabla g_{\theta}Z]  \end{array} \right).
\end{multline*}
The matrix $\mathbf{M}$ being the difference between these two quantities, we get the claim of the lemma.

\subsection{Proofs of MC-MLE consistency and asymptotic normality}
\subsubsection{MC-MLE consistency}
The following proof is a straightforward adaptation of Wald's proof of consistency for the MLE (\citet{wald_1949}). The sketch of proof is mainly inspired from \citet{geyer_lecture_notes}, which has the merit of giving a very accessible presentation of this technical proof.

To begin, define the opposite of the Kullback-Leibler divergence:
\[\lambda(\theta)=\mathbb{E}_{\theta^{\star}}\left[\log\frac{f_{\theta}(Y)}{f_{\theta^{\star}}(Y)}\right]\le0.\]

Since the model is identifiable, $\lambda$ has a unique maximizer achieved at $\theta^{\star}$.
It may be $-\infty$ for some values of $\theta$, but this does not pose
problems in the following proof.

For convenience, we choose to analyse the MC-MLE objective function through the
following translational motion (sharing the same maximiser with $\lIS$):
\[
M_{n}^{\mathrm{IS}}(\theta,\nu)=\frac{1}{n}\sum_{i=1}^{n}\log\Big\{\frac{e^{\nu}h_{\theta}(Y_{i})}{e^{\nu^{\star}}h_{\theta^{\star}}(Y_{i})}\Big\}+1-e^{\nu}\frac{1}{r_{n}}\sum_{j=1}^{r_{n}}\frac{h_{\theta}(X_{j})}{h_{\psi}(X_{j})}.
\]

For any $\xi\in\Xi=\Theta\times\mathbb{R}$,
the law of large numbers yields $M_{n}^{\mathrm{IS}}(\xi)\overset{a.s.}{\underset{n\rightarrow+\infty}{\longrightarrow}}\mathcal{M}^{\mathrm{IS}}(\xi)$
where 
\[
\mathcal{M}^{\mathrm{IS}}(\theta,\nu)=\lambda(\theta)+\nu+\log\frac{\mathcal{Z}(\theta)}{\mathcal{Z}(\psi)}+1-e^{\nu}\frac{\mathcal{Z}(\theta)}{\mathcal{Z}(\psi)}\le0.
\]

Note that by construction $\mathcal{M}^{\mathrm{IS}}$ also has a unique maximiser at
$\xi^{\star}=(\theta^{\star},\nu^{\star})$. 

Let $\eta>0$. Define $K_{\eta}=\{\xi\in K:d(\xi,\xi^{\star})\ge\eta\}$
where $K$ is the compact set defined in (C2). Under (H3), continuity of the maps $\theta\mapsto h_{\theta}(x)$ and monotone
convergence ensure that for any $\xi\in K_{\eta}$,
\[
\lim_{\varepsilon\downarrow0}\mathbb{E}_{\theta^{\star}}\left[\underset{(\phi,\mu)\in B(\xi,\varepsilon)}{\sup}\log\frac{e^{\mu}h_{\phi}(Y)}{e^{\nu^{\star}}h_{\theta^{\star}}(Y)}\right]=\mathbb{E}_{\theta^{\star}}\left[\log\frac{e^{\nu}h_{\theta}(Y)}{e^{\nu^{\star}}h_{\theta^{\star}}(Y)}\right],
\]
while dominated convergence ensures that 
\[
\lim_{\varepsilon\downarrow0}\mathbb{E}_{\psi}\left[\underset{(\phi,\mu)\in B(\xi,\varepsilon)}{\inf}e^{\mu}\frac{h_{\phi}(X)}{h_{\psi}(X)}\right]=e^{\nu}\frac{\mathcal{Z}(\theta)}{\mathcal{Z}(\psi)}.
\]
Thus for any $\xi\in K_{\eta}$ and $\gamma>0$,
we can always find $\varepsilon_{\xi}>0$ such that simultaneously:
\[
\mathbb{E}_{\theta^{\star}}\left[\underset{(\phi,\mu)\in B(\xi,\varepsilon_{\xi})}{\sup}\log\frac{e^{\mu}h_{\phi}(Y)}{e^{\nu^{\star}}h_{\theta^{\star}}(Y)}\right]\le\mathbb{E}_{\theta^{\star}}\left[\log\frac{e^{\nu}h_{\theta}(Y)}{e^{\nu^{\star}}h_{\theta^{\star}}(Y)}\right]+\frac{\gamma}{2},
\]
and 
\[
\mathbb{E}_{\psi}\left[\underset{(\phi,\mu)\in B(\xi,\varepsilon_{\xi})}{\inf}e^{\mu}\frac{h_{\phi}(X)}{h_{\psi}(X)}\right]\ge e^{\nu}\frac{\mathcal{Z}(\theta)}{\mathcal{Z}(\psi)}-\frac{\gamma}{2}.
\]
The set of open balls $\{B(\xi,\varepsilon_{\xi})\ :\ \xi\in K\}$
form an open cover of $K_{\eta}$ from which we can extract a finite
subcover by compactness, i.e. we can build a finite set $\{\xi_{1},...,\xi_{p}\}\subset K_{\eta}$
such that $K_{\eta}\subset\bigcup_{k=1}^{p}B(\xi_{k},\varepsilon_{\xi_{k}})$.
This yields the following inequality: 
\begin{multline*}
\underset{\xi\in K_{\eta}}{\sup}M_{n}^{\mathrm{IS}}(\xi)  \le\underset{k=1,...,p}{\max}\Bigg\{\frac{1}{n}\sum_{i=1}^{n}\underset{(\phi,\mu)\in B(\xi_{k},\varepsilon_{\xi_{k}})}{\sup}\left(\log\frac{e^{\mu}h_{\phi}(Y_{i})}{e^{\nu^{\star}}h_{\theta^{\star}}(Y_{i})}\right)\\
  +1-\frac{1}{r_{n}}\sum_{j=1}^{r_{n}}\underset{(\phi,\mu)\in B(\xi_{k},\varepsilon_{\xi_{k}})}{\inf}\left(e^{\mu}\frac{h_{\phi}(X_{j})}{h_{\psi}(X_{j})}\right)\Bigg\}. 
\end{multline*}
The right hand side converges almost surely as the law of large numbers
applies simultaneously on a finite set. We can thus bound the upper limit: 
\begin{multline*}
\underset{n\rightarrow+\infty}{\lim\sup}\underset{\xi\in K_{\eta}}{\sup}M_{n}^{\mathrm{IS}}(\xi)  \le\underset{k=1,...,p}{\max}\Bigg\{\mathbb{E}_{\theta^{\star}}\left[\underset{(\phi,\mu)\in B(\xi_{k},\varepsilon_{\xi_{k}})}{\sup}\left(\log\frac{e^{\mu}h_{\phi}(Y)}{e^{\nu^{\star}}h_{\theta^{\star}}(Y)}\right)\right]\\
  \quad+1-\mathbb{E}_{\psi}\left[\underset{(\phi,\mu)\in B(\xi_{k},\varepsilon_{\xi_{k}})}{\inf}\left(e^{\mu}\frac{h_{\phi}(X)}{h_{\psi}(X)}\right)\right]\Bigg\},
\end{multline*}
\[ \underset{n\rightarrow+\infty}{\lim\sup}\underset{\xi\in K_{\eta}}{\sup}M_{n}^{\mathrm{IS}}(\xi) \le\underset{k=1,...,p}{\max} \mathcal{M}^{\mathrm{IS}}(\xi_{k})+\gamma\le\underset{\xi\in K_{\eta}}{\sup}\mathcal{M}^{\mathrm{IS}}(\xi)+\gamma.\]
Moreover $\gamma$ is arbitrary small, thus the inequality still holds when $\gamma$ is zero:
\begin{equation}
\underset{n\rightarrow+\infty}{\lim\sup}\underset{\xi\in K_{\eta}}{\sup}M_{n}^{\mathrm{IS}}(\xi)\le\underset{\xi\in K_{\eta}}{\sup}\mathcal{M}^{\mathrm{IS}}(\xi)\hspace{0.5cm}\as \label{eq:heart_IS}
\end{equation}
To conclude, let us prove that the right hand side is negative. Indeed, subadditivity of the supremum yields \[\underset{\xi\in K_{\eta}}{\sup}\mathcal{M}^{\mathrm{IS}}(\theta,\nu)\le\underset{\xi\in K_{\eta}}{\sup}\lambda(\theta)+\underset{\xi\in K_{\eta}}{\sup}\left(\nu+\log\frac{\mathcal{Z}(\theta)}{\mathcal{Z}(\psi)}+1-e^{\nu}\frac{\mathcal{Z}(\theta)}{\mathcal{Z}(\psi)}\right)\]
where the second term is non positive by construction. Under (H3), it is easy to check that $\lambda$ is upper semi continuous, which implies in particular that $\lambda$ achieves its maximum on any compact set. Consequently: $\underset{\xi\in K_{\eta}}{\sup}\mathcal{M}^{\mathrm{IS}}(\xi)\le\underset{\xi\in K_{\eta}}{\sup}\lambda(\theta)<0$.

The last part of the proof is the same as for NCE consistency (see the appendix).

 \subsubsection{MC-MLE asymptotic normality}
For convenience, for any $\xi=(\theta,\nu)$, we note $g_{\xi}(x)=\nu+\log h_{\theta}(x)$.

Let $G_n^{\mathrm{IS}}(\xi)=\nabla_{\xi}\lIS(\xi)$ and $\mathbf{H}_n^{\mathrm{IS}}(\xi)=\nabla_{\xi}^2\lIS(\xi)$. We have
$$G_n^{\mathrm{IS}}(\xi)=\frac{1}{n}\sum_{i=1}^n\nabla_{\xi}g_{\xi}(Y_i)-\frac{1}{m_n}\sum_{j=1}^{m_n}\nabla_{\xi}g_{\xi}(X_j)\frac{\exp\{g_{\xi}(X_j)\}}{h_{\psi}(X_j)},$$
\begin{equation}
\mathbf{H}_n^{\mathrm{IS}}(\xi)=\frac{1}{n}\sum_{i=1}^n\nabla_{\xi}^2g_{\xi}(Y_i)-\frac{1}{m_n}\sum_{j=1}^{m_n}\Big\{(\nabla_{\xi}^2+\nabla_{\xi}\nabla_{\xi}^T)g_{\xi}(X_j)\Big\}\frac{\exp\{g_{\xi}(X_j)\}}{h_{\psi}(X_j)}.\label{eq:hess_is}
\end{equation}

We start by proving that, almost surely,
\begin{equation}\underset{\xi\in B(\xi^{\star},\varepsilon)}{\sup}\|\mathbf{H}_n^{\mathrm{IS}}(\xi)-\mathbf{H}(\xi)\|\cvn0,\label{eq:cv_hess_is}\end{equation}
where 
$$\mathbf{H}(\xi)=\mathbb{E}_{\theta^{\star}}\bigg[\nabla_{\xi}^2g_{\xi}(Y)\bigg]-\mathbb{E}_{\psi}\bigg[\Big\{(\nabla_{\xi}^2+\nabla_{\xi}\nabla_{\xi}^T)g_{\xi}(X)\Big\}\frac{\exp\{g_{\xi}(X)\}}{h_{\psi}(X)}\bigg].$$

To prove \eqref{eq:cv_hess_is}, split the supremum norm in two and apply Lemma \ref{lem:ULLN_theorem} to both empirical averages in definition \eqref{eq:hess_is}. Both supremum norms are integrable under (H4), this is proven in the following.
$$\nabla_{\xi}^2g_{\xi}(x)=\left(\begin{array}{cc}
\nabla_{\theta}^2\log h_{\theta}(x) & 0  \\
0 & 0  \end{array} \right)\hspace{0.7cm}\nabla_{\xi}\nabla_{\xi}^Tg_{\xi}(x)=\left(\begin{array}{cc}
\nabla_{\theta}\nabla_{\theta}^T\log h_{\theta}(x) & \nabla_{\theta}\log h_{\theta}(x)  \\
\nabla_{\theta}^T\log h_{\theta}(x) & 1  \end{array} \right)$$
First supremum norm is integrable under (H4), since
$$\int_{\mathcal{X}}\underset{\xi\in B(\xi^{\star},\varepsilon)}{\text{sup}}\ \|\nabla_{\xi}^2g_{\xi}(x)\|h_{\theta^{\star}}\mu(\dx)\le\int_{\mathcal{X}}\underset{\theta\in B(\theta^{\star},\varepsilon)}{\text{sup}}\ \|\nabla_{\theta}^2\log h_{\theta}(Y)\|\underset{\theta\in B(\theta^{\star},\varepsilon)}{\text{sup}}h_{\theta}(x)\mu(\dx)<+\infty.$$

For the second one, use the following decomposition:
$$\|(\nabla_{\xi}^2+\nabla_{\xi}\nabla_{\xi}^T)g_{\xi}(x)\|_1=\|(\nabla_{\theta}^2+\nabla_{\theta}\nabla_{\theta}^T)\log h_{\theta}(x)\|_1+2\|\nabla_{\theta}\log h_{\theta}(x)\|_1+1,$$
$$\|(\nabla_{\theta}^2+\nabla_{\theta}\nabla_{\theta}^T)\log h_{\theta}(x)\|_1\le\|\nabla_{\theta}^2\log h_{\theta}(x)\|_1+\|\nabla_{\theta}\log h_{\theta}(x)\|_1^2,$$
$$\|\nabla_{\theta}\log h_{\theta}(x)\|_1\le1+\|\nabla_{\theta}\log h_{\theta}(x)\|_1^2.$$
This yields a finite upper bound under (H4), since
\begin{multline*}
    \int_{\mathcal{X}}\underset{\xi\in B(\xi^{\star},\varepsilon)}{\sup}\|(\nabla_{\xi}^2+\nabla_{\xi}\nabla_{\xi}^T)g_{\xi}(x)\|_1\exp\{g_{\xi}(x)\}\mu(\dx)\le\\
     e^{\nu^{\star}+\varepsilon}\int_{\mathcal{X}}\underset{\theta\in B(\theta^{\star},\varepsilon)}{\sup}\Big(\|\nabla_{\theta}^2\log h_{\theta}(x)\|_1+3\|\nabla_{\theta}\log h_{\theta}(x)\|_1^2+3\Big)\underset{\theta\in B(\theta^{\star},\varepsilon)}{\text{sup}}h_{\theta}(x)\mu(\dx)<+\infty.
\end{multline*}
Note also that, at the point $\xi=\xi^{\star}$, functions $\mathbf{H}$ and $-\mathbf{J}$ coincide, where
$$\mathbf{J}(\xi)=\mathbb{E}_{\theta}\left[\nabla_{\xi}\nabla_{\xi}^Tg_{\xi}(Y)\right]=\left(\begin{array}{cc}
\mathbb{E}_{\theta}\Big[\nabla_{\theta}\nabla_{\theta}^T\log h_{\theta}(Y)\Big] & \mathbb{E}_{\theta}\Big[\nabla_{\theta}\log h_{\theta}(Y)\Big]  \\
\mathbb{E}_{\theta}\Big[\nabla_{\theta}^T\log h_{\theta}(Y)\Big] & 1  \end{array} \right).$$
In particular, the matrix $\mathbf{J}(\xi^{\star})$ is definite positive, since Schur's complement is also the Fisher Information, definite positive by assumption:
$$\mathbb{E}_{\theta}\Big[\nabla_{\theta}\nabla_{\theta}^Tg_{\theta}(Y)\Big]-\mathbb{E}_{\theta}\Big[\nabla_{\theta}g_{\theta}(Y)\Big]\mathbb{E}_{\theta}\Big[\nabla_{\theta}^Tg_{\theta}(Y)\Big]=\mathbb{V}_{\theta}\Big(\nabla_{\theta}\log h_{\theta}(Y)\Big)=\mathbf{I}(\theta).$$

Now we establish the weak convergence of the gradient. We have
\begin{multline*}
    \sqrt{n}G_n^{\mathrm{IS}}(\xi^{\star})=\sqrt{n}\left(\frac{1}{n}\sum_{i=1}^n\nabla_{\xi}g_{\xi}(Y_i)-\mathbb{E}_{\theta}\big[\nabla_{\xi}g_{\xi}(Y)\big]\right)_{|_{\xi=\xi^{\star}}}\\
    -\sqrt{\frac{n}{m_n}}\sqrt{m_n}\left(\frac{1}{m_n}\sum_{j=1}^{m_n}\nabla_{\xi}g_{\xi}(X_j)\frac{f_{\theta}(X_j)}{f_{\psi}(X_j)}-\mathbb{E}_{\theta}\big[\nabla_{\xi}g_{\xi}(Y)\big]\right)_{|_{\xi=\xi^{\star}}}.
\end{multline*}
Simulations and observations are assumed independent, thus Slutsky's lemma yields
$$\sqrt{n}G_n^{\mathrm{IS}}(\xi^{\star})\cvd\mathcal{N}_{d+1}\Big(0, \boldsymbol{\Sigma}(\xi^{\star})+\tau^{-1}\boldsymbol{\Gamma}(\xi^{\star})\Big),$$
where $$\boldsymbol{\Sigma}(\xi)=\mathbb{V}_{\theta}\Big(\nabla_{\xi}g_{\xi}(Y)\Big)=\left(\begin{array}{cc}
\mathbf{I}(\theta) & 0  \\
0 & 0  \end{array} \right),$$
and $$\boldsymbol{\boldsymbol{\Gamma}}(\xi)=\mathbb{V}_{\psi}\left(\varphi_{\xi}^{\mathrm{IS}}(X)\right)+2\sum_{i=1}^{+\infty}\mathrm{Cov}\Big(\varphi_{\xi}^{\mathrm{IS}}(X_0),\varphi_{\xi}^{\mathrm{IS}}(X_i)\Big),\qquad\varphi_{\xi}^{\mathrm{IS}}=(\nabla_{\xi}g_{\xi})\frac{f_{\theta}}{f_{\psi}}.$$

Finally, Lemma \ref{lem:asymptotic_normality} applies:
$$\sqrt{n}\left(\emcmle-\xi^{\star}\right)\cvd\mathcal{N}\Big(0_{\mathbb{R}^{d+1}},\mathbf{V}_{\tau}^{\mathrm{IS}}(\xi^{\star})\Big)$$
where $\mathbf{V}_{\tau}^{\mathrm{IS}}(\xi)=\mathbf{J}(\xi)^{-1}\left\{\boldsymbol{\Sigma}(\xi)+\tau^{-1}\boldsymbol{\Gamma}(\xi)\right\}\mathbf{J}(\xi)^{-1}$.

\subsection{Proofs related to exponential families}

The following calculations are entirely classical. For the sake of completeness, we present the few tricks required for proving Propositions 1, 2 and 3.

To begin, define $b(x)=\sgn(S(x))$, the vector composed by the signs of each component of $S(x)$. Note that for any $\theta\in\Theta$, the following supremum is necessarily achieved on the boundary of the 1-ball, in the direction of the sign vector:
\begin{equation}
\underset{\|\phi-\theta\|_1\le\varepsilon}{\text{sup}}\exp\left\{\phi^TS(x)\right\}= \exp\left\{(\theta+\varepsilon b(x))^TS(x)\right\}.\label{eq:exp_1}
\end{equation}
Since $\|S(x)\|_1=b(x)^TS(x)$, we have (for the 1-norm for instance):
$$\underset{\phi\in B(\theta,\varepsilon)}{\sup}\left(\log\frac{h_{\phi}(x)}{h_{\theta^{\star}}(x)}\right)=(\theta-\theta^{\star})^TS(x)+\varepsilon\|S(x)\|\le(\|\theta-\theta^{\star}\|+\varepsilon)\|S(x)\|,$$
which proves the claim of Proposition 2, since
\[ \int_{\setX}\underset{\phi\in B(\theta,\varepsilon)}{\sup}\left(\log\frac{h_{\phi}(x)}{h_{\theta^{\star}}(x)}\right)_{+}h_{\theta^{\star}}(x)\mu(\dx)\le(\|\theta-\theta^{\star}\|+\varepsilon)\int_{\setX}\|S(x)\|h_{\theta^{\star}}(x)\mu(\dx)<+\infty.\]

For Propositions 1 and 3, use also the fact that $\|S(x)\|_1=b(x)^TS(x)$
and that $y\le e^y$ for any $y\in\R$. We have
\begin{equation}
\left\|S(x)\right\|_1^2\le\varepsilon^{-2}\exp\left\{2\varepsilon b(x)^TS(x)\right\}.\label{eq:exp_2}
\end{equation}
 Equations \eqref{eq:exp_1} and \eqref{eq:exp_2} can be combined as follows:
 \begin{multline*}
     \int_{\setX}(1+\|S(x)\|^2)\underset{\phi\in B(\theta,\varepsilon)}{\sup}h_{\phi}(x)\mu(\dx)\le\sum_{b\in\{-1,1\}^d}\int_{\setX}\exp\left\{(\theta+b\varepsilon)^TS(x)\right\}\mu(\dx) \\ +\varepsilon^{-2}\sum_{b\in\{-1,1\}^d}\int_{\setX}\exp\left\{(\theta+3b\varepsilon)^TS(x)\right\}\mu(\dx),
 \end{multline*}
 and
 \begin{multline*}
     \mathbb{E}_{\psi}\left[(1+\|S(X)\|^2)\underset{\phi\in B(\theta,\varepsilon)}{\sup}\left(\frac{h_{\phi}(X)}{h_{\psi}(X)}\right)^2\right]\le \sum_{b\in\{-1,1\}^d} \mathbb{E}_{\psi}\left[\left(\frac{\exp\left\{(\theta+b\varepsilon)^TS(X)\right\}}{h_{\psi}(X)}\right)^2\right]\\ +\varepsilon^{-2}\sum_{b\in\{-1,1\}^d} \mathbb{E}_{\psi}\left[\left(\frac{\exp\left\{(\theta+2b\varepsilon)^TS(X)\right\}}{h_{\psi}(X)}\right)^2\right].
 \end{multline*}
Choosing $\theta=\mle$ in the preceding inequalities yields Proposition 1, while choosing $\theta=\theta^{\star}$ yields Proposition 3.
\end{document}